\documentclass[a4paper,reqno,11pt]{amsart}

\usepackage{graphicx}
\usepackage{psfrag}
\usepackage{perpage}
\usepackage{url}
\usepackage{color}
\usepackage{mathrsfs}

\usepackage{amsmath, amsfonts, amssymb, amsthm, amscd}

\usepackage[utf8]{inputenc}
\usepackage[T1]{fontenc}
\usepackage{microtype}

\usepackage[a4paper,scale={0.72,0.74},marginratio={1:1},footskip=7mm,headsep=10mm]{geometry}

\usepackage{hyperref}

\makeatletter
\def\@secnumfont{\bfseries\scshape}

\def\section{\@startsection{section}{1}%
  \z@{.7\linespacing\@plus\linespacing}{.5\linespacing}%
  {\normalfont\large\bfseries\scshape\centering}}

\def\subsection{\@startsection{subsection}{2}%
  \z@{.5\linespacing\@plus.7\linespacing}{-.5em}%
  {\normalfont\bfseries\scshape}}

\def\subsubsection{\@startsection{subsubsection}{3}%
  \z@{.5\linespacing\@plus.7\linespacing}{-.5em}%
  {\normalfont\scshape}}

\def\specialsection{\@startsection{section}{1}%
  \z@{\linespacing\@plus\linespacing}{.5\linespacing}%
  {\normalfont\centering\large\bfseries\scshape}}
\makeatother

\makeatletter

\renewenvironment{proof}[1][\proofname]{\par
\pushQED{\qed}%
\normalfont \topsep4\p@\@plus4\p@\relax
\trivlist
\item[\hskip\labelsep
\bfseries
#1\@addpunct{.}]\ignorespaces
}{%
\popQED\endtrivlist\@endpefalse
}
\makeatother

\makeatletter
\newcommand \Dotfill {\leavevmode \leaders \hb@xt@ 6pt{\hss .\hss }\hfill \kern \z@}
\makeatother

\makeatletter
\def\@tocline#1#2#3#4#5#6#7{\relax
  \ifnum #1>\c@tocdepth 
  \else
    \par \addpenalty\@secpenalty\addvspace{#2}%
    \begingroup \hyphenpenalty\@M
    \@ifempty{#4}{%
      \@tempdima\csname r@tocindent\number#1\endcsname\relax
    }{%
      \@tempdima#4\relax
    }%
    \parindent\z@ \leftskip#3\relax \advance\leftskip\@tempdima\relax
    \rightskip\@pnumwidth plus4em \parfillskip-\@pnumwidth
    #5\leavevmode\hskip-\@tempdima
      \ifcase #1
       \or\or \hskip 1.65em \or \hskip 3.3em \else \hskip 4.95em \fi%
      #6\nobreak\relax
    \Dotfill
    \hbox to\@pnumwidth{\@tocpagenum{#7}}\par
    \nobreak
    \endgroup
  \fi}
\makeatother

\makeatletter
\def\l@section{\@tocline{1}{0pt}{1pc}{}{\scshape}}
\renewcommand{\tocsection}[3]{%
\indentlabel{\@ifnotempty{#2}{\ignorespaces#1 #2.\hskip 0.7em}}#3}
\def\l@subsection{\@tocline{2}{0pt}{1pc}{5pc}{}}

\def\l@subsubsection{\@tocline{3}{0pt}{1pc}{7pc}{}}

\makeatother

\numberwithin{equation}{section}

\newtheoremstyle{mytheorem}{.7\linespacing\@plus.3\linespacing}{.7\linespacing\@plus.3\linespacing}%
     {\itshape}
     {}
     {\bfseries}
     {. }
     {0.3ex}
     {\thmname{{\bfseries #1}}\thmnumber{ {\bfseries #2}}\thmnote{ (#3)}}  

\theoremstyle{mytheorem}

\newtheorem{theorem}{Theorem}[section]
\newtheorem{lemma}[theorem]{Lemma}

\newtheorem{corollary}[theorem]{Corollary}
\newtheorem{remark}[theorem]{Remark}

\newtheorem{assumption}[theorem]{Assumption}

\newtheorem{conjecture}[theorem]{Conjecture}

\newcommand{\bbB}{{\ensuremath{\mathbb B}} }

\newcommand{\bbE}{{\ensuremath{\mathbb E}} }

\newcommand{\bbN}{{\ensuremath{\mathbb N}} }

\newcommand{\bbP}{{\ensuremath{\mathbb P}} }

\newcommand{\bbR}{{\ensuremath{\mathbb R}} }

\newcommand{\bbT}{{\ensuremath{\mathbb T}} }

\newcommand{\bbV}{{\ensuremath{\mathbb V}} }

\newcommand{\bbZ}{{\ensuremath{\mathbb Z}} }

\newcommand{\cA}{{\ensuremath{\mathcal A}} }
\newcommand{\cB}{{\ensuremath{\mathcal B}} }
\newcommand{\cC}{{\ensuremath{\mathcal C}} }

\newcommand{\cN}{{\ensuremath{\mathcal N}} }

\newcommand{\cP}{{\ensuremath{\mathcal P}} }

\newcommand{\cS}{{\ensuremath{\mathcal S}} }

\newcommand{\gb}{\beta}

\newcommand{\gz}{\zeta}

\newcommand{\go}{\omega}

\renewcommand{\tilde}{\widetilde}          
\DeclareMathSymbol{\leqslant}{\mathalpha}{AMSa}{"36} 
\DeclareMathSymbol{\geqslant}{\mathalpha}{AMSa}{"3E} 
\DeclareMathSymbol{\eset}{\mathalpha}{AMSb}{"3F}     

\newcommand{\sumtwo}[2]{\sum_{\substack{#1 \\ #2}}} 
\newcommand{\sumthree}[3]{\sum_{\substack{#1 \\ #2 \\ #3}}} 

\newcommand{\R}{\mathbb{R}}

\newcommand{\Z}{\mathbb{Z}}
\newcommand{\N}{\mathbb{N}}

\newcommand\bs{\boldsymbol}

\newcommand{\PEfont}{\mathrm}
\DeclareMathOperator{\var}{\ensuremath{\PEfont Var}}

\newcommand{\p}{\ensuremath{\PEfont P}}
\newcommand{\e}{\ensuremath{\PEfont E}}
\newcommand{\E}{\e}
\renewcommand{\P}{\p}
\DeclareMathOperator{\sign}{sign}

\newcommand\bP{\ensuremath{\bs{\mathrm{P}}}}
\newcommand\bE{\ensuremath{\bs{\mathrm{E}}}}

\DeclareMathOperator{\bbvar}{\ensuremath{\mathbb{V}ar}}
\DeclareMathOperator{\bbcov}{\ensuremath{\mathbb{C}ov}}

\newenvironment{myenumerate}{%
\renewcommand{\theenumi}{\arabic{enumi}}%
\renewcommand{\labelenumi}{{\rm(\theenumi)}}%
\begin{list}{\labelenumi}
	{%
	\setlength{\itemsep}{0.4em}%
	\setlength{\topsep}{0.5em}%
	\setlength\leftmargin{2.45em}%
	\setlength\labelwidth{2.05em}%
	\setlength{\labelsep}{0.4em}%
	\usecounter{enumi}%
	}%
	}%
{\end{list}
}

\newenvironment{aenumerate}{%
\renewcommand{\theenumi}{\alph{enumi}}%
\renewcommand{\labelenumi}{{\rm(\theenumi)}}%
\begin{list}{\labelenumi}
	{%
	\setlength{\itemsep}{0.4em}%
	\setlength{\topsep}{0.5em}%
	\setlength\leftmargin{2.45em}%
	\setlength\labelwidth{2.05em}%
	\setlength{\labelsep}{0.4em}%
	\usecounter{enumi}%
	}%
	}%
{\end{list}
}

\newenvironment{ienumerate}{%
\renewcommand{\theenumi}{\roman{enumi}}%
\renewcommand{\labelenumi}{{\rm(\theenumi)}}%
\begin{list}{\labelenumi}
	{%
	\setlength{\itemsep}{0.4em}%
	\setlength{\topsep}{0.5em}%
	\setlength\leftmargin{2.45em}%
	\setlength\labelwidth{2.05em}%
	\setlength{\labelsep}{0.4em}%
	\usecounter{enumi}%
	}%
	}%
{\end{list}
}

\renewenvironment{enumerate}{
\begin{myenumerate}}%
{\end{myenumerate}}

\newenvironment{myitemize}{%
\begin{list}{$\bullet$}%
 	{%
	\setlength{\itemsep}{0.4em}%
	\setlength{\topsep}{0.5em}%
	\setlength\leftmargin{2.45em}%
	\setlength\labelwidth{2.05em}%
	\setlength{\labelsep}{0.4em}%
	}%
	}%
{\end{list}}

\renewenvironment{itemize}{
\begin{myitemize}}%
{\end{myitemize}}

\MakePerPage[2]{footnote} 

\renewcommand{\epsilon}{\varepsilon}
\renewcommand{\theta}{\vartheta}
\renewcommand{\rho}{\varrho}

\newcommand{\dd}{\text{\rm d}}             

\newcommand\hbeta{{\hat{\beta}}}
\newcommand\hh{{\hat{h}}}

\newcommand\hlambda{{\hat{\lambda}}}

\newcommand\sfC{\mathsf C}
\newcommand\sfc{\mathsf c}

\newcommand\Inf{\mathrm{Inf}}
\newcommand\fin{\mathrm{fin}}

\newcommand\re{\mathrm{ref}}
\newcommand\eff{\mathrm{eff}}

\newcommand\bsigma{\boldsymbol\sigma}
\newcommand\bpsi{\boldsymbol\psi}
\newcommand\bphi{\boldsymbol\phi}
\newcommand\bPsi{\boldsymbol\Psi}
\newcommand\btau{\boldsymbol\tau}
\newcommand\bZ{\boldsymbol Z}
\newcommand\bA{\boldsymbol A}
\newcommand\bF{\boldsymbol F}
\newcommand\bmu{\boldsymbol \mu}
\newcommand\bnu{\boldsymbol \nu}
\newcommand\bh{\boldsymbol h}

\usepackage{dsfont} 
\newcommand{\ind}{\mathds{1}}

\newcommand\bbOmegadelta{\Omega_\delta}
\newcommand\bbOmega{\Omega}

\begin{document}
\title[Polynomial chaos and scaling limits of disordered systems]{Polynomial chaos and\\
scaling limits of disordered systems}

\begin{abstract}

Inspired by recent work of Alberts, Khanin and Quastel~\cite{AKQ14a},
we formulate general conditions ensuring that a sequence of
multi-linear polynomials of independent random variables
(called \emph{polynomial chaos} expansions) converges to a limiting random variable, given by
a \emph{Wiener chaos} expansion over the $d$-dimensional white noise.
A key ingredient in our approach is a \emph{Lindeberg principle} for polynomial chaos expansions,
which extends earlier work of Mossel, O'Donnell and Oleszkiewicz~\cite{MOO10}.
These results provide a unified framework to study the
\emph{continuum and weak disorder scaling limits}
of statistical mechanics systems that are \emph{disorder relevant}, including
the disordered pinning model, the (long-range) directed polymer
model in dimension $1+1$, and the two-dimensional
random field Ising model. This gives a new perspective in the study of disorder relevance,
and leads to interesting new continuum models that warrant further studies.

\end{abstract}

\author[F.Caravenna]{Francesco Caravenna}
\address{Dipartimento di Matematica e Applicazioni\\
 Universit\`a degli Studi di Milano-Bicocca\\
 via Cozzi 55, 20125 Milano, Italy}
\email{francesco.caravenna@unimib.it}

\author[R.Sun]{Rongfeng Sun}
\address{Department of Mathematics\\
National University of Singapore\\
10 Lower Kent Ridge Road, 119076 Singapore
}
\email{matsr@nus.edu.sg}

\author[N.Zygouras]{Nikos Zygouras}
\address{Department of Statistics\\
University of Warwick\\
Coventry CV4 7AL, UK}
\email{N.Zygouras@warwick.ac.uk}

\date{\today}

\keywords{Continuum Limit, Finite Size Scaling, Lindeberg Principle,
Polynomial Chaos, Wiener Chaos, 
Disordered Pinning Model, Directed Polymer Model,
Random Field Ising Model}
\subjclass[2010]{Primary: 82B44; Secondary: 82D60, 60K35}

\maketitle

\tableofcontents

\section{Introduction}
\label{sec:intro}

In this paper, we consider statistical mechanics models
defined on a lattice,
in which disorder acts as an external ``random field''.
We focus on models that are \emph{disorder relevant}, in the
sense that arbitrarily weak disorder changes the
qualitative properties of the model.
We will show that, when the homogeneous model
(without disorder) has a non-trivial continuum limit,
disorder relevance manifests itself via the convergence
of the disordered model to a
\emph{disordered continuum limit}, if
the disorder strength and lattice mesh are suitably rescaled.

Our approach is inspired by recent work of Alberts, Khanin and
Quastel~\cite{AKQ14a} on the directed polymer model in
dimension $1+1$. Here we follow a different path,
establishing a general convergence result for polynomial
chaos expansions based on a Lindeberg principle.
This extends earlier work
of Mossel, O'Donnell and Oleszkiewicz~\cite{MOO10}
to optimal (second) moment assumptions,
and is of independent interest.

\smallskip

In this section, we present somewhat informally the main ideas
of our approach in a unified framework, emphasizing the
natural heuristic considerations. The precise formulation of
our results is given in Sections~\ref{ConvergencePoly} and~\ref{sec:scaling},
which can be read independently (both of each other and of the present one).
The proofs are contained in Sections~\ref{S:lindeberg} to~\ref{sec:Isingproof},
while some technical parts have been deferred to the Appendixes.
Throughout the paper,
we use the conventions $\N := \{1,2,3,\ldots\}$ and $\N_0 := \N \cup \{0\}$,
and we denote by $Leb$ the Lebesgue measure on $\R^d$.

\subsection{Continuum limits of disordered systems}

Consider an open set $\Omega \subseteq \R^d$
and define the lattice $\bbOmegadelta := (\delta \Z)^d \cap \Omega$, for $\delta > 0$.
Suppose that a \emph{reference probability measure}
$\P_{\bbOmegadelta}^\re$ is given on $\R^{\bbOmegadelta}$,
which describes a real-valued field $\sigma = (\sigma_x)_{x\in\bbOmegadelta}$.
We focus on the case when each $\sigma_x$ takes two possible values (typically $\sigma_x \in \{0,1\}$
or $\sigma_x \in \{-1,1\}$).

Let $\omega := (\omega_x)_{x\in\bbOmegadelta}$ be i.i.d.\ random variables (also independent of $\sigma$)
with zero mean, unit variance, and locally finite exponential moments, which represent the \emph{disorder}.
Probability and expectation for $\omega$ will be denoted respectively by $\bbP$ and $\bbE$.

Given $\lambda>0$, $h\in\R$ and a $\bbP$-typical realization
of the disorder $\omega$, we define the \emph{disordered model}
as the following probability measure $\P^\omega_{{\bbOmegadelta}; \lambda,h}$ for the field
$\sigma = (\sigma_x)_{x\in\bbOmegadelta}$:
\begin{equation} \label{eq:model}
	\P^\omega_{{\bbOmegadelta}; \lambda,h} (\dd\sigma)
	:= \frac{e^{\sum_{x\in\bbOmegadelta} (\lambda \omega_x +h) \sigma_x}}
	{Z^\omega_{{\bbOmegadelta}; \lambda,h}} \P_{\bbOmegadelta}^\re (\dd\sigma)\,,
\end{equation}
where the normalizing constant, called the {\em partition function}, is defined by
\begin{equation} \label{eq:partfun}
	Z^\omega_{{\bbOmegadelta}; \lambda,h} := \E_{\bbOmegadelta}^\re
	\big[ e^{\sum_{x\in\bbOmegadelta} (\lambda \omega_x+h) \sigma_x} \big] \,.
\end{equation}
The {\em quenched free energy} $F(\lambda, h)$ of the model
is defined as the rate of exponential growth of $Z^\omega_{{\bbOmegadelta}; \lambda,h}$
as $\Omega \uparrow \R^d$ for fixed $\delta$ (or
equivalently,\footnote{We assume the natural consistency
condition that $\P^\re_{(c\bbOmega)_{c\delta}}$ coincides with $\P^\re_{\bbOmegadelta}$,
for any $c>0$.}
as $\delta \downarrow 0$ for fixed $\Omega$):
\begin{equation}\label{eq:dfe}
	F(\lambda, h) :=
	\limsup_{\Omega \uparrow \R^d}
	\frac{1}{|\bbOmegadelta|} \bbE \big[ \log Z^\omega_{{\bbOmegadelta}; \lambda,h} \big] =
	\limsup_{\delta\downarrow 0}
	\frac{1}{|\bbOmegadelta|} \bbE \big[ \log Z^\omega_{{\bbOmegadelta}; \lambda,h} \big] \,.
\end{equation}
Discontinuities in the derivatives of the free energy correspond to phase transitions.
A fundamental question
is, does arbitrary disorder
(i.e., $\lambda>0$) radically change the behavior of the homogeneous model
(i.e., $\lambda=0$), such as qualitative properties of the law of the field
$\sigma$ and/or the smoothness of the free energy in $h$? When the answer is affirmative,
the model is called {\em disorder relevant}. In such cases, we will show that
the disordered model typically admits a non-trivial scaling limit as $\delta \downarrow 0$,
provided $\lambda, h \to 0$ at appropriate rates.

\smallskip

Informally speaking, our key assumption is that the discrete field
$\sigma = (\sigma_x)_{x\in\bbOmegadelta}$, under the reference law $\P_{\bbOmegadelta}^\re$ and
after a suitable rescaling, converges as $\delta \downarrow 0$
to a ``continuum field'' $\bsigma = (\bsigma_x)_{x\in\Omega}$,
possibly distribution-valued, with law $\bP_{\Omega}^\re$.
(Our precise assumptions will be about the convergence of correlation functions,
see \eqref{eq:polyconv}.)
Although the approach we follow is very general, we describe three specific models,
to be discussed extensively in the sequel.
\begin{enumerate}
\item The \emph{disordered pinning model $(d=1)$.} \
Let $\tau = (\tau_k)_{k\ge 0}$ be a \emph{renewal process} on $\N$
with $\p(\tau_1 = n) = n^{-(1+\alpha) + o(1)}$, with $\alpha \in (\frac{1}{2}, 1)$.
Take $\Omega = (0,1)$,
$\delta = \frac{1}{N}$ for $N \in \N$
and define $\P_{\bbOmegadelta}^\re$ as the law of
$(\sigma_x := \ind_{\delta\tau}(x))_{x\in \bbOmegadelta}$,
where $\delta \tau = \{\frac{1}{N} \tau_n\}_{n\ge 0}$ is viewed as a random subset of $\Omega$.
The continuum field $\bP_{\Omega}^\re$ is $(\sigma_x = \ind_{\btau}(x))_{x \in (0,1)}$
where $\btau$ denotes the $\alpha$-stable regenerative set
(the zero level set of a Bessel($2(1-\alpha)$) process).

\item The (long-range) \emph{directed polymer model}. \
Let $(S_n)_{n\ge 0}$ be a random walk on $\Z$ with i.i.d. increments,
in the domain of attraction of an $\alpha$-stable law, with $1 < \alpha \le 2$.
Take $\Omega = (0,1) \times \R$, $\delta = \frac{1}{N}$ for $N \in \N$
and, abusing notation,
set $\bbOmegadelta := \big((\delta \Z) \times (\delta^{1/\alpha}\Z)\big) \cap \Omega$.
The ``effective dimension'' for this model is therefore $d_\eff := 1+\frac{1}{\alpha}$.
Define $\P_{\bbOmegadelta}^\re$ as the law of the field
$(\sigma_x := \ind_{A_\delta}(x))_{x\in \bbOmegadelta}$,
where $A_\delta := \{(\frac{n}{N}, \frac{S_n}{N^{1/\alpha}})\}_{n \ge 0}$ is viewed
as a random subset of $\Omega$.
The continuum field $\bP_{\Omega}^\re$ is
$(\bsigma_x = \ind_{\bA}(x))_{x \in (0,1) \times \R}$
where $\bA=\{(t,X_t)\}_{t\ge 0}$ and $(X_t)_{t\ge 0}$ is an $\alpha$-stable L\'evy
process (Brownian motion for $\alpha = 2$).

\item The \emph{random field Ising model $(d=2)$.} \
Take any bounded and connected set $\Omega \subseteq \R^2$ with smooth boundary
and define $\P_{\bbOmegadelta}^\re$ to be the \emph{critical} Ising model on $\bbOmegadelta$
with inverse temperature $\beta = \beta_c =
\frac{1}{2}\log(1+\sqrt{2})$ and $+$ boundary condition.
The (distribution valued) continuum field $\bP_{\Omega}^\re$ has been recently
constructed in~\cite{CGN12,CGN13}, using breakthrough results on the scaling limit
of correlation functions of the critical two-dimensional Ising model, determined in~\cite{CHI12}.
\end{enumerate}
The restrictions on the dimensions and parameters of these
models are linked precisely to the disorder relevance issue, as
it will be explained later.

\smallskip

Since the reference law $\P_{\bbOmegadelta}^\re$ has a weak limit $\bP_{\Omega}^\re$
as $\delta \downarrow 0$,
a natural question emerges:
can one obtain a limit also
for the disordered model $\P^\omega_{{\bbOmegadelta}; \lambda,h}$,
under an appropriate scaling of the coupling constants $\lambda, h$?
(We mean, of course,
a \emph{non-trivial} limit,
which keeps track of $\lambda, h$; otherwise, it suffices
to let $\lambda, h \to 0$ very fast to recover the ``free case''
$\bP_{\Omega}^\re$.)

A natural strategy is to look at the exponential weight in \eqref{eq:model}.
As $\delta \downarrow 0$ the
discrete disorder $\omega = (\omega_x)_{x\in\bbOmegadelta}$
approximates the white noise $W(\dd x)$,
which is a sort of random signed measure on $\Omega$ (see Subsection~\ref{sec:wnn} for more details).
Then one might hope to define the candidate
continuum disordered model $\bP^W_{{\bbOmegadelta}; \hlambda,\hh}$ by
\begin{equation} \label{eq:modelcont}
	\frac{\dd \bP^W_{{\Omega}; \hlambda,\hh}}{\dd \bP_{\Omega}^\re} (\bsigma)
	:= \frac{e^{\int_\Omega \bsigma_x \left(\hlambda \, W(\dd x)+\hh \, \dd x\right)}}
	{\bZ^W_{{\Omega}; \hlambda, \hh}} \,,
\end{equation}
in analogy with \eqref{eq:model}, with $\bZ^W_{{\Omega}; \hlambda, \hh}$ defined
accordingly, like in \eqref{eq:partfun}.

Unfortunately, formula \eqref{eq:modelcont} typically makes no sense
for $\hlambda \ne 0$,
because the configurations of the continuum field
$\bsigma = (\bsigma_x)_{x\in\Omega}$ under $\bP_{\Omega}^\re$
are too rough or ``thin'' for the integral over $W(\dd x)$ to be meaningful
(cf.\ the three motivating models listed above).
We stress here that the difficulty is substantial and not just technical:
for pinning and directed polymer models,
one can show \cite{AKQ14b,CSZ14} that
the scaling limit $\bP^W_{{\Omega}; \hlambda,\hh}$ of
$\P^\omega_{{\bbOmegadelta}; \lambda,h}$
exists, but for $\hlambda \ne 0$ \emph{it is not absolutely continuous with respect to
$\bP_{\Omega}^\re$}.
In particular, it is hopeless to define the continuum disordered
model through a Radon-Nikodym density, like
in \eqref{eq:modelcont}.

\subsection{General strategy and results}

In this paper
\emph{we focus on the disordered partition function}
$Z_{{\bbOmegadelta}; \lambda,h}^\omega$. We show that,
when $\delta \downarrow 0$ and $\lambda, h$
are scaled appropriately, the partition function admits a non-trivial limit in distribution,
which is explicit and \emph{universal}
(i.e., it does not depend on the fine details of the model).

Switching from the random probability law
$\P_{{\bbOmegadelta}; \lambda,h}^\omega$ to
the random number $Z_{{\bbOmegadelta}; \lambda,h}^\omega$
is of course a simplification, whose relevance may
not be evident. It turns out that the partition function
\emph{contains the essential information on the model}. In fact,
the scaling limit of
$Z_{{\bbOmegadelta}; \lambda,h}^\omega$, for sufficiently many domains $\Omega$
and ``boundary conditions'', allows to reconstruct
\emph{the full scaling limit of $\P_{{\bbOmegadelta}; \lambda,h}^\omega$}.
This task has been achieved in \cite{AKQ14b} for the directed polymer model
based on simple random walk,
and in \cite{CSZ14} for the disordered pinning model.
We discuss the case of the
long-range directed polymer model in Remark \ref{rem:remdpre} below.

The scaling limit of
$Z_{{\bbOmegadelta}; \lambda,h}^\omega$ can also
describe the \emph{universal behavior} of the free energy $F(\lambda,h)$
as $\lambda, h \to 0$, cf.\ Subsection~\ref{sec:discussion}.
This explains the key role of the partition function
and is a strong motivation for focusing on it in the first place.

\smallskip

We now describe our approach.
The idea is to consider the so-called ``high-temperature
expansion'' ($|\lambda|, |h| \ll 1$) of
the partition function $Z^\omega_{\bbOmegadelta; \lambda, h}$.
When the field takes two values
(say $\sigma_x \in \{0,1\}$, for simplicity),
we can factorize and ``linearize'' the exponential in \eqref{eq:partfun}:
\begin{equation}\label{eq:Zstart}
	Z^\omega_{\bbOmegadelta; \lambda, h} =
	\E_{\bbOmegadelta}^\re \Bigg[ \prod_{x \in \bbOmegadelta} \big(1 + \epsilon_x \sigma_x
	\big) \Bigg] \,, \qquad \text{where} \qquad
	\epsilon_x := e^{(\lambda \omega_x + h)} - 1 \,.
\end{equation}
Let us introduce the \emph{$k$-point correlation function}
$\psi_{\bbOmegadelta}^{(k)}(x_1, \ldots, x_k)$ of the
field under the reference law, defined for
$k\in\N$ and distinct $x_1, \ldots, x_k \in \Omega$ by
\begin{equation} \label{eq:corrfun}
	\psi_{\bbOmegadelta}^{(k)}(x_1, \ldots, x_k) :=
	\E_{\bbOmegadelta}^\re \big[\sigma_{x_1} \, \sigma_{x_2}
	\cdots \sigma_{x_k} \big] \,,
\end{equation}
where we set $\sigma_x := \sigma_{x_\delta}$ with $x_\delta$ being
the point in $\bbOmegadelta$ closest to $x \in \Omega$.
(We define the correlation function on all points of $\Omega$
for later convenience, and set it to be zero
whenever $(x_i)_\delta = (x_j)_\delta$ for some $i\ne j$.)
A binomial expansion of the product in \eqref{eq:Zstart} then yields
\begin{equation} \label{eq:Zstart2}
	Z^\omega_{\bbOmegadelta; \lambda, h} =
	1 + \sum_{k=1}^{|\bbOmegadelta|} \;
	\frac{1}{k!} \,
	\sum_{(x_1, x_2, \ldots, x_k) \in (\bbOmegadelta)^k}
	\psi^{(k)}_{\bbOmegadelta}(x_1, \ldots, x_k)
	\, \prod_{i=1}^k \epsilon_{x_i} \,,
\end{equation}
where the $k!$ accounts for the fact that we sum over
\emph{ordered} $k$-uples $(x_1, \ldots, x_k)$.
We have rewritten the partition function as a
multi-linear polynomial of the independent random variables $(\epsilon_x)_{x\in\bbOmegadelta}$
(what is called a \emph{polynomial chaos} expansion),
with coefficients
given by the $k$-point correlation function of the reference field.
Note that, by Taylor expansion,
\begin{equation} \label{eq:Taylor}
	\bbE[\epsilon_x] \simeq h + \tfrac{1}{2}\lambda^2 =: h' \,, \qquad
	\bbvar[\epsilon_x] \simeq \lambda^2 \,.
\end{equation}

\smallskip
The crucial fact is that, for $|\lambda|, |h| \ll 1$, the distribution of
a polynomial chaos expansion, like the right hand side of \eqref{eq:Zstart2}, is insensitive toward
the marginal distribution of the random variables $(\epsilon_x)_{x\in\bbOmegadelta}$,
as long as mean and variance are kept fixed.
\emph{A precise formulation of this loosely stated invariance principle
is given in Section~\ref{ConvergencePoly},
cf.\ Theorems~\ref{th:lindeberg} and~\ref{C:lindeberg2},
in the form of a Lindeberg principle}.
Denoting by $(\tilde\omega_x)_{x\in\bbOmegadelta}$ a family of i.i.d.
standard Gaussians, by \eqref{eq:Taylor} we can then approximate
\begin{equation} \label{eq:Zstart3}
	Z^\omega_{\bbOmegadelta; \lambda, h} \simeq
	1 + \sum_{k=1}^{|\bbOmegadelta|} \;
	\frac{1}{k!} \,
	\sum_{(x_1, x_2, \ldots, x_k) \in (\bbOmegadelta)^k}
	\psi^{(k)}_{\bbOmegadelta}(x_1, \ldots, x_k)
	\, \prod_{i=1}^k \big(\lambda \tilde \omega_{x_i} + h' \big) \,.
\end{equation}
Let us now introduce white noise $W(\cdot)$ on $\R^d$
(see Subsection~\ref{sec:wnn}):
setting
$\Delta := (-\frac{\delta}{2}, \frac{\delta}{2})^d$,
we replace each $\tilde\omega_x$ by
$\delta^{-d/2}W(x + \Delta)$. Since $h' = h' \delta^{-d} Leb(x+\Delta)$,
the inner sum in \eqref{eq:Zstart3} coincides, up to
boundary terms, with the following (deterministic + stochastic) integral:
\begin{equation} \label{eq:Zstart4}
	\idotsint_{\Omega^k}
	\psi^{(k)}_{\bbOmegadelta}(x_1, \ldots, x_k)
	\, \prod_{i=1}^k \big( \lambda \, \delta^{-d/2} \, W(\dd x_i)
        \,+\, h' \, \delta^{-d} \, \dd x_i  \big) \,.
\end{equation}
(We recall that $\psi^{(k)}_{\bbOmegadelta}(x_1, \ldots, x_k)$ is a piecewise constant function.)

\smallskip

We can now state our crucial assumption:
we suppose that, for every $k\in\N$, there exist a symmetric
function $\bpsi_\Omega^{(k)}: (\R^d)^k \to \R$
and an exponent $\gamma \in [0,\infty)$ such that
\begin{equation} \label{eq:polyconv}
	(\delta^{-\gamma})^k \, \psi_{\bbOmegadelta}^{(k)} (x_1, \ldots, x_k)
	\,\xrightarrow[\,\delta \downarrow 0\,]{}\,
	\bpsi_{\Omega}^{(k)} (x_1, \ldots, x_k) \qquad
       \text{in } L^2(\Omega^k) \,.
\end{equation}
By \eqref{eq:Zstart4}, if we fix $\hlambda \ge 0, \hh \in \R$
and rescale the coupling constants as follows:
\begin{equation} \label{eq:scalingpar}
	\lambda = \hlambda \, \delta^{d/2-\gamma} \,, \qquad \ \
	h' = \hh \, \delta^{d-\gamma} \quad \
	\big(\text{where} \ h' := h + \tfrac{\lambda^2}{2} \big) \,,
\end{equation}
equations \eqref{eq:Zstart3}-\eqref{eq:Zstart4} suggest that
$Z_{\bbOmegadelta; \lambda,h}^\omega$ converges
in distribution as $\delta\downarrow 0$ to
a random variable
which admits a \emph{Wiener chaos expansion} with respect to the white noise $W(\cdot)$:
\begin{equation} \label{eq:Zcont}
	Z_{\bbOmegadelta; \lambda,h}^\omega \;\xrightarrow[\,\delta \downarrow 0\,]{d}\;
	\bZ^W_{\Omega; \hlambda, \hh} :=
	1 + \sum_{k=1}^{\infty} \;
	\frac{1}{k!} \,
	\idotsint_{\Omega^k}
	\bpsi^{(k)}_{\Omega}(x_1, \ldots, x_k)
	\, \prod_{i=1}^k \big(\hlambda \, W(\dd x_i) +  \hh \, \dd x_i\big) \,.
\end{equation}
This is precisely what happens, as \emph{it
follows from our main convergence results described in Section~\ref{ConvergencePoly},
cf.\ Theorems~\ref{th:general}
and~\ref{th:general2}}.
It is natural to call the random variable $\bZ^W_{\Omega; \hlambda, \hh}$
in \eqref{eq:Zcont} the \emph{continuum partition function},
because it is the scaling limit of $Z^\omega_{\bbOmegadelta; \lambda, h}$.

\begin{remark}\rm
The $L^2$ convergence in \eqref{eq:polyconv} typically
imposes $\gamma < \frac{d}{2}$, cf.\ \eqref{eq:conddgamma} below, which means that the
disorder coupling constants $\lambda, h$ vanish as $\delta \downarrow 0$,
by \eqref{eq:scalingpar}.
The fact that the continuum partition function $\bZ^W_{\Omega; \hlambda, \hh}$
in \eqref{eq:Zcont} is nevertheless a \emph{random object}
(for $\hlambda > 0$) is a manifestation of \emph{disorder relevance}.
We elaborate more on this issue in Subsection~\ref{sec:discussion}.
\end{remark}

Let us finally give a quick look at the three motivating models.
\emph{The complete results are described in Section~\ref{sec:scaling},
cf.\ Theorems~\ref{pinning scaling}, \ref{DP scaling}
and~\ref{T:RFIM}.}
Note that the scaling exponents in \eqref{eq:scalingpar}
are determined by the dimension $d$ and by the exponent $\gamma$
appearing in \eqref{eq:polyconv}.

\begin{enumerate}
\item For the \emph{disordered pinning model} ($d=1$), one has
$\gamma = 1-\alpha$ by renewal theory~\cite{D97}. Relation \eqref{eq:scalingpar}
(for $\delta = \frac{1}{N}$) yields
\begin{equation*}
	\lambda = \frac{\hlambda}{N^{\alpha-\frac{1}{2}}} \,, \qquad
	h' = \frac{\hh}{N^\alpha} \,.
\end{equation*}
Notice that $h' = (const.) \lambda^{\frac{2\alpha}{2\alpha-1}}$
is precisely the scaling of the critical curve~\cite{G10}.

\item For the (long-range) \emph{directed polymer model} ($d_\eff=1+\frac{1}{\alpha}$), one has
$\gamma = \frac{1}{\alpha}$ by Gnedenko's local limit theorem \cite{BGT87}.
Therefore \eqref{eq:scalingpar}
(for $\delta = \frac{1}{N}$ and $d$ relpaced by $d_\eff$) gives
\begin{equation*}
	\lambda = \frac{\hlambda}{N^{\frac{\alpha - 1}{2\alpha}}} \,, \qquad
	h' = \frac{\hh}{N} \,.
\end{equation*}
(The parameter $h'$ is actually irrelevant for this model, and one usually sets $h' \equiv 0$,
i.e., $h = -\frac{1}{2} \lambda^2$.)
In the case $\alpha = 2$, when the underlying random walk has zero mean and finite variance,
one recovers the scaling $\lambda \approx N^{-1/4}$ determined in \cite{AKQ14a}.

\item For the \emph{random field Ising model} ($d=2$) one has
$\gamma = \frac{1}{8}$ \cite{CHI12}, hence by \eqref{eq:scalingpar}
\begin{equation} \label{eq:isingscale}
	\lambda = \hlambda \, \delta^{\frac{7}{8}} \,, \qquad
	h = \hh \, \delta^{\frac{15}{8}} \,.
\end{equation}
(Note that $h$ instead of $h'$ appears in this relation;
moreover one should look at the normalized partition function
$\exp(-\frac{1}{2} \lambda^2 |\bbOmegadelta|)\, Z^\omega_{\bbOmegadelta; \lambda, h}$.
This is because $\sigma_x \in \{-1,+1\}$ instead of $\{0,1\}$, hence the starting relation
\eqref{eq:Zstart} requires a correction.)
\end{enumerate}

\subsection{Discussion and perspectives}
\label{sec:discussion}

We now collect some comments and observations and point out some further
directions of research.

\medskip
\noindent\textbf{1. (Disorder relevance).}
The main motivation of our approach is
to understand the issue of \emph{disorder relevance}, i.e.
whether the addition of a small amount of disorder
modifies the nature of the phase transition of the underlying homogeneous model.
Remarkably, the key condition \eqref{eq:polyconv},
which determines the class of models to which our approach
applies, is consistent with the Harris criterion for disorder relevance,
as we now discuss.

First we note that when $\gamma > 0$, which is the most interesting case, condition
\eqref{eq:polyconv} indicates that
the reference law has polynomially decaying correlations,
which is the signature that \emph{we are at
the critical point of a continuous phase transition}. In our context,
this means that the order parameter
$m_h := \lim_{\delta\downarrow 0} |\bbOmegadelta|^{-1} \E_{\bbOmegadelta; 0, h}
[ \sum_{x \in \bbOmegadelta} \sigma_x]$
in the homogeneous model ($\lambda = 0$)
vanishes continuously, but non-analytically, as
$h \to 0$, cf. \eqref{eq:model} and \eqref{eq:partfun}. 

When \eqref{eq:polyconv} holds \emph{pointwise},
with $\gamma > 0$,
the limiting correlation function typically
diverges polynomially on diagonals, with the same exponent $\gamma$:
\begin{equation} \label{eq:polyvanish}
	\bpsi_{\Omega}^{(k)} (x_1, \ldots, x_k) \approx \| x_i - x_j \|^{-\gamma}
	\qquad \text{as} \quad x_i \to x_j \,.
\end{equation}To have finite $L^2$ norm
(which is necessary for $L^2$ convergence in \eqref{eq:polyconv}),
such a local divergence must be locally square-integrable
in $\R^d$: this means that
$(d-1) - 2\gamma > -1$, i.e.
\begin{equation} \label{eq:conddgamma}
	\gamma < \frac{d}{2} \,.
\end{equation}
(Note that $(\frac{d}{2}-\gamma)$ is precisely the scaling exponent
of the coupling constant $\lambda$ in \eqref{eq:scalingpar}.)

Relation \eqref{eq:conddgamma} matches with the
\emph{Harris criterion} for disorder relevance \cite{H74}.
This was originally introduced in the context
of the Ising model with \emph{bond disorder}, but it can be
naturally rephrased for general disordered system
(cf. \cite{G10}, \cite{CCFS86}):
denoting by $\nu$ the \emph{correlation length exponent} of the homogeneous system
($\lambda = 0$), it asserts that
a $d$-dimensional system is disorder-relevant when $\nu < \frac{2}{d}$
and irrelevant when $\nu > \frac{2}{d}$
(the remaining case $\nu = \frac{2}{d}$ being dubbed marginal).
The exponent $\nu$ is usually defined in terms of field correlations:
\begin{equation} \label{eq:physcorr}
	\big| \E_{\bbOmegadelta; 0,h}(\sigma_x \sigma_y)
	- \E_{\bbOmegadelta; 0,h}(\sigma_x) \E_{\bbOmegadelta; 0,h}(\sigma_y) \big|
	\underset{\delta\downarrow 0}{\approx}
	e^{- \frac{|x-y|}{\delta\xi(h)}} \,, \qquad \text{with} \qquad
	\xi(h) \underset{h \downarrow 0}{\approx} h^{-\nu} \,.
\end{equation}
Since such an exponent can be difficult to compute, it is typical to
consider alternative notions of correlation length $\xi(h)$,
linked to finite size scaling \cite{CCFS86,CCFS89}.
In our context, it is natural to define
\begin{equation} \label{eq:ourcorr}
	\xi(h)^{-1} := \max\{ \delta > 0: \ \ Z_{\bbOmegadelta; 0,h} > A \} \,,
\end{equation}
where $A > 0$ is a fixed large constant, whose precise value is immaterial.
Recalling \eqref{eq:scalingpar}, we can rewrite \eqref{eq:Zcont} for $\lambda = 0$ as
\begin{equation*}
	\lim_{h\downarrow 0} Z_{\Omega_{\delta}; 0, h} = \bZ_{\Omega; 0, \hh} \,,
	\qquad \text{with} \qquad
	\delta = \delta_{h,\hh} := (h/\hh)^{1/(d-\gamma)} \,.
\end{equation*}
If $\hh \mapsto \bZ_{\Omega; 0, \hh}$ is increasing
(e.g., when $\bpsi^{(k)}_{\Omega}(x_1, \ldots, x_k) \ge 0$,
by \eqref{eq:Zcont})
denoting by $\hh_A$ the unique solution to $\bZ_{\Omega; 0, \hh} = A$,
under some natural regularity assumptions it follows that
\begin{equation*}
	\xi(h) \underset{h\downarrow 0}{\sim} \big(\delta_{h,\hh_A}\big)^{-1} =
	\bigg(\frac{\hh_A}{h} \bigg)^{1/(d-\gamma)} \approx h^{-\nu} \,, \qquad
	\text{with} \quad \nu = \frac{1}{d-\gamma} \,.
\end{equation*}
This shows that \emph{Harris' condition $\nu < \frac{2}{d}$ coincides with
the key condition $\gamma < \frac{d}{2}$ of our approach}, cf. \eqref{eq:conddgamma},
ensuring the square-integrability of the
limiting correlations.

\smallskip

The correlation length \eqref{eq:ourcorr} is expected to be equivalent
to the classical one \eqref{eq:physcorr}, in the sense that it should have
the same critical exponent (cf.\ \cite{G07} for disordered pinning models)
when the phase transition is continuous, that is when
$\gamma > 0$ in \eqref{eq:polyvanish}.
When $\gamma = 0$, which
is the signature of a discontinuous (first-order) phase transition,
our approach still applies and gives the scaling limit of
the disorder partition function,
but there is no direct link with disorder relevance
(cf.\ \eqref{eq:toogen}
below and the following discussion).

\smallskip

In summary, our approach suggests an \emph{alternative
view on disorder relevance}, in which the randomness survives in the continuum
limit with vanishing coupling constants.
In fact, relation \eqref{eq:Zcont} can be seen as
a rigorous \emph{finite size scaling} relation \cite{Car88} for disordered systems
(the special case of non-disordered pinning models is treated in \cite{Soh09}).

\begin{remark}\rm
We can now explain the parameter restrictions in the motivating models:
condition \eqref{eq:conddgamma} is fulfilled by the \emph{disordered pinning model}
($d=1$, $\gamma = 1-\alpha$) when $\alpha > \frac{1}{2}$,
by the (long-range) \emph{directed polymer model}
($d_\eff = 1 + \frac{1}{\alpha}$, $\gamma = \frac{1}{\alpha}$) when $\alpha > 1$,
and by the critical \emph{random field Ising model} ($d=2$, $\gamma = \frac{1}{8}$).
\end{remark}
\medskip

\noindent\textbf{2. (Universality).}
The convergence in distribution
of the discrete partition function $Z^\omega_{\bbOmegadelta; \lambda, h}$
toward its continuum counterpart
$\bZ^W_{\Omega; \hlambda, \hh}$, cf. \eqref{eq:Zcont},
is an instance of \emph{universality}. In fact:
\begin{itemize}
\item the details of the disorder distribution are irrelevant:
any family $(\omega_x)_{x\in\bbOmegadelta}$ of i.i.d.
random variables with zero mean, finite variance and locally
finite exponential moments scales in the limit to the same continuum
object, namely white noise $W(\cdot)$;

\item also the fine details of the reference law $\P_{\bbOmegadelta}^\re$
disappear in the limit: any family
$\psi^{(k)}_{\bbOmegadelta}$ of discrete $k$-point correlation functions converging
to the same limit \eqref{eq:polyconv} yields the same continuum
partition function $\bZ^W_{\Omega; \hlambda, \hh}$ in \eqref{eq:Zcont}.
\end{itemize}

At a deeper level, the continuum partition function
sheds light on the discrete free energy $F(\lambda, h)$, cf.\ \eqref{eq:dfe},
in the weak disorder regime $\lambda, h \to 0$.
Defining the \emph{continuum free energy}
\begin{equation} \label{eq:contfe}
	\bF(\hlambda, \hh) := \limsup_{\Omega \uparrow \R^d}
	\frac{1}{Leb(\Omega)} \bbE \big[ \log \bZ^W_{\Omega; \hlambda, \hh} \big] \,,
\end{equation}
and setting
$\lambda_\delta := \delta^{d/2-\gamma} \hlambda$ and
$h_\delta := \delta^{d-\gamma} \hh - \frac{1}{2}(\lambda_\delta)^2$,
cf.\ \eqref{eq:scalingpar}, one is led to the following
\begin{equation}\label{eq:conjfe}
	\text{\emph{Conjecture}}: \qquad
       \lim_{\delta \downarrow 0} \, \frac{F(\lambda_\delta, h_\delta)}
	{\delta^d} = \bF(\hlambda, \hh) \,.\footnote{A notational remark:
for the directed polymer model, the denominator
in \eqref{eq:conjfe} should be $\delta$ instead of $\delta^d$, due to a different
normalization of the discrete free energy, cf.\ Subsection~\ref{sec:dpre};
for the random field Ising model, one should set
$h_\delta := \delta^{d-\gamma} \hh$ without the
``$- \frac{1}{2}(\lambda_\delta)^2$ correction'', as already discussed
after \eqref{eq:isingscale}. These notational details are discussed
in Section~\ref{sec:scaling} for each model, while here we keep a unified approach.}
\end{equation}
The heuristics goes as follows:
by \eqref{eq:dfe} we can write (replacing $\limsup$ by $\lim$ for simplicity)
\begin{equation} \label{eq:co1}
\lim_{\delta \downarrow 0} \, \frac{F(\lambda_\delta, h_\delta)}
	{\delta^d}  =
	\lim_{\delta \downarrow 0} \,
	\lim_{\Omega \uparrow \R^d} \, \frac{1}{\delta^d}
	\frac{1}{|\bbOmegadelta|} \bbE \big[ \log
	Z_{\bbOmegadelta; \lambda_\delta, h_\delta}^\omega \big] \,;
\end{equation}
on the other hand, applying \eqref{eq:Zcont} in
\eqref{eq:contfe} (assuming uniform integrability)
and noting that $Leb(\Omega) = \lim_{\delta \downarrow 0} \delta^d |\bbOmegadelta|$, one gets
\begin{equation} \label{eq:co2}
	\bF(\hlambda, \hh) = \lim_{\Omega \uparrow \R^d} \,
	\lim_{\delta \downarrow 0} \ \frac{1}{\delta^d}
	\frac{1}{|\bbOmegadelta|} \bbE \big[ \log
	Z_{\bbOmegadelta; \lambda_\delta, h_\delta}^\omega \big] \,.
\end{equation}
Therefore proving \eqref{eq:conjfe}
amounts to \emph{interchanging the infinite volume} ($\Omega \uparrow \R^d$)
and \emph{continuum and weak disorder} ($\delta\downarrow 0$) limits.
This is in principle a delicate issue, but we expect relation \eqref{eq:conjfe} to
hold in many interesting cases,
such as the three motivating models
in the specified parameters range
(and, more generally, when the continuum correlations are ``non trivial'';
see the next point).
This is an interesting open problem.

Relation \eqref{eq:conjfe}
implies that the discrete free energy $F(\lambda, h)$
has a \emph{universal shape for weak disorder $\lambda, h \to 0$}.
This leads to sharp predictions on the asymptotic behavior
of free energy-related quantities, such as critical
curves and order parameters.
Consider, e.g., the average magnetization
$\langle \sigma_0 \rangle_{\beta_c, h}$
in the critical Ising model on $\Z^2$ with
a homogeneous external field $h > 0$.
If relation \eqref{eq:conjfe} holds
(with $d=2$, $\lambda_\delta = \hlambda\, \delta^{\frac{7}{8}}$,
$h_\delta = \hh\, \delta^{\frac{15}{8}}$, cf.\ \eqref{eq:isingscale},
and we look at the case $\hlambda = 0$),
differentiating both sides with respect to $\hh$
\emph{suggests} that
\begin{equation} \label{eq:sharpCGN}
	\lim_{h\downarrow 0} \frac{\langle \sigma_0
\rangle_{\beta_c, \, h}}{h^{\frac{1}{15}}}
	= \frac{\partial \bF}{\partial \hh} (0,1) \,,
\end{equation}
which would sharpen the results in \cite{CGN12b}.
Analogous predictions can be formulated
for disorder pinning and directed polymer models
(see Section~\ref{sec:scaling}).
Of course, proving such precise estimates
is likely to require substantial additional work,
but having a candidate for the limiting constants,
like in \eqref{eq:sharpCGN}, can be of great help.

\begin{remark}\rm
Relation \eqref{eq:conjfe} (in a stronger form) has been proved in \cite{BdH97,CG10}
for the so-called \emph{disordered copolymer model}, by means of a subtle
coarse-graining procedure. We mention that our approach can also be applied
to the copolymer model, yielding a Wiener chaos expansion
as in \eqref{eq:Zcont} for the continuum partition function.
\end{remark}

\medskip
\noindent\textbf{3. (First-order phase transitions).}
Relation \eqref{eq:polyconv} can hold
with $\gamma = 0$ (i.e.,
the $k$-point correlation function converges
without rescaling) and with a ``trivial'' factorized limit:
\begin{equation} \label{eq:toogen}
	\psi^{(k)}_{\bbOmegadelta}(x_1,\ldots, x_k) \xrightarrow[\delta\downarrow 0]{}
	\bpsi^{(k)}_{\Omega}(x_1,\ldots, x_k) := \rho^k \,,
	\qquad \text{with} \quad \rho \in (0,\infty) .
\end{equation}
This is typical for a system at the critical point of a \emph{first-order phase transitions}
(i.e., the order parameter $m_h :=
\lim_{\delta\downarrow 0} |\bbOmegadelta|^{-1} \E_{\bbOmegadelta; 0, h}
[ \sum_{x \in \bbOmegadelta} \sigma_x]$, as a function of $h$,
has a jump discontinuity at $h=0$).
Examples include the \emph{pinning model} for $\alpha > 1$ and the
\emph{Ising model}
for $\beta > \beta_c$. Plugging \eqref{eq:toogen} into
\eqref{eq:Zcont} and performing the integration, cf.\
\cite[\S3.2]{J97}, one gets
\begin{equation} \label{eq:contexp}
	\bZ_{\Omega; \hlambda, \hh}^W =
	\exp\bigg\{\rho \hlambda W(\Omega) +
	\bigg( \rho \hh - \frac{1}{2} (\rho\hlambda)^2 \bigg) Leb(\Omega) \bigg\} \,.
\end{equation}
This explicit formula allows exact asymptotic computations on
the discrete model: e.g., relation \eqref{eq:Zcont} yields for suitable
values of $\zeta \in \R$ (when uniform integrability holds)
\begin{equation}\label{eq:inci}
	\bbE\Big[ \big( Z_{\bbOmegadelta; \lambda,h}^\omega \big)^\zeta
	\Big] \xrightarrow[\delta\downarrow 0]{}
	\bbE\Big[ \big(\bZ_{\Omega; \hlambda, \hh}^W\big)^\zeta \Big]
	= \exp\bigg\{ \rho \zeta \bigg( \hh - \frac{1}{2}
	\rho \hlambda^2 (1-\zeta) \bigg) Leb(\Omega) \bigg\} \,.
\end{equation}
Incidentally, for disordered pinning models with $\alpha > 1$, this estimate
clarifies
the strategy for the sharp asymptotic behavior of the critical curve $h_c(\lambda)$
as $\lambda \downarrow 0$,
determined
in \cite{BCPSZ14} (even though the proof therein is carried out with different techniques).

Unfortunately, the continuum partition function \eqref{eq:contexp}
can fail in capturing some key properties of the discrete model,
because it is shared by many ``too different''
models: relation \eqref{eq:toogen} asks that
the field variables under the reference law $\P^{\re}$ become uncorrelated
as $\delta\downarrow 0$, but is insensitive toward the
\emph{correlation decay},
which could be polynomial, exponential, or even finite-range (like in the extreme case
of a ``trivial'' reference law $\P^\re$, under which
$(\sigma_x)_{x\in\bbOmegadelta}$ are i.i.d.\ with
$\E^{\re}(\sigma_x) = \rho$).
Since the correlation decay can affect substantially the discrete free energy,
conjecture \eqref{eq:conjfe} usually \emph{fails} under \eqref{eq:toogen}.

For example,
for disordered pinning models with $\alpha > 1$, one always has $F(\lambda, h) \ge 0$
(there is only
a polynomial cost for the underlying renewal process not to return before time $N$,
and the energy
of such a renewal configuration is $0$). On the other hand, if the renewal jump distribution
has finite exponential moments,
then there is an exponential cost for the renewal not to
return before time $N$, and $F(\lambda, h)<0$
if $h$ is sufficiently negative. Both models satisfy \eqref{eq:toogen}
with $\rho = 1/\E[\tau_1]$ and thus their continuum partition functions
coincide, but their free energies
depend on finer detail of the renewal distribution (beyond the value of
$\E[\tau_1]$) and are therefore radically different,
causing \eqref{eq:conjfe} to fail. The continuum free energy is
$\bF(\hlambda, \hh) = \rho \hh - \frac{1}{2} (\rho\hlambda)^2$, cf. \eqref{eq:contexp}
and \eqref{eq:contfe}, which can attain negative values.
\medskip

\noindent
\textbf{4. (Moment assumptions).} In our convergence results,
cf.~Theorems~\ref{pinning scaling}, \ref{DP scaling}
and~\ref{T:RFIM}, we assume that the disorder variables $(\omega_x)_{x\in\bbOmegadelta}$
have finite exponential moments, which
guarantees that the expectation and variance in \eqref{eq:Taylor} are well-defined. However,
this assumption can be relaxed to
finite moments. The necessary number of moments
depends on the model and can be determined by the requirement that
the typical maximum value of the variables
$\omega_x$ ``sampled'' by the field does not exceed the reciprocal of the
disorder strength $\lambda$ (so that a truncation of $\omega_x$
at level $\lambda^{-1}$ provides a good approximation).

For example,
in the long-range directed polymer model,
one expects that the path will be confined (at weak disorder) in a
box of size $N\times N^{1/\alpha}$. If the disorder variables have
a polynomial tail $\bbP(\omega_x> y)\approx y^{-\eta}$ as $y\uparrow \infty$,
their maximum in such a box is of the order $N^{\frac{1}{\eta}\frac{1+\alpha}{\alpha}}$.
Since $\lambda \approx N^{-\frac{\alpha-1}{2\alpha}}$ for this model,
cf.\ Theorem~\ref{DP scaling},
one expects the validity of the convergence result as long as
$N^{\frac{1}{\eta}\frac{1+\alpha}{\alpha}} \ll N^{\frac{\alpha-1}{2\alpha}}$, i.e.\ for
$\eta\ge2(\alpha+1)/(\alpha-1)$. For $\alpha=2$, this gives $\eta\ge6$, which was conjectured
in \cite{AKQ14a} and recently proved in \cite{DZ}.

Similarly, for the pinning model the number of relevant variables is of
order $N$ and $\lambda \approx N^{-(\alpha - \frac{1}{2})}$,
cf.\ Theorem~\ref{pinning scaling},
leading to a conjectured value 
$\eta\ge2/(2\alpha-1)$; for the RFIM, the
number of relevant variables is of order $N^2$
and $\lambda \approx N^{-7/8}$,
cf.\ Theorem~\ref{T:RFIM}, leading to
a conjectured value  $\eta\ge16/7$.

\section{From polynomial to Wiener chaos via Lindeberg}
\label{ConvergencePoly}

In this section, which can be read independently of the previous one,
we first recall the main properties of white noise on $\R^d$
(Subsection~\ref{sec:wnn}) and
define polynomial chaos expansions (Subsection~\ref{sec:Polydef}).
We then formulate our main general theorem (Subsection~\ref{sec:Poly}),
ensuring convergence of polynomial chaos toward Wiener chaos expansions.
This is based on a Lindeberg principle (Subsection~\ref{sec:Linde})
which extends results in \cite{MOO10} to optimal second moment assumptions.
Subsections~\ref{sec:Poly} and~\ref{sec:Linde} can be read independently.

The space of Lebesgue square-integrable functions $f: \R^d \to \R$ is
denoted by $L^2(\R^d)$, and we set
$\| f \|_{L^2(\R^d)}^2 = \int_{\R^d} f(x)^2 \dd x$.
For more details on the white noise, we refer to \cite{J97,PT10}.

\subsection{White noise in a nutshell}
\label{sec:wnn}

By \emph{white noise on $\R^d$} we mean a
Gaussian process $W = (W(f))_{f \in L^2(\R^d)}$
with $\bbE[W(f)] = 0$ and $\bbcov (W(f), W(g)) = \int_{\R^d} f(x) g(x) \dd x$,
defined on some probability space $(\Omega_W, \cA, \bbP)$.
Since the specified covariance is a symmetric and
positive definite function,
such a process exists (and is unique in law).

If $A_1, A_2, \ldots$ are disjoint Borel sets with finite Lebesgue measure, it follows that
the random variables $(W(A_i) := W(\ind_{A_i}))_{i = 1, 2, \ldots}$
are independent $\cN(0,Leb(A_i))$
and the relation $W(\bigcup_{i \ge 1} A_i) = \sum_{i\ge 1} W(A_i)$ holds a.s..
Consequently, it is suggestive to use the notation
\begin{equation} \label{eq:Wsingle}
	\int_{\R^d} f(x) \, W(\dd x) := W(f) \,,
\end{equation}
even though $W(\cdot)$ is a.s. \emph{not} a signed measure on $\R^d$.
For $d=1$, $W(f)$ coincides with the usual
Wiener integral $\int f(t) \, \dd W_t$ with respect to the Brownian motion $W_t :=
W(\ind_{[0,t]})$.

One can define a multi-dimensional stochastic integral
$W^{\otimes k}(f)$, for $k\in\N$ and suitable $f: (\R^d)^k \to \R$, as follows.
For ``special indicator functions'' $f = \ind_{A_1 \times \ldots \times A_k}$ built over
\emph{disjoint} bounded Borel sets $A_1, \ldots, A_k \subseteq \R^d$,
one poses $W^{\otimes k}(f) := W(\ind_{A_1}) \cdots W(\ind_{A_k})$.
This definition is extended, by linearity, to the space
$\cS_k$ of ``special simple functions'',
i.e. finite linear combinations of
special indicator functions.
Since a permutation of the arguments of $f$ leaves $W^{\otimes k}(f)$ invariant,
it is sufficient to consider \emph{symmetric} functions $f$, which we do henceforth.
One then observes that $\bbE[W^{\otimes k}(f)] = 0$ and
the crucial \emph{Ito isometry} is satisfied:\footnote{For the
isometry \eqref{eq:WI} it is essential to
``avoid diagonals'': this is the reason for taking special indicator functions,
corresponding to products of \emph{disjoint} Borel sets.}
\begin{equation} \label{eq:WI}
	\bbcov(W^{\otimes k}(f), W^{\otimes l}(g)) = k! \, \ind_{\{k=l\}}
	\int_{(\R^d)^k} f(x_1, \ldots, x_k) \, g(x_1, \ldots, x_k) \, \dd x_1 \cdots \dd x_k \,.
\end{equation}
Since $\cS_k$ is dense in $L^2((\R^d)^k)$,
one can finally extend the definition of $W^{\otimes k}(f)$
to every symmetric $f \in L^2((\R^d)^k)$, in such a way
that \eqref{eq:WI} still holds. Like in \eqref{eq:Wsingle}, we will write
\begin{equation} \label{eq:Wk}
	 \idotsint_{(\R^d)^k} f(x_1, \ldots, x_k) \, W(\dd x_1) \cdots W(\dd x_k)
	:= W^{\otimes k}(f)  \,.
\end{equation}
One also sets $W^{\otimes 0}(c) := c$
for $c\in L^2((\R^d)^0) := \R$.

Note that $W^{\otimes k}(f)$ is a random variable defined on $(\Omega_W, \cA, \bbP)$,
with zero mean (for $k\ge 1$) and finite variance,
which is measurable with respect to the $\sigma$-algebra
$\sigma(W)$ generated by the white noise $W$.
(One can show that $W^{\otimes k}(f)$ is non Gaussian for $k > 1$ and $f \not\equiv 0$.)
Remarkably, \emph{every} square-integrable random variable $X$ defined on $\Omega_W$,
which is measurable with respect to $\sigma(W)$,
can be written as the $L^2$-convergent series
\begin{equation} \label{eq:WC}
	X = \sum_{k=0}^\infty \frac{1}{k!} W^{\otimes k}(f_k) \,,
\end{equation}
called \emph{Wiener chaos expansion},
for a unique choice of symmetric functions $f_k \in L^2((\R^d)^k)$
satisfying $\sum_{k=0}^\infty \frac{1}{k!} \| f_k \|_{L^2(\R^d)}^2 < \infty$,
by \eqref{eq:WI}.
In other terms, the multiple stochastic integrals $W^{\otimes k}(f)$ span
the whole Hilbert space $L^2(\Omega_W, \sigma(W), \bbP)$.

\subsection{Polynomial chaos}
\label{sec:Polydef}

Let $\bbT$ be a finite or countable
index set (e.g., $\bbT = \{1,\ldots, N\}$, $\bbT = \N$, $\bbT = \Z^d$). We set
\begin{equation*}
	\cP^\fin(\bbT):=\{ I\subseteq \bbT:\ |I|<\infty\} \,.
\end{equation*}
Any function $\psi : \cP^\fin(\bbT) \to \R$ determines a (formal, if $|\bbT|=\infty$)
\emph{multi-linear polynomial}
$\Psi$:
\begin{equation} \label{eq:Psi}
	\Psi (x) = \sum_{I \in \cP^\fin(\bbT)} \psi(I) \, x^I \,,
	\qquad \quad \text{where} \qquad \quad
	x^I := \prod_{i\in I} x_i \quad \text{with} \quad x^\emptyset := 1 \,.
\end{equation}
We say that $\psi : \cP^\fin(\bbT) \to \R$ is the {\em kernel} of $\Psi$.

Let now $\zeta:=(\zeta_i)_{i \in \bbT}$ be a family of independent (but not necessarily identically
distributed) random variables. We say that a random variable $X$ admits a
{\em polynomial chaos expansion} with respect to $\zeta$ if it can be expressed as
$X = \Psi(\zeta)=\Psi((\zeta_i)_{i\in \bbT})$ for some multi-linear polynomial
$\Psi$. Of course, when $|\bbT|=\infty$ some care is needed:
by $X = \Psi(\zeta)$ we mean that
for any sequence $\Lambda_N\subset \bbT$ with $|\Lambda_N|<\infty$
and $\Lambda_N\uparrow \bbT$ one has
\begin{equation} \label{eq:Xproba}
	X = \lim_{N\to\infty}
	\sum_{I \subseteq \Lambda_N} \psi(I) \, \zeta^I \qquad
	\text{in probability} \,.
\end{equation}

\begin{remark}\rm
When $\bbvar(\zeta_i) \ne 0$ for all $i\in\bbT$, we can assume that
all the variances are equal with no loss of generality:
it suffices to redefine $\psi(I) \to \psi(I)(\prod_{i\in I} \bbvar(\zeta_i))^{-1/2}$.
\end{remark}

\begin{remark}\rm\label{rem:L2easy}
When the independent random variables $\zeta:=(\zeta_i)_{i \in \bbT}$
have zero mean and the same variance $\sigma^2$,
an easy sufficient condition for \eqref{eq:Xproba}, with $L^2$ convergence, is
\begin{equation}\label{eq:CPsi0}
	\sum_{I \in \cP^\fin(\bbT)} (\sigma^2)^{|I|}\psi(I)^2 <\infty \,,
\end{equation}
because $\bbE[\zeta^I \zeta^J] = 0$ for $I \ne J$.
For variables with non-zero mean $\mu:=(\mu_i)_{i\in \bbT}$
(always with the same variance $\sigma^2$),
sharp conditions for $L^2$ convergence in \eqref{eq:Xproba}
involve $\mu$ and the kernel $\psi$ jointly. As we show below, practical
sufficient ``factorized'' conditions are
\begin{equation} \label{eq:factcond}
	\sum_{i\in\bbT} \mu_i^2 < \infty \,; \qquad \ \ \
 	\exists \epsilon >0: \ \ \sum_{I \in \cP^\fin(\bbT)}
	(1+\epsilon)^{|I|} (\sigma^2)^{|I|} \psi(I)^2 <\infty \,.
\end{equation}
\end{remark}

\subsection{Convergence of polynomial chaos to Wiener chaos}
\label{sec:Poly}

Consider for $\delta \in (0,1)$ an index
set $\bbT_\delta \subset \R^d$ and a family of polynomial chaos expansions
$(\Psi_\delta(\zeta_\delta))_{\delta\in (0,1)}$,
defined from kernels $\psi_\delta : \cP^\fin(\bbT_\delta)\to \R$ and from independent random
variables $\zeta_\delta:=(\zeta_{\delta, x})_{x\in \bbT_\delta}$.
If $\bbT_\delta$ converges to the
continuum $\R^d$ as $\delta\downarrow 0$ (e.g., $\bbT_\delta:=(\delta\Z)^d$),
then after suitable scaling, the random variables
$(\zeta_{\delta, x})_{x\in \bbT_\delta}$ approximate the white noise
$W(\dd x)$ on $\R^d$.
If the kernel $\psi_\delta$, suitably rescaled, converges as $\delta \downarrow 0$ to
a continuum kernel $\bpsi_0 : \cP^\fin(\R^d) \to \R$,
it is plausible that the polynomial chaos expansion $\Psi_\delta(\zeta_\delta)$
approximates a \emph{Wiener chaos expansion} $\bPsi_0$, cf. \eqref{eq:Wk}-\eqref{eq:WC},
with kernel $\bpsi_0$.
This is precisely what we are going to show.

First we introduce some notation. Each random variable $\zeta_x$ indexed by a point $x$ in an index set $\bbT\subset\R^d$ will be associated with a cell in $\R^d$ containing $x$, and functions defined on $\bbT^k$ will be extended to functions defined on $(\R^d)^k$.

\begin{itemize}
\item Let $\cB(\R^d)$ denote the Borel subsets of $\R^d$.
Given a locally finite set $\bbT \subset \R^d$, we call $\cC: \bbT \to \cB(\R^d)$
a {\em tessellation} of $\R^d$ indexed by $\bbT$, if $(\cC(x))_{x\in \bbT}$ form a
disjoint partition of $\R^d$ such that $x\in \cC(x)$ for each $x\in \bbT$. We call $\cC(x)$
the {\em cell} associated with $x\in \bbT$. In most cases, $(\cC(x))_{x\in \bbT}$ will be the cells
of a cubic lattice. However, there are natural examples where this is not the case, such as the directed polymer model defined from
a simple symmetric random walk, or the Ising model defined on non-cubic lattices.

\item Once a tessellation $\cC$ is fixed,
any function $f: \bbT \to \R$ is automatically extended to $f: \R^d \to \R$ by
assigning value $f(y):=f(x)$ for all $y\in \cC(x)$, for each $x\in \bbT$.
Note that for such extensions $\Vert f\Vert_{L^2(\R^d)}^2 = \sum_{x\in\bbT} f(x)^2 \, Leb(\cC(x))$.

\item Analogously, for any $\psi : \cP^\fin(\bbT) \to \R$, we first extend it to
$\psi : \bigcup_{k=0}^\infty\bbT^k \to \R$ by setting $\psi(x_1, \ldots, x_k) :=
\psi (\{x_1, \ldots, x_k\})$
if the $x_i$ are distinct,
and $\psi(x_1, \ldots, x_k) :=0$ otherwise. We then
extend it to $\psi : \bigcup_{k=0}^\infty (\R^d)^k \to \R$ by assigning
value $\psi(x_1, \ldots, x_k)$ to all points in $\cC(x_1)\times \cdots \times \cC(x_k)$,
for each $k\in\N$ and $x_1, \ldots, x_k\in \bbT$.

\item
Given $\psi : \cP^\fin(\R^d) \to \R$, its extension to $\psi : \bigcup_{k=0}^\infty (\R^d)^k \to \R$
is defined similarly (no cells involved). It will be clear from the context which version of $\psi$
is being used.

\item
Finally, given a measurable function $\psi : \bigcup_{k=0}^\infty (\R^d)^k \to \R$,
we denote by $\Vert \psi \Vert_{L^2((\R^d)^k)}$ the $L^2$ norm
of the restriction of $\psi$ to $(\R^d)^k$, i.e.
$$
\Vert \psi \Vert_{L^2((\R^d)^k)}^2 = \idotsint_{(\R^d)^k} \psi(x_1, \ldots, x_k)^2
\, \dd x_1 \cdots \dd x_k.
$$
\end{itemize}
We are now ready to state our main convergence result,
proved in Section~\ref{sec:genproof}.

\begin{theorem}[Convergence of polynomial chaos to Wiener chaos, $L^2$ case]\label{th:general}
Assume that for $\delta \in (0,1)$ the following ingredients are given:
\begin{itemize}
\item Let $\bbT_\delta$ be a locally finite subset of $\R^d$;

\item Let $\zeta_\delta:=(\zeta_{\delta, x})_{x\in \bbT_\delta}$ be independent random
variables in $L^2$ with the same variance,
$$
\bbE[\zeta_{\delta, x}]=\mu_{\delta}(x) \qquad \mbox{and}
\qquad  {\rm \bbV ar}(\zeta_{\delta, x})= \sigma^2_\delta \,,
$$
such that
$((\zeta_{\delta, x} - \bbE[\zeta_{\delta, x}])^2)_{\delta \in (0,1), x\in \bbT_\delta}$
are uniformly integrable;

\item Let $\Psi_\delta(z)$ be a formal multi-linear polynomial
with kernel $\psi_\delta : \cP^\fin(\bbT_\delta) \to \R$, cf.\ \eqref{eq:Psi};

\item Let $\cC_\delta$ be a tessellation of $\R^d$ indexed by $\bbT_\delta$,
where every cell $\cC_\delta(x)$ has the
same volume $v_\delta := Leb(\cC_\delta(x))$.
\end{itemize}
Assume that $v_\delta \to 0$ as $\delta\downarrow 0$, and
that the following conditions are satisfied:
\begin{ienumerate}
\item\label{it:1} There exist $\bsigma_0 \in (0,\infty)$ and $\bmu_{0} \in L^2(\R^d)$ such
that
\begin{equation}\label{eq:barmu}
	\lim_{\delta \downarrow 0} \sigma_\delta = \bsigma_0 \,, \qquad
	\lim_{\delta\downarrow 0} \| \bar\mu_\delta - \bmu_{0} \|_{L^2(\R^d)}  = 0, \qquad
	\text{where } \
	\bar\mu_\delta(x) := v_\delta^{-1/2}\mu_\delta(x);
\end{equation}

\item\label{it:2} There exists $\bpsi_{0}: \cP^\fin(\R^d) \to \R$, with
$\Vert \bpsi_{0}\Vert_{L^2((\R^d)^k)}<\infty$ for every $k\in\N_0$, such that
\begin{equation} \label{eq:barpsi}
	\lim_{\delta\downarrow 0} \| \bar\psi_\delta -\bpsi_{0}\|_{L^2((\R^d)^k)}^2 = 0, \qquad
	\text{where } \ \bar\psi_\delta(I) := v_\delta^{-|I|/2}\psi_\delta(I);
\end{equation}

\item\label{it:3} For some $\epsilon > 0$
(or even $\epsilon = 0$, if $\mu_\delta(x) \equiv 0$)
\begin{equation}\label{psidtail1}
	\lim_{\ell \to \infty} \, \limsup_{\delta\downarrow 0}
	\sum_{I \in \cP^\fin(\bbT_\delta), |I| > \ell} (1+\epsilon)^{|I|} \,
	(\sigma_\delta^2)^{|I|} \psi_\delta(I)^2 = 0 \,.
\end{equation}

\end{ienumerate}
Then the polynomial chaos expansion
$\Psi_\delta(\zeta_\delta)$ is well-defined
and converges in distribution as $\delta\downarrow 0$ to
a random variable $\bPsi_{0}$ with an explicit Wiener chaos expansion:
\begin{equation} \label{eq:general}
\Psi_\delta(\zeta_\delta) \,\xrightarrow[\,\delta\downarrow 0\,]{d}\,
\bPsi_{0} := \sum_{k = 0}^\infty \frac{1}{k!} \idotsint_{(\R^d)^k} \bpsi_{0} (y_1, \ldots, y_k)
	\prod_{i=1}^k \Big(\bsigma_0 W(\dd y_i) + \bmu_{0}(y_i) \dd y_i \Big) \,,
\end{equation}
where $W(\cdot)$ denotes white noise on $\R^d$.

The series in \eqref{eq:general} converges in $L^2$,
and $\bbE[\Psi_\delta(\zeta_\delta)^2] \to \bbE[\bPsi_{0}^2]$. Consequently,
for any coupling of $\Psi_\delta(\zeta_\delta)$ and $\bPsi_{0}$ such
that $\Psi_\delta(\zeta_\delta) \to \bPsi_{0}$ a.s.,
one has $\bbE[|\Psi_\delta(\zeta_\delta) - \bPsi_{0}|^2] \to 0$.

The convergence \eqref{eq:general} extends to the joint distribution of a finite collection of
polynomial chaos expansions $(\Psi_{i,\delta}(\zeta_\delta))_{1\leq i\leq M}$,
provided $(\Psi_{i,\delta})_{\delta \in (0,1)}$ satisfies \eqref{it:2}-\eqref{it:3}
above for each $i$.
\end{theorem}

\begin{remark}\rm \label{rem:properdef}
Let us be more precise about the random variable $\bPsi_{0}$ in \eqref{eq:general}.
Setting $\bnu(x) := \bmu_0(x) / \bsigma_0$, it can be rewritten as
\begin{equation} \label{eq:general2}
\bPsi_{0} = \sum_{k = 0}^\infty \frac{1}{k!} \idotsint_{(\R^d)^k} \bpsi_{0} (y_1, \ldots, y_k)
\bsigma_0^k
\prod_{i=1}^k \Big( W(\dd y_i) + \bnu(y_i) \dd y_i \Big) \,,
\end{equation}
which can be viewed as a ``Wiener chaos expansion with respect to the
\emph{biased} white noise $W_{\bnu}(\dd x) := W(\dd x) + \bnu(x) \dd x$''.
The rigorous definition of such an expansion goes as follows.
For every \emph{fixed} $k\in\N$,
the integral over $(\R^d)^k$ in \eqref{eq:general2} can be defined by
expanding the product and
integrating out the ``deterministic coordinates'' (those corresponding
to $\bnu(y_i) \dd y_i$), obtaining a finite sum
of well-defined (lower-dimensional) ordinary stochastic integrals,
like in \eqref{eq:Wk}.
Regrouping the terms,
the series in \eqref{eq:general2} becomes an ordinary Wiener chaos
expansion, like in \eqref{eq:WC}. In
analogy with the polynomial case \eqref{eq:factcond},
we show in Section~\ref{sec:genproof} that
the $L^2$-convergence of the series
is ensured by the conditions that
$\bmu_0 \in L^2(\R^d)$ and that
\begin{equation} \label{eq:condepsilon}
	\exists\, \epsilon > 0: \qquad
	\sum_{k = 0}^\infty \frac{1}{k!} (1+\epsilon)^k \,
	(\bsigma_0^2)^{k} \, \| \bpsi_0 \|_{L^2((\R^d)^k)}^2 < \infty \,,
\end{equation}
which follow by assumptions \eqref{it:1}-\eqref{it:2}-\eqref{it:3} in Theorem~\ref{th:general}.
\end{remark}

\subsubsection{Beyond the $L^2$ case}
\label{sec:beyond}

There is a useful alternative interpretation of
\eqref{eq:general}-\eqref{eq:general2}.
If $(\Omega_W, \cA, \bbP)$ is the probability space on which
the white noise $W = (W(f))_{f\in L^2(\R^d)}$ is defined,
for every $\bnu \in L^2(\R^d)$ we introduce a new probability $\bbP_{\bnu}$ on $\Omega_W$ by
\begin{equation} \label{eq:RN}
	\frac{\dd \bbP_{\bnu}}{\dd \bbP} := e^{W(\bnu) - \frac{1}{2} \bbE[W(\bnu)^2]}
	= e^{\int_{\R^d} \bnu(x) W(\dd x) - \frac{1}{2} \int_{\R^d} \bnu(x)^2 \dd x} \,.
\end{equation}
It turns out that the ``biased stochastic integrals'' in \eqref{eq:general2}
have the same joint distribution as the ordinary stochastic integrals
(with $\bnu$ replaced by $0$) under the probability $\bbP_{\bnu}$,
by the \emph{Cameron-Martin theorem} (cf.\ Appendix~\ref{sec:wnapp}). As a consequence,
the random variable $\bPsi_0$ in \eqref{eq:general} enjoys the following equality in distribution,
setting $\bnu(x) := \bmu_0(x) / \bsigma_0$:
\begin{equation} \label{eq:general3}
	\bPsi_0 \overset{d}{=}
	\sum_{k = 0}^\infty \frac{1}{k!} \idotsint_{(\R^d)^k}
	\bpsi_{0} (y_1, \ldots, y_k) \, \bsigma_0^k \,
	W(\dd y_1) \cdots W(\dd y_k)
	\quad \ \text{under} \ \ \bbP_{\bnu} \,,
\end{equation}
provided the series (either in \eqref{eq:general} or \eqref{eq:general3}, equivalently)
converges in probability.

Let us now assume the weaker version of relation \eqref{eq:condepsilon} for $\epsilon = 0$, i.e.
\begin{equation} \label{eq:condepsilon0}
	\sum_{k = 0}^\infty \frac{1}{k!} \,
	(\bsigma_0^2)^{k} \, \| \bpsi_0 \|_{L^2((\R^d)^k)} < \infty \,.
\end{equation}
Under this condition, the series in \eqref{eq:general3} converges in $L^2$
\emph{under the original probability $\bbP$}, by the
It\^o isometry \eqref{eq:WI}.
Since the Radon-Nikodym density \eqref{eq:RN} has finite moments of all orders,
it follows by \eqref{eq:general3} and an application of H\"older inequality
(see~\eqref{eq:bass} for the details) that \emph{the series in \eqref{eq:general}
defining $\bPsi_0$ converges in $L^p$ for every $p \in (0,2)$
when \eqref{eq:condepsilon0} holds}
(even though it might not converge in $L^2$,
if \eqref{eq:condepsilon} fails).

As a consequence, by performing an $L^p$ analysis for $p < 2$, we can
weaken condition~\eqref{it:3} in Theorem~\ref{th:general},
setting $\epsilon = 0$ in \eqref{psidtail1}, under mild restrictions
on the disorder distribution (due to the implementation of a change of measure
like in \eqref{eq:general3} for polynomial chaos).

\begin{theorem}\label{th:general2}
{\bf (Convergence of polynomial chaos to Wiener chaos, $L^{2-}$ case)}
Let the same assumptions as in Theorem~\ref{th:general} hold, with condition
\eqref{it:3} therein weakened by setting $\epsilon=0$ in \eqref{psidtail1}.
Assume further that $\lim_{\delta\downarrow 0}\Vert\mu_\delta\Vert_\infty=0$,
and that either of the following two conditions is satisfied:
\begin{aenumerate}

\item\label{it:a} \hfill$\inf\limits_{\delta\in (0,1), x\in\bbT_\delta} \min\big\{
\bbP(\zeta_{\delta,x}>0), \bbP(\zeta_{\delta,x}<0),
\bbvar (\zeta_{\delta, x}|\zeta_{\delta, x}>0),
\bbvar (\zeta_{\delta, x}|\zeta_{\delta, x}<0) \big\} >0$;\hfill\,

\item\label{it:b}
\begin{equation}\label{psidtail2}
	\forall C>0: \qquad
	\lim_{\delta\downarrow 0} \sum_{I \in \cP^\fin(\bbT_\delta), \, |I| > \Vert\mu_\delta\Vert_\infty^{-1}} e^{C \Vert\mu_\delta\Vert_\infty|I|}
	(\sigma_\delta^2)^{|I|} \psi_\delta(I)^2 = 0 \,.
\end{equation}
\end{aenumerate}
Then the polynomial chaos expansion
$\Psi_\delta(\zeta_\delta)$ is well-defined
and converges in distribution as $\delta\downarrow 0$ to
the random variable $\bPsi_0$ defined by \eqref{eq:general},
or equivalently \eqref{eq:general3}.
For all $p \in (0,2)$, the series therein converges in $L^p$,
and furthermore $\bbE[|\Psi_\delta|^p] \to \bbE[|\bPsi_{0}|^p]$.
The conclusion extends to a finite collection $(\Psi_{i, \delta}(\zeta_\delta))_{1\leq i\leq M}$.
\end{theorem}

\subsection{Lindeberg principle for polynomial chaos}
\label{sec:Linde}

The key ingredients in our proof of Theorem~\ref{th:general} are two Lindeberg principles for
polynomial chaos. As we discuss in Remark \ref{MOOremark}, they extend Theorem 3.18
in~\cite{MOO10} in two ways: firstly, we relax the finite third moment assumption of \cite{MOO10} to
an optimal condition of
uniform integrability of the square of the random variables; secondly, we allow
random variables with non-zero mean.

We point out that the first extension is actually not needed for our applications to
disordered systems, due to the assumption of finite exponential moments
for the disorder random variables. However, it is
an extension of general interest
and will be useful if one attempts to weaken the
moment assumptions on the disorder random variables,
as discussed at the end of Section~\ref{sec:discussion}.
We remark that Lindeberg principles have also played crucial roles in recent breakthrough results on
random matrices~\cite{C06, TV11}.

Given a polynomial chaos expansion $\Psi(\zeta)$ with respect to
a family $\zeta:=(\zeta_i)_{i\in \bbT}$ of independent random variables
(cf.\ Subsection~\ref{sec:Polydef}),
we will control how the distribution of $\Psi(\zeta)$ changes when we replace
$\zeta$ by independent Gaussian random variables $\xi:=(\xi_i)_{i\in \bbT}$
with the same mean and variance as $\zeta$.

Given a multi-linear polynomial $\Psi(x)=\Psi((x_i)_{i\in \bbT})$ as in (\ref{eq:Psi}),
with kernel $\psi$, we set
\begin{equation}\label{eq:CPsi}
	\sfC_\Psi := \sum_{I \in \cP^\fin(\bbT), \,I\neq \emptyset} \psi(I)^2 \,,
\end{equation}
and define the \emph{influence} of the $i$-th variable $x_i$ on $\Psi$ by
\begin{equation} \label{eq:Inf}
	\Inf_i[\Psi] := \sum_{I\in \cP^\fin(\bbT), \,I \ni i} \psi(I)^2 \,.
\end{equation}
Note that, if $\zeta = (\zeta_i)_{i\in \bbT}$ are independent
random variables with zero mean and unit variance,
$$
\sfC_\Psi = \bbvar[\Psi(\zeta)]\,, \qquad
\Inf_i[\Psi] =\bbE \big[ \bbvar \big[ \Psi(\zeta) \big| (\zeta_j)_{j \in \bbT \setminus \{i\}} \big]
\big],
$$
which is just the influence of the random variable $\zeta_j$ on $\Psi(\zeta)$ 
introduced in~\cite{MOO10} (for more on the notion of influence, see e.g.~\cite{KKL88, BKKKL92} 
and the references in~\cite{MOO10}).
We also define the degree $\ell$ truncations
$\Psi^{\le \ell}$ and $\Psi^{> \ell}$ of the multi-linear polynomial $\Psi$ by
\begin{equation} \label{eq:Psid}
	\Psi^{\le \ell}(x) := \sum_{I\in \cP^\fin(\bbT), \, |I| \le \ell} \psi(I)\, x^I,
\qquad \
\Psi^{> \ell}(x) := \sum_{I\in \cP^\fin(\bbT), \, |I| > \ell} \psi(I) \, x^I,
\end{equation}
whose kernels will be denoted by $\psi^{\le \ell}(I) = \psi(I) \ind_{\{ |I|\le \ell\}}$
and $\psi^{>\ell}(I) = \psi(I) \ind_{\{ |I|>\ell\}}$.

\medskip

We are now ready to state and comment on our Lindeberg principles,
that will be proved in Section~\ref{S:lindeberg}.

\begin{theorem}[Lindeberg principle, zero mean case]\label{th:lindeberg}
Let $\zeta = (\zeta_i)_{i \in \bbT}$ and $\xi = (\xi_i)_{i \in \bbT}$ be
two families of independent
random variables, with zero mean and unit variance.
Let $\Psi(x)$ be a multi-linear polynomial as in \eqref{eq:Psi},
with $\sfC_\Psi = \sum_{I \in \cP^\fin(\bbT)} \psi(I)^2 < \infty$.
Then the polynomial chaos expansions $\Psi(\zeta)$, $\Psi(\xi)$ are well-defined
$L^2$ random variables.

Defining for $M\in [0,\infty]$ the
maximal truncated moments
\begin{equation}\label{eq:mtrunc}
m_2^{> M} := \sup_{X \in \bigcup_{i\in \bbT} \{ \zeta_i, \xi_i\}} \bbE[X^2 \ind_{|X|> M}] \,,
\qquad \
m_3^{\leq M} := \sup_{X \in \bigcup_{i\in \bbT} \{ \zeta_i, \xi_i\}} \bbE[|X|^3 \ind_{|X|\leq M}],
\end{equation}
the following relation holds: for every $f:\R \to \R$ of class $\mathscr C^3$ with
\begin{equation} \label{eq:normPsi2}
	C_f := \max\{\|f'\|_\infty, \|f''\|_\infty, \|f'''\|_\infty\} <\infty,
\end{equation}
for every $\ell \in \N$, and for every $M\in (0,\infty]$
large enough such that $m^{>M}_{2} \le \frac{1}{4}$, one has
\begin{equation}\label{eq:lindeberg}
\begin{aligned}
	\big| \bbE\big[ f(\Psi(\zeta)) \big] - \bbE\big[ f(\Psi(\xi)) \big] \big|
	\le \ C_f \bigg\{ & 2 \, \sqrt{\sfC_{\Psi^{>\ell}}} \ + \
	\sfC_{\Psi^{\le \ell}} \, 16 \ell^2 \, m_{2}^{>M} \\
	& + \
	\sfC_{\Psi^{\le \ell}} \, 70^{\ell+1} \, \big(m^{\le M}_{3}\big)^{\ell}
	\sqrt{\max_{i \in \bbT} \big( \Inf_i\big[ \Psi^{\le \ell} \big] \big)} \bigg\},
\end{aligned}
\end{equation}
where $\sfC_\cdot$, $\Inf_i[\cdot]$ and $\Psi^{\le \ell}$, $\Psi^{> \ell}$ are defined in
\eqref{eq:CPsi}, \eqref{eq:Inf} and \eqref{eq:Psid}.
\end{theorem}

Intuitively, this theorem shows that $\Psi(\zeta)$ and $\Psi(\xi)$
are close in distribution when the right hand side of \eqref{eq:lindeberg} is small.
Despite its technical appearance, each of the three terms
inside the brackets can be easily controlled:
\begin{itemize}
\item The first term is controlled by $\sfC_{\Psi^{>\ell}}
= \sum_{|I| > \ell} \psi(I)^2$, which is small for $\ell$ large.

\item The second term is controlled by $m_{2}^{>M}$, which is small for $M$ large
if the random variables $(\zeta_i)_{i \in \bbT}$ and $(\xi_i)_{i \in \bbT}$
have \emph{uniformly integrable squares} (e.g., if they are i.i.d.).

\item The third term is controlled by the \emph{maximal influence}
$\max_{i \in \bbT} \Inf_i\big[ \Psi^{\le \ell} \big]$,
which is small if the multi-linear polynomial $\Psi^{\le \ell}$ is sufficiently ``spread-out''.
\end{itemize}
In particular, we shall see that the conditions of Theorem~\ref{th:general}
allow us to exploit \eqref{eq:lindeberg}.

\begin{remark}\label{MOOremark}\rm
When the polynomial
$\Psi = \Psi^{\le \ell}$ has degree $\ell$
and the random variables $\zeta_i$, $\xi_i$
have third absolute moments bounded by $m_3 < \infty$,
relation \eqref{eq:lindeberg} for $M = \infty$ reduces to
\begin{equation*}
	\big| \bbE\big[ f(\Psi(\zeta)) \big] - \bbE\big[ f(\Psi(\xi)) \big] \big|
	\le C_f \, \sfC_\Psi\,	
	70^{\ell+1} \, (m_3)^\ell \, \max_{i \in \bbT} \sqrt{\Inf_i[\Psi]} \,.
\end{equation*}
This is the key estimate proved by Mossel, O'Donnell and Oleszkiewicz
in~\cite{MOO10}, see Theorem 3.18 under hypothesis $\boldsymbol{H2}$,
with the prefactor $70^{\ell +1}$ instead of $30^\ell$.
Our Theorem~\ref{th:lindeberg} thus provides an extension
of \cite[Theorem 3.18]{MOO10} to finite second-moment assumptions.

Some of the results in \cite{MOO10} are formulated in the more
general setting of multi-linear polynomials over \emph{orthornormal ensembles}.
Although we stick for simplicity to the case of independent random variables,
our approach can be adapted to deal with orthonormal ensembles. In
fact, we follow the same line of proof of \cite{MOO10},
which is based on Lindeberg's original approach,
with two refinements: a sharper approximation of the remainder
in Taylor's expansion and a fine truncation on
the random variables, cf.\ Section~\ref{S:lindeberg} for details.
\end{remark}

As a corollary to Theorem~\ref{th:lindeberg}, we can treat the case where we add
non-zero mean to the random variables $(\zeta_i)_{i\in \bbT}$ and $(\xi_i)_{i\in \bbT}$.
The following result is also proved in Section~\ref{S:lindeberg}.

\begin{theorem}[Lindeberg principle, non-zero mean case]\label{C:lindeberg2}
Let $\zeta = (\zeta_i)_{i \in \bbT}$ and $\xi = (\xi_i)_{i \in \bbT}$
be as in Theorem~\ref{th:lindeberg},
and define the
maximal truncated moments $m_2^{>M}$, $m_3^{\le M}$ by \eqref{eq:mtrunc}.
Let $\mu:=(\mu_i)_{i\in \bbT}$ be a family of real numbers with
\begin{equation}\label{eq:Cmu}
	\sfc_\mu := \sum_{i\in \bbT}\mu_i^2<\infty \,,
\end{equation}
and define the $\mu$-biased families
$\tilde \zeta:=\zeta+\mu=(\zeta_i+\mu_i)_{i\in \bbT}$ and
$\tilde \xi := \xi+\mu =(\xi_i+\mu_i)_{i\in \bbT}$.

Let $\Psi(x)$ be a multi-linear polynomial as in \eqref{eq:Psi}. Setting for $\epsilon > 0$
\begin{equation}\label{eq:Psieps}
\Psi^{(\epsilon)}(x) = \sum_{I\in \cP^\fin(\bbT)} (1+\epsilon)^{|I|/2} \psi(I) x^I ,
\end{equation}
assume that $\sfC_{\Psi^{(\epsilon)}} = \sum_{I \in \cP^\fin(\bbT)} (1+\epsilon)^{|I|}
\psi(I)^2 < \infty$
for some $\epsilon > 0$. Then the polynomial chaos expansions
$\Psi(\tilde\zeta)$ and $\Psi(\tilde\xi)$ are well-defined
$L^2$ random variables.

For every $f:\R \to \R$ of class $\mathscr C^3$ with $C_f < \infty$, cf. \eqref{eq:normPsi2},
for every $\ell\in \N$ and
for every $M\in [0,\infty]$ large enough such that $m^{>M}_{2} \le \frac{1}{4}$,
the following relation holds:
\begin{equation}\label{eq:lindeberg0}
\begin{aligned}
	\big| \bbE\big[ f(\Psi(\tilde\zeta)) \big] - \bbE\big[ f(\Psi(\tilde\xi)) \big] \big|
	\le & \ e^{2 \sfc_\mu/\epsilon} \, C_f \bigg\{ 2 \, \sqrt{\sfC_{\Psi^{(\epsilon), >\ell}}} \ + \
	\sfC_{\Psi^{(\epsilon), \le \ell}} \, 16 \ell^2 \, m_{2}^{>M} \\
	& \quad + \
	\sfC_{\Psi^{(\epsilon), \le \ell}} \, 70^{\ell+1} \, \big(m^{\le M}_{3} \big)^\ell
	\sqrt{\max_{i \in \bbT} \big( \Inf_i\big[ \Psi^{(\epsilon), \le \ell} \big] \big)} \bigg\},
\end{aligned}
\end{equation}
where $\sfC_\cdot$, $\Inf_i[\cdot]$ are defined in
\eqref{eq:CPsi}, \eqref{eq:Inf} and
$\Psi^{(\epsilon), >\ell}$, $\Psi^{(\epsilon), \le \ell}$ are defined
as in \eqref{eq:Psid}.
\end{theorem}

\section{Scaling limits of disordered systems}
\label{sec:scaling}

In this section, which can be read independently of Sections~\ref{sec:intro}
and~\ref{ConvergencePoly}, we consider three much-studied statistical mechanics models: the
\emph{disordered pinning model} (Subsection~\ref{sec:pin}),
the (long-range) \emph{directed polymer model} in dimension $1+1$ (Subsection~\ref{sec:dpre}), and
the two-dimensional \emph{random field Ising model} (Subsection~\ref{sec:Ising}).
For each model, we show that the partition function has a non-trivial limit in distribution,
in the continuum and weak disorder regime, given by an explicit Wiener chaos expansion
with respect to the white noise on $\R^d$ (see Subsection~\ref{sec:wnn}
for some reminders).
The proofs, given in Sections~\ref{sec:pinproof}, \ref{sec:dpreproof}
and~\ref{sec:Isingproof}, are based on the general convergence results
of Section~\ref{ConvergencePoly} (cf. Theorem~\ref{th:general}).

For each model, the disorder will be given by a countable family of i.i.d.\ random variables
$\omega_i$ with zero mean, finite variance and locally finite exponential moments:
\begin{equation}\label{eq:disasspin}
\bbE[\omega_i]=0,\qquad \bbvar (\omega_i)=1,
\qquad \exists t_0 > 0: \ \  \Lambda(t):=\log \bbE[e^{t \omega_i} ]<\infty
\ \ \text{for $|t|<t_0$} \,.
\end{equation}
Our approach actually works in the much more general setting when disorder
is given by a triangular array of independent (but not necessarily identically
distributed) random variables, in the spirit of Theorem~\ref{th:general},
but we stick to the i.i.d. case for the sake of simplicity.

\subsection{Disordered pinning model}
\label{sec:pin}

Consider a discrete renewal process $\tau:=\{\tau_n\}_{n\geq 0}$,
that is $\tau_0 = 0$ and the increments $\{\tau_n-\tau_{n-1}\}_{n\ge 1}$
are i.i.d. $\N$-valued random variables.
We assume that $\tau$ is \emph{non-terminating}, that is $\p(\tau_1 < \infty) = 1$, and that
\begin{equation}\label{eq:taualpha}
	\P(\tau_1=n)=\frac{L(n)}{n^{1+\alpha}}
	\,, \qquad \forall n \in \N,
\end{equation}
where $\alpha\in [0,+\infty)$ and $L: (0,\infty) \to (0,\infty)$ is a slowly varying function
\cite{BGT87}. One could also consider the periodic case,
when \eqref{eq:taualpha} holds for $n \in p\N$ and $\P(\tau_1=n)=0$ if $n\not\in p\N$,
for some period $p \in \N$. For simplicity,
we focus on the aperiodic case $p=1$.

Let $\omega = (\omega_n)_{n\in \bbN_0}$ be a sequence of i.i.d. random variables,
independent of $\tau$, satisfying \eqref{eq:disasspin}.
The \emph{disordered pinning model} is the random probability
law $\P^{\omega}_{N,\beta,h}$
on subsets on $\bbN_0$, indexed by $\omega$ and by
$N \in \N$, $\beta \ge 0$ and $h\in\R$, defined by
\begin{eqnarray*}
\dd \P^{\omega}_{N,\beta,h}(\tau) :=\frac{1}{Z^{\omega}_{N,\beta,h}}
e^{\sum_{n=1}^N (\beta \omega_n
	-\Lambda(\beta) + h) \ind_{\{n\in\tau\}}} \dd \P(\tau) \,,
\end{eqnarray*}
where we recall that $\Lambda(\beta) := \log \bbE[e^{\beta \omega_1}]$, and
the \emph{partition function} $Z_{N,\gb,h}^{\go}$ is defined by
\begin{equation}\label{eq:Zpin}
	Z_{N,\gb,h}^{\go}:=\E\left[ e^{\sum_{n=1}^N (\beta \omega_n
	-\Lambda(\beta) + h) \ind_{\{n\in\tau\}}} \right] \,.
\end{equation}
We also consider the \emph{conditioned partition function}
\begin{equation}\label{eq:Zpinc}
	Z_{N,\gb,h}^{\go,c}:=\E\left[ e^{\sum_{n=1}^N (\beta \omega_n
	-\Lambda(\beta) + h) \ind_{\{n\in\tau\}}}\,\Big| N\in\tau \right].
\end{equation}

The disordered pinning model exhibits an interesting localization/delocalization phase transition.
This can be quantified via the (quenched) free energy, which is defined as
\begin{eqnarray}\label{discrete fe}
F (\beta,h):=\lim_{N\to \infty} \frac{1}{N}\log Z_{N,\gb,h}^{\go}
= \lim_{N\to \infty} \frac{1}{N}  \bbE \big[ \log Z_{N,\gb,h}^{\go} \big] \,,
\qquad \bbP(\dd\omega)\text{--a.s.}\,.
\end{eqnarray}
By restricting the partition function to configurations such that $\tau\cap [1,N] =\emptyset$,
it is easily seen that $F(\beta,h)\geq 0$. The
localized and delocalized regimes ($\mathcal{L}, \mathcal{D}$ respectively) can be defined as
\begin{eqnarray*}
\mathcal{L}:=\{ (\beta,h) \colon F(\beta,h)>0\},\quad
\mathcal{D}:=\{ (\beta,h) \colon F(\beta,h)=0\}.
\end{eqnarray*}
We refer to \cite{G10} for more information on the structure of the phase transition
and, in particular, for quantitative estimates on the \emph{critical curve}
\begin{equation}\label{eq:ccpin}
	h_c(\beta):= \sup \{h \in \R\colon \ F(\beta,h) =0\}
	= \inf \{h \in \R\colon \ F(\beta,h) > 0\} \,.
\end{equation}

\smallskip

We can now state our main result on the disordered pinning model,
to be proved in Section~\ref{sec:pinproof}.
To lighten notation, we write $Z_{Nt,\beta,h}^{\omega}$ to mean
$Z_{\lfloor Nt \rfloor,\beta,h}^{\omega}$.

\begin{theorem}[Scaling limit of disordered pinning models]\label{pinning scaling}
Let the aperiodic renewal process $\tau$
either satisfy \eqref{eq:taualpha} for some $\alpha \in (\frac{1}{2}, 1)$, or have
finite mean $\E[\tau_1]<\infty$ (which happens, in particular, when
\eqref{eq:taualpha} holds with $\alpha > 1$).
For $N\in\N$, $\hat\beta > 0$ and $\hat h\in\R$, set
\begin{equation} \label{eq:scalingbetah}
   \begin{split}
    \beta_N = \begin{cases}
    \displaystyle
    \hbeta \frac{L(N)}{N^{\alpha - 1/2}} & \text{if } \ \frac{1}{2} < \alpha < 1 \\
    \displaystyle\rule{0em}{1.8em}\hbeta \frac{1}{\sqrt{N}} & \text{if } \ \E[\tau_1]<\infty
    \end{cases} \,,
    \qquad h_N =
    \begin{cases}
    \displaystyle \hh \frac{L(N)}{N^{\alpha}} & \text{if } \ \frac{1}{2} < \alpha < 1 \\
	\displaystyle\rule{0em}{1.8em}\displaystyle \hh \frac{1}{N} & \text{if } \
	\E[\tau_1]<\infty
    \end{cases} \,.
  \end{split}
\end{equation}
Then, for every $t \ge 0$, the conditioned partition function
$Z_{Nt,\beta_N,h_N}^{\omega, c}$
of the disordered pinning model converges in distribution as $N\to\infty$
to the random variable $\bZ_{t,\hat\beta,\hat h}^{W, c}$ given by
\begin{equation} \label{eq:Zlimpin}
\bZ_{t,\hat\beta,\hat h}^{W, c} :=
1+\sum_{k=1}^\infty \frac{1}{k!}
\idotsint\limits_{[0,t]^k}  \bpsi^{c}_t(t_1,\ldots,t_k) \,
\prod_{i=1}^k \big( \hbeta\,W(\dd t_i) + \hh\, \dd t_i \big) \,,
\end{equation}
where $W(\cdot)$ denotes white noise on $\R$
and $\bpsi_t^{c} (t_1,\ldots,t_k)$ is a symmetric function, defined for
$0<t_1<\cdots<t_k < t$ by
\begin{equation} \label{eq:cokec}
\bpsi^{c}_{t}(t_1,\ldots,t_k) =
\begin{cases}
\displaystyle\frac{C_\alpha^k \,t^{1-\alpha}}{t_1^{1-\alpha} (t_2-t_1)^{1-\alpha}
\cdots (t_k-t_{k-1})^{1-\alpha}(t-t_{k})^{1-\alpha}} \qquad & \text{if } \ \frac{1}{2} < \alpha < 1 \\
\displaystyle\rule{0pt}{1.8em}
\frac{1}{\e[\tau_1]^k}  \qquad\qquad\qquad\qquad & \mbox{if } \ \E[\tau_1]<\infty
\end{cases} \,,
\end{equation}
where $C_\alpha := \frac{\alpha \sin(\pi\alpha)}{\pi}$.
The series in \eqref{eq:Zlimpin} converges in $L^2$, and
one has the convergence of the corresponding second moments:
$\bbE[(Z_{Nt,\beta_N,h_N}^{\omega, c})^2]
\to \bbE[(\bZ_{t,\hat\beta,\hat h}^{W, c})^2]$ as $N\to\infty$.

An analogous statement holds for the free (unconditioned)
partition function $Z_{Nt,\beta_N,h_N}^{\omega}$, where
the limiting random
variable $\bZ_{t,\hat\beta,\hat h}^{W}$ is defined as in \eqref{eq:Zlimpin},
with kernel
\begin{equation} \label{eq:coke}
\bpsi(t_1,\ldots,t_k) :=
\left\{
\begin{aligned}
\frac{C_\alpha ^k }{ t_1^{1-\alpha} (t_2-t_1)^{1-\alpha}\cdots (t_k-t_{k-1})^{1-\alpha}}
\qquad & \text{if } \ \frac{1}{2} < \alpha < 1 \\
\frac{1}{\e[\tau_1]^k}  \qquad\qquad\qquad\qquad & \mbox{if  } \ \E[\tau_1]<\infty
\end{aligned}
\right. \,.
\end{equation}

When $\E[\tau_1]<\infty$, for both the free and conditioned case,
the continuum partition function has an explicit distribution: for every $t \ge 0$
\begin{equation}\label{pinWCag1}
\bZ_{t,\hat\beta,\hat h}^{W, c}
\overset{d}{=} \bZ_{t,\hat\beta,\hat h}^{W} \overset{d}{=}
\exp\bigg\{\frac{\hbeta}{\e[\tau_1]}W_t+ \bigg(\frac{\hh}{\e[\tau_1]} - \frac{\hbeta^2}{2\e[\tau_1]^2}\bigg)t\bigg\} \,,
\end{equation}
where $W = (W_t)_{t\ge 0}$ denotes a standard Brownian motion.
\end{theorem}

\begin{remark}\rm\label{rem:wi}
The stochastic integrals in \eqref{eq:Zlimpin} can be rewritten more directly
as follows: denoting by $W = (W_t)_{t\ge 0}$ a standard Brownian motion, we have
\begin{equation} \label{eq:wi}
\bZ_{t,\hat\beta,\hat h}^{W} :=
1+\sum_{k=1}^\infty \ \ \idotsint\limits_{0<t_1<\cdots<t_k<t}  \bpsi_{t}(t_1,\ldots,t_k) \,
\prod_{i=1}^k
\big( \hbeta\,\dd W_{t_i} + \hh\, \dd t_i \big) \,,
\end{equation}
where the integrals can be viewed as ordinary Ito integrals: it suffices to
first integrate over $(\hbeta \dd W_{t_1} + \hh \dd t_1)$ for $t_1 \in (0,t_2)$,
then over $(\hbeta \dd W_{t_2} + \hh \dd t_2)$ for $t_2 \in (0,t_3)$, etc.
\end{remark}

\begin{remark}\rm
Theorem~\ref{pinning scaling} extends readily
to the convergence of the joint distribution
of a finite collection of partition functions (conditioned or free).
Analogously, the two parameter family of partition functions
\begin{equation}\label{eq:two-para}
	 Z_{\beta_N,h_N}^{\omega, c}(Ns,Nt):= \E\left[ e^{\sum_{n=Ns+1}^{Nt-1}
	 (\beta \omega_n
	-\Lambda(\beta) + h) \ind_{\{n\in\tau\}}}\,\Big| Ns,\,Nt\in\tau \right] \,, \quad
	\text{for } 0<s<t \,,
\end{equation}
converges in finite-dimensional distributions to
a two-parameter process $(\bZ^{W,c}_{\hat\beta, \hat h}(s,t))_{0<s<t}$.
In \cite{CSZ14} we upgrade this result
to a convergence in distribution in
the space of continuous functions, equipped with the uniform topology.
This allows to construct the continuum limit of the disordered pinning measure
$\P^{\omega}_{N,\beta,h}$.
\end{remark}

\begin{remark}\label{rem:pinfe}\rm
It is natural to call
the random variable $\bZ_{t,\hat\beta,\hat{h}}^{W}$
in Theorem~\ref{pinning scaling}
the \emph{continuum partition function} and to define the corresponding \emph{continuum free energy}
\begin{equation} \label{eq:cfepin}
	\bF^{(\alpha)}(\hbeta, \hh) = \lim_{t \to \infty}
	\frac{1}{t} \, \bbE\Big[ \log \bZ_{t,\hat\beta,\hat{h}}^{W} \Big] \,,
\end{equation}
where $\bbE$ denotes expectation with respect to the white noise $W$,
provided the limit exists.
For $\frac{1}{2}<\alpha<1$, we expect that the continuum and discrete free energies are related via
 \begin{equation}\label{fepinscale}
 \bF^{(\alpha)}(\hat\beta,\hat h) = \lim_{\delta\downarrow 0} \frac{F\,
 (\delta^{\alpha-\frac{1}{2}} L(\tfrac{1}{\delta})\,\hat\beta\,,\,
 \delta^{\alpha} L(\tfrac{1}{\delta}) \,\hat h) }{\delta} \,.
 \end{equation}
 This would follow if one could interchange the limits in the formal computation
\begin{equation}\label{eq:pinexch}
	\bF^{(\alpha)}(\hbeta, \hh) = \lim_{t \to \infty}
	\frac{1}{t} \, 	\log \bZ_{t,\hat\beta,\hat{h}}^{W} \,\overset{d}{=}\,
	\lim_{t \to \infty}
	\frac{1}{t} \, \lim_{N\to\infty}
	\log Z_{tN, \beta_N, h_N}^\omega \,.
\end{equation}
where $\beta_N,h_N$ scale as in \eqref{eq:scalingbetah}.
Such an interchange of limits has been made possible in the related
copolymer model, cf. \cite{BdH97,CG10}.
Proving the validity of relation \eqref{fepinscale}
is a very interesting open problem.
Even the existence of the continuum free energy
in \eqref{eq:cfepin} ---possibly also in the $\bbP$-a.s. sense,
like for the discrete case \eqref{discrete fe}--- is a non-trivial issue.
\end{remark}

Relation \eqref{fepinscale} is appealing because of its
implication of universality: it states that
the discrete free energy $F(\beta,h)$ has a \emph{universal shape}
in the weak disorder regime $\beta,h \to 0$,
given by the continuum free energy, which depends only on the
parameter $\alpha$ and not on finer details of the renewal distribution.
Inverting the relation $\beta = \delta^{\alpha-\frac{1}{2}} L(\tfrac{1}{\delta})$,
it is possible to rewrite \eqref{fepinscale} for $\hbeta = 1$ as
\begin{equation} \label{fepinscale1}
	 \bF^{(\alpha)}(1,\hat h)=
	 \lim_{\beta \downarrow 0} \frac{F\big(\beta, \tilde L(\tfrac{1}{\beta})
	 \beta^{\frac{2\alpha}{2\alpha-1}} \hh \big)}
	 {\widehat L(\tfrac{1}{\beta}) \beta^{\frac{2}{2\alpha-1}}} \,,
\end{equation}
where $\tilde L(\cdot)$ and $\widehat L(\cdot)$ are suitable
slowly varying functions determined by $L(\cdot)$.

\medskip

{\small
Given a slowly varying function $\phi$ and $\gamma > 0$,
we define the slowly varying functions $\bar \phi_\gamma$
and $\phi^*$ by
\begin{equation*}
	\bar \phi_\gamma(x) := \frac{1}{\phi(x^{1/\gamma})} \,, \qquad \quad
	\phi^* \big( x \phi(x) \big) \sim \frac{1}{\phi(x)} \quad \text{as} \ x \to \infty \,,
\end{equation*}
where the existence of $\phi^*$ is guaranteed by \cite[Theorems~1.5.13]{BGT87}.
Then as $\beta \downarrow 0$
\begin{equation*}
	\delta^{\alpha-\frac{1}{2}} L(\tfrac{1}{\delta}) = \beta
	\quad \Longrightarrow \quad
	\frac{1}{\delta^{\alpha - \frac{1}{2}}}
	\bar L_{\alpha - \frac{1}{2}}\bigg( \frac{1}{\delta^{\alpha - \frac{1}{2}}} \bigg)
	\sim \frac{1}{\beta}
	\quad \Longrightarrow \quad
	L\bigg( \frac{1}{\delta} \bigg)
	\sim \big( \bar L_{\alpha - \frac{1}{2}} \big)^*\bigg( \frac{1}{\beta} \bigg) \,,
\end{equation*}
hence by $\delta^{\alpha-\frac{1}{2}} L(\tfrac{1}{\delta}) = \beta$ we obtain
$\delta \sim \beta^{\frac{2}{2\alpha - 1}} \widehat L(\frac{1}{\beta})$ and
$\delta^{\alpha} L(\frac{1}{\delta}) \sim
\beta^{\frac{2\alpha}{2\alpha - 1}} \tilde L(\frac{1}{\beta})$, where
\begin{equation} \label{eq:hattildeL}
	\widehat L(x) :=
	[( \bar L_{\alpha - \frac{1}{2}} )^*(x)]^{-\frac{2}{2\alpha - 1}} \,, \qquad
	\tilde L(x) :=
	[( \bar L_{\alpha - \frac{1}{2}} )^*(x)]^{-\frac{1}{2\alpha - 1}} \,.
\end{equation}
Plugging this into \eqref{fepinscale} (with $\hbeta = 1$), we get \eqref{fepinscale1}.
}

\medskip

Defining the critical curve of the continuum free energy
in analogy with \eqref{eq:ccpin}, i.e.\
\begin{equation*}
	\bh_c^{(\alpha)}(\hbeta) := \sup\{\hh \in \R: \ \bF^{(\alpha)}(\hbeta, \hh) = 0\}
	= \inf\{\hh \in \R: \ \bF^{(\alpha)}(\hbeta, \hh) > 0\} \,
\end{equation*}
relation \eqref{fepinscale1} leads us to the following

\begin{conjecture}
For any disordered pinning model satisfying \eqref{eq:taualpha}
with $\alpha \in (\frac{1}{2},1)$,
the critical curve $h_c(\beta)$
has the following universal asymptotic behavior
(defining $\tilde L(\cdot)$ by \eqref{eq:hattildeL}):
\begin{equation*}
\lim_{\beta\downarrow 0} \frac{h_c(\beta)}{\tilde L(\tfrac{1}{\beta})
\beta^{\frac{2\alpha}{2\alpha-1}}} \,=\, \bh_c^{(\alpha)}(1)\,.
\end{equation*}
\end{conjecture}

\noindent
Further support to this conjecture is provided by the fact that
(non-matching) upper and lower bounds for $h_c(\beta)$ of the order
$\tilde L(\tfrac{1}{\beta}) \beta^{\frac{2\alpha}{2\alpha-1}}$ were proved in
\cite{A08,AZ09}.

\begin{remark}\rm
The case when \eqref{eq:taualpha} holds with
$\alpha = 1$ and $\E[\tau_1] = \infty$, i.e.\ $\sum_{n\in\N} L(n)/n = \infty$,
can also be included in Theorem~\ref{pinning scaling} (we have omitted
it for notational lightness), setting
\begin{equation*}
	\beta_N = \hbeta \frac{\ell(N)}{\sqrt{N}}\,, \qquad h_N = \hh \frac{\ell(N)}{N} \,,
	\qquad \text{where} \qquad \ell(N) := \sum_{n=1}^N \frac{L(n)}{n}
\end{equation*}
is a slowly varying function,
and with $\bpsi^{c}_t(t_1,\ldots, t_k) =
\bpsi(t_1,\ldots, t_k) \equiv 1$.
This is easily checked from the proof in Section~\ref{sec:pinproof},
because $\P(n \in \tau) \sim \frac{1}{\ell(n)}$, cf. \cite[Theorem 8.7.5]{BGT87}.

On the other hand, the case $\alpha=1/2$ appears to be fundamentally different,
because the continuum kernels $\bpsi_t^{c}$, $\bpsi$ are no longer
$L^2$ integrable and therefore the stochastic integrals are not properly
defined. When $\alpha=1/2$ and $\sum_{n\in\N}1/(n \, L(n)^2)=\infty$ ---in particular,
when $L(n) \sim (const.)$ as $n\to\infty$, as for the simple random walk on $\Z$---
we expect that a nontrivial continuum limit should exist.
This appears to be a challenging open problem.
\end{remark}

\subsection{Directed polymer model}
\label{sec:dpre}

Consider a random walk $S = (S_n)_{n\in\N_0}$ on $\Z$, with law $\P$.
Let $\omega = (\omega(n,x))_{n\in\N, x \in \Z}$ be a family of i.i.d.
random variables, independent of $S$, with zero mean, unit variance
and locally finite exponential moments, cf.\ \eqref{eq:disasspin}.
The $(1+1)$-dimensional \emph{directed polymer model} is the
random probability law $\P_{N,\beta}^{\omega}$ for the walk $S$ defined
for $N \in \N$ and $\beta \ge 0$ by
\begin{equation}\label{eq:Pdirpol}
	\dd \P_{N,\beta}^{\omega}(S) :=
	\frac{1}{Z_{N,\beta}^{\omega}} \, e^{\sum_{n=1}^N (\beta \omega(n,S_n)
	- \Lambda(\beta) )}
	\, \dd \P(S) \,,
\end{equation}
where we recall that $\Lambda(\beta) := \log \bbE[e^{\beta \omega_{n,x}}]$ and
the \emph{partition function} $Z_{N,\beta}^{\omega}$ is defined by
\begin{equation} \label{eq:Zdirpol}
	Z_{N,\beta}^{\omega} = \E \big[ e^{\beta \sum_{n=1}^N \omega(n,S_n) } \big]
	e^{-\Lambda(\beta) N} \,.
\end{equation}
For $y\in\Z$, we also define the
\emph{constrained} point-to-point partition function $Z_{N,\beta}^{\omega}(y)$
and the \emph{conditioned} point-to-point partition function $Z_{N,\beta}^{\omega,c}(y)$, setting
\begin{equation} \label{eq:Zpoint2point}
\begin{split}
	Z_{N,\beta}^{\omega}(y) & = \E \big[ e^{\beta \sum_{n=1}^N \omega(n,S_n) }
	\, \ind_{\{S_N=y\}} \big] e^{-\Lambda(\beta) N} \,, \\
	Z_{N,\beta}^{\omega,c}(y) & = \E \big[ e^{\beta \sum_{n=1}^N \omega(n,S_n) }
	\, \big| \, S_N=y \big]
	e^{-\Lambda(\beta) N} \,.
\end{split}
\end{equation}
Plainly, $Z_{N,\beta}^{\omega} =\sum_{y\in \bbZ} Z_{N,\beta}^{\omega}(y)
= \sum_{y\in \bbZ} Z_{N,\beta}^{\omega,c}(y)\,\P(S_N=y)$.

When $(S_n)_{n\geq 1}$ is the simple symmetric random walk on $\Z$,
we have the much-studied \emph{Directed Polymer in Random Medium}.
First introduced in the physics literature in \cite{HH85}, this model
has received particular attention due to its connection to the
Kardar-Parisi-Zhang (KPZ) equation and its universality class,
cf.\ \cite{CSY04}, \cite{C12} for a review. In particular, the
point-to-point partition function
can be thought of as an approximation of the solution of the
Stochastic Heat Equation (SHE), whose logarithm is the so-called Hopf-Cole solution of
the KPZ equation. This was made rigorous in \cite{AKQ14a},
showing that when $\beta = \beta_N$ is scaled as $N^{-1/4}$
(the so-called intermediate disorder regime) and $y$ is scaled as $N^{-1/2}$,
the point-to-point partition function $Z_{Nt,\beta}^{\omega}(y)$
converges in distribution to a continuum process, which solves the SHE.

Our approach allows to extend the results in \cite{AKQ14a}.
Not only can we deal with general zero mean, finite
variance random walks,
which are the natural ``universality class'' of the simple symmetric random walk on $\Z$;
we can also consider random walks attracted to stable
laws with index $\alpha \in (1,2)$, exploring new universality classes.
When allowing ``big jumps'', it is natural to call
$\P_{N,\beta}^{\omega}$ the \emph{long-range} directed polymer model.

\smallskip

Let us now state precisely our assumptions on the random walk.

\begin{assumption} \label{ass:rwdp}
Let $S = (S_n)_{n\ge 0}$ be a
random walk on $\Z$, with $S_0=0$ and with i.i.d. increments $(S_n - S_{n-1})_{n\ge 1}$,
such that for some $\alpha \in (1,2]$ the following holds:
\begin{itemize}
\item $(\text{\emph{Case }}\alpha = 2)$: $\E[S_1]=0$ and $\sigma^2 := \var(S_1) < \infty$;

\item $(\text{\emph{Case }}1 < \alpha < 2)$: $\E[S_1]=0$ and
there exist $\gamma \in [-1,1]$ and $C \in (0,\infty)$ such that
\begin{equation} \label{eq:stablerw}
	\P(S_1 > n) \sim \big( C\tfrac{1+\gamma}{2} \big) \frac{1}{n^\alpha} \,, \qquad
	\P(S_1 < -n) \sim \big( C\tfrac{1-\gamma}{2} \big) \frac{1}{n^\alpha} \,, \qquad
	\text{as } n\to\infty \,.
\end{equation}
\end{itemize}
\end{assumption}

This means that the random walk $S$ is in the
\emph{domain of normal attraction}
of a stable law of index $\alpha \in (1,2]$ and
(for $\alpha < 2$)  skewness parameter
$\gamma$.
(The adjective ``normal'' refers to the absence
of slowly varying functions.)
In other words, $S_n / n^{1/\alpha}$ converges in distribution
as $n\to\infty$ to a random variable $Y$,
which has law $\cN(0,\sigma^2)$ if $\alpha = 2$, while for
$1 < \alpha < 2$
\begin{equation}\label{eq:stable law}
	\E[e^{i t Y}] = e^{-c_\alpha C \, |t|^\alpha (1-i \gamma (\sign t) \tan\frac{\pi\alpha}{2})} \,,
	\qquad \text{for a suitable} \quad c_\alpha > 0 \,.
\end{equation}
We remark that $Y$ satisfies the same conditions as $S_1$ in Assumption~\ref{ass:rwdp}
and has an absolutely continuous law, with a
bounded and continuous density $g(\cdot)$. For $t > 0$ we set
\begin{equation}\label{eq:stable transition}
	g_t(x) := \frac{1}{t^{1/\alpha}} g \bigg( \frac{x}{t^{1/\alpha}} \bigg) \,.
\end{equation}
We stress that $g(\cdot)$ depends only on the parameters $(\alpha,\sigma^2)$
or $(\alpha,\gamma, C)$ in Assumption~\ref{ass:rwdp}.

\smallskip

The \emph{period} of a random walk $S$ on $\Z$ is the largest $p \in \N$
such that $\P(S_1 \in p \Z + r) = 1$, for some $r \in \{0,\ldots, p-1\}$.
For instance, the simple symmetric random walk
on $\Z$ has period $p=2$, because $\P(S_1 \in 2\Z + 1) = 1$.
To lighten notation, given $(s,y) \in \R^+ \times \R$, we
write $Z_{s,\beta}^{\omega,c}(y)$
to mean $Z_{\tilde s,\beta}^{\omega,c}(\tilde y)$,
where $\tilde s := \lfloor s \rfloor$ and
$\tilde y := \max\{z \in  p\Z + r \tilde s: \, z \le y\}$.

We are ready to state our main result for this model, to be proved in Section~\ref{sec:dpreproof}.

\begin{theorem}[Scaling limits of directed polymer models]\label{DP scaling}
Let $S$ be a random walk on $\Z$
satisfying Assumption~\ref{ass:rwdp} for some $\alpha \in (1,2]$.
For $N\in\N$ and $\hat\beta > 0$, set
\begin{equation*}
	\beta_N := \frac{\hat\beta}{N^{\frac{\alpha-1}{2\alpha}}} \,.
\end{equation*}
For every $t \ge 0$ and $x\in\R$, the conditioned point-to-point partition function
$Z_{Nt,\beta_N}^{\omega,c}(N^{1/\alpha}x)$ converges in distribution as $N\to\infty$ to
the random variable $\bZ_{t,\hat\beta}^{W,c}(x)$ given by
\begin{equation} \label{eq:ZcontDP}
	\bZ_{t,\hat\beta}^{W,c}(x) :=
	1 + \sum_{k=1}^\infty \frac{\hat\beta^k}{k!} \,
	\idotsint\limits_{ ([0,t]\times\bbR)^k}
	\bpsi_{t,x}^{c}
	\big( (t_1,x_1),\ldots,(t_k,x_k) \big) \, \prod_{i=1}^k W(\dd t_i \, \dd x_i) \,,
\end{equation}
where $W(\cdot)$ denotes space-time white noise
(i.e., white noise on $\R^2$) and the symmetric function
$\bpsi_{t,x}^{c}\big(
(t_1,x_1),\ldots,(t_k,x_k) \big)$ is defined for $0<t_1<\cdots<t_k<t$ by
\begin{equation} \label{eq:contkerDP}
	\bpsi_{t,x}^{c}\big( (t_1,x_1),\ldots,(t_k,x_k) \big) \, :=
	\Bigg( \prod_{i=1}^k \sqrt{p} \,g_{t_i-t_{i-1}}(x_i-x_{i-1})
	\Bigg) \,\frac{ g_{t-t_{k}}(x-x_{k})}{ g_t(x) } \,,
\end{equation}
where $x_0 := 0$ and $p \in \N$ is the period of the random walk.
The series in \eqref{eq:ZcontDP} converges in $L^2$
and furthermore
$\bbE[(Z_{Nt,\beta_N}^{\omega,c}(N^{1/\alpha}x))^2] \to
\bbE[(\bZ_{t,\hat\beta}^{W,c}(x))^2]$ as $N\to\infty$.

An analogous statement holds for the free partition function
$Z_{Nt,\beta_N}^{\omega}$, where the limiting random variable
$\bZ_{t,\hat\beta}^{W}$ is defined as in \eqref{eq:ZcontDP}, with kernel
\begin{equation*}
	\bpsi_t \big( (t_1,x_1),\ldots,(t_k,x_k) \big) \,:=
	\prod_{i=1}^k \sqrt{p} \,g_{t_i-t_{i-1}}(x_i-x_{i-1}) \,.
\end{equation*}
\end{theorem}

\begin{remark}\rm \label{rem:remdpre}
Theorem~\ref{DP scaling} extends to the convergence of the joint distribution
of a finite collection of partition functions (conditioned or free).
In particular, the four-parameter family
$Z_{Ns,Nt,\beta_N}^{\omega,c}(N^{1/\alpha}x, N^{1/\alpha}y)$,
defined for $(s,x)$ and $(t,y)$ in $\R^+ \times \R$ by
\begin{equation*}
	Z_{Ns,Nt;\beta_N}^{\omega,c}(N^{1/\alpha}x, N^{1/\alpha}y)
	:= \E \big[ e^{\sum_{n= Ns +1}^{Nt -1} (\beta \omega(n,S_n) - \Lambda(\beta)) } \,
	 \big| \, S_{Ns}=N^{1/\alpha}x,\,S_{Nt}=N^{1/\alpha}y \big] \,,
\end{equation*}
converges in finite-dimensional distributions
to a four-parameter family of continuum conditioned partition functions
$\{\bZ^{W,c}_{s,t;\hat\beta}(x,y)\}_{(s,x), (t,y)\in \R^+ \times \bbR}$.

Similar to the case $\alpha=2$ studied in~\cite{AKQ14b}, we expect that this
convergence can be upgraded to a convergence in distribution
in the space of continuous functions, equipped with the uniform topology. We can then use these partition
functions to define a \emph{continuum long-range directed polymer model} (which corresponds intuitively
to an ``$\alpha$-stable L\'evy process in a white noise random medium''),
by specifying its finite-dimensional distributions as done in \cite{AKQ14b} for the
Brownian case $\alpha=2$.
\end{remark}

\begin{remark}\rm
The free energy of the (discrete) directed polymer model is defined by
$$
F(\beta):=\lim_{N\to\infty}\frac{1}{N} \bbE \big[ \log Z^\omega_{N,\beta} \big] \,,
$$
where we expect that the limit exists (also $\bbP$-a.s. and in $L^1(\dd\bbP)$),
as in the usual setting.
It is natural to define the free energy of the continuum model analogously, i.e.\
$$
	\bF(\hat\beta):=\lim_{t\to\infty}\frac{1}{t}
	\bbE \big[\log \bZ^W_{t,\hat\beta} \big] \,,
$$
assuming of course that the limit exists.
We stress that $\bF(\cdot)$ is a \emph{universal}
quantity, which depends only on the parameters $(\alpha,\sigma^2)$
or $(\alpha,\gamma, C)$ in Assumption~\ref{ass:rwdp}
(furthermore, the parameters $\sigma^2$, $C$ enter as simple
scale factors).
We also note that $\bF(\hat\beta)
= \bF(1) \hat\beta^{\frac{\alpha - 1}{2\alpha}}$,
by an easy scaling argument.
In analogy with Remark~\ref{rem:pinfe}, we are led to the following

\begin{conjecture}
For any directed polymer model
satisfying Assumption~\ref{ass:rwdp},
the free energy $F(\beta)$ has the following universal asymptotic behavior
for weak disorder:
\begin{equation*}
\lim_{\beta\downarrow 0} \frac{F\,(\beta) }{\beta^{\frac{2\alpha}{\alpha-1}}} \,=\,
\lim_{\delta\downarrow 0} \frac{F\,(\delta^{\frac{\alpha-1}{2\alpha}})}{\delta}
\,=\,  \bF(1).
 \end{equation*}
 \end{conjecture}

\noindent
When $\alpha=2$ we would then get
$F(\beta) \sim \bF(1) \beta^4$, which is supported by the
(non-matching) upper and lower bounds on $F(\beta)$ obtained in \cite{L10}.
\end{remark}

\begin{remark}\rm
For $1 < \alpha < 2$,
the function $g_t(\cdot)$ in \eqref{eq:stable transition} is the marginal density
of the $\alpha$-stable L\'evy process $(X_t)_{t \in \R^+}$ whose infinitesimal generator
is given by a multiple of
\begin{equation}
\Delta^{\alpha/2, \gamma}f(x)
:=\int_{-\infty}^{+\infty}
\big( f(x+y)-f(x) -yf'(x) \big) \Big( \frac{1+\gamma}{|y|^{1+\alpha}}
\ind_{\{y>0\}} + \frac{1-\gamma}{|y|^{1+\alpha}} \ind_{\{y<0\}}\Big)
\, \dd y \,.
\end{equation}
In the symmetric case $\gamma=0$, this reduces to the much studied \emph{fractional Laplacian}
\begin{equation}\label{fLaplace}
\Delta^{\alpha/2}f(x) := \int_{-\infty}^\infty
\big( f(x+y)-f(x)- y f'(x) \big) \frac{1}{|y|^{1+\alpha}} \, \dd y \,.
\end{equation}

Let us stick for simplicity to the symmetric case $\gamma = 0$.
It is natural to call
$\bZ^{W,c}_{t,\hat\beta}(x)$ in \eqref{eq:ZcontDP} the continuum \emph{conditioned} point-to-point
partition function.
Introducing the continuum \emph{constrained}
point-to-point partition function by
\begin{equation*}
	\bZ^{W}_{t,\hat\beta}(x) := \bZ^{W,c}_{t,\hat\beta}(x) \, g_t(x) \,,
\end{equation*}
one can check that the process $u(t,x) = \bZ^{W}_{t,\hat\beta}(x)$ has a version that
is continuous in $(t,x)$ and, up to a scaling factor, it is a mild solution to the one dimensional
stochastic PDE
\begin{equation}\label{eq:SFHE}
\begin{cases}
\partial_t  u =\Delta^{\alpha/2} \, u +\sqrt{p}\hat\beta \, \dot W\,u & \\
u(0, \cdot ) =\delta_0(\cdot) &
\end{cases} \,,
\end{equation}
which we can call the {\it stochastic fractional heat equation} (SFHE),
generalizing the usual SHE (which corresponds to $\alpha = 2$).
Uniqueness of mild solutions for the SFHE follows from standard techniques,
see discussions in~\cite{CJKS14} and references therein.

Let us also consider
the process $A_{\alpha,\hat\beta}(\cdot):=\log \bZ_{1,\hat\beta}^{W}(\cdot)$.
When $\alpha = 2$, this is the \emph{cross-over
process} studied in \cite{ACQ11,SS10}, which owes its name to the fact that
its one-point marginals interpolate between
the Gaussian distribution (in the limit $\hat\beta\to 0$) and the Tracy-Widom GUE
distribution (in the limit $\hat\beta\to\infty$).
When $\alpha < 2$, it is
easy to see that $A_{\alpha,\hat\beta}(\cdot)$ is again
asymptotically Gaussian for $\hat\beta\to 0$
(the contribution from the first stochastic integral in its Wiener chaos expansion is dominant
over the iterated integrals, which are multiplied by higher powers of $\hat\beta$).
However, it is far from obvious whether $A_{\alpha,\hat\beta}(\cdot)$
converges to some asymptotic process $A_{\alpha,\infty}(\cdot)$ as $\hat\beta\to\infty$
and whether such a process describes some universality class
for long-range random polymers, last passage percolation and growth models,
generalizing the the so-called Airy process obtained for $\alpha = 2$.
Besides a very recent work on the limit shapes of long-range first-passage
percolation model \cite{CD13},
long-range polymer type models do not appear to have been studied systematically before.
\end{remark}

\subsection{Random field Ising model}
\label{sec:Ising}

Given a bounded $\bbOmega \subseteq \Z^2$, we set
$$
\partial \bbOmega := \{x \in \Z^2\backslash \bbOmega : \Vert x-y\Vert =1 \mbox{ for some }
y\in\bbOmega\}.
$$
For a fixed parameter
$\beta \geq 0$, representing the ``inverse temperature'',
the \emph{Ising model on $\bbOmega$ with $+$ boundary condition (and zero external field)}
is the probability measure $\P_{\bbOmega}^{+}$ on the set of spin configurations
$\{\pm1 \}^\bbOmega$, where each
$\sigma := (\sigma_x)_{x\in \bbOmega} \in \{\pm 1\}^\bbOmega$ has probability
\begin{equation}\label{eq:Isingbasic}
	\P_{\bbOmega}^{+} (\sigma) := \frac{1}{Z_{\bbOmega}^+} \,
	\exp\Bigg\{ \beta \sum_{x\sim y \in \bbOmega \cup \partial \bbOmega }
	\sigma_x \sigma_y \Bigg\} \, \prod_{x\in \partial \bbOmega} \ind_{\{\sigma_x=+1\}}.
\end{equation}
Here $x\sim y$ denotes an unordered nearest-neighbor pair in $\Z^2$,
and $Z_{\bbOmega}^+$ is the normalizing constant.
The value of $\beta$ will soon be fixed, which is why we do not indicate it
in $\P_{\bbOmega}^{+}$.

\begin{remark}\rm
It is well-known that as $\bbOmega \uparrow \Z^2$,
the sequence of probability measures $\P_{\bbOmega}^{+}$ has a unique infinite volume limit
$\P_{\Z^2}^+$, which of course depends on $\beta$, such that
$$
\E_{\Z^2}^+ (\sigma_0)>0 \quad  \mbox{if and only if} \quad  \beta > \beta_c
:= \frac{1}{2}\log (1+\sqrt{2}).
$$
By an obvious symmetry, for $\beta>\beta_c$ there also exists an infinite volume Gibbs measure
$\P_{\Z^2}^-$ satisfying $\E_{\Z^2}^- (\sigma_0)<0$.
Given the coexistence of multiple infinite volume Gibbs measures,
the Ising model is said to have a first order phase transition
for $\beta>\beta_c$. The same result holds in higher dimensions, with different values of $\beta_c$.

A question that attracted significant interest
was whether this picture will be altered by the addition of a small random external field.
After long debates, this question was settled by Bricmont and Kupiainen~\cite{BK88},
who showed that the first order phase transition persists for the random field Ising model
in dimensions $d\geq 3$ at low temperatures (i.e., for large $\beta$),
and by Aizenman and Wehr~\cite{AW90} who showed the absence of first order phase transition
in dimension $2$ at any temperature. See~\cite[Chap.~7]{B06} for an overview.
\end{remark}

Henceforth we fix $\beta = \beta_c := \frac{1}{2} \log (1+\sqrt{2})$,
so that \emph{$\P_{\bbOmega}^+$ denotes the two-dimensional critical Ising model}.
Let $\omega:=(\omega_x)_{x\in \Z^2}$ be
a family of i.i.d.\ random variables satisfying \eqref{eq:disasspin},
representing the \emph{disorder}.
Given $\lambda:=(\lambda_x)_{x\in\Z^2}\geq 0$
and $h:=(h_x)_{x\in\Z^2}\in\R$, representing the disorder strength and bias respectively,
the \emph{random field Ising model} (RFIM) is
the probability measure $\P^{+,\omega}_{\bbOmega, \lambda, h}$ on $\{\pm1 \}^\bbOmega$ with
\begin{equation}\label{RFIMP}
	\P^{+,\omega}_{\bbOmega, \lambda, h}(\sigma) =  \frac{1}{Z^{+,\omega}_{\bbOmega, \lambda, h}} \,
	\exp\Bigg\{ \sum_{x\in \bbOmega} (\lambda_x \omega_x +h_x) \sigma_x \Bigg\}
	\, \P^{+}_{\bbOmega} (\sigma),
\end{equation}
where the normalizing constant, called \emph{partition function}, is given by
\begin{equation}\label{RFIMZ}
	Z^{+,\omega}_{\bbOmega, \lambda, h} =
	\E^{+}_{\bbOmega}
	\Bigg[ \exp\Bigg\{ \sum_{x\in \bbOmega} (\lambda_x \omega_x +h_x) \sigma_x \Bigg\} \Bigg].
\end{equation}
Note that we allow the disorder strength $\lambda$ and bias $h$
to vary from site to site.
Also observe that choosing $\P_{\bbOmega}^+$ as a ``reference law''
means that $Z^{+,\omega}_{\bbOmega, \lambda, h} = 1$
for $\lambda, h \equiv 0$ (with $\beta=\beta_c$).

\smallskip

Fix now a bounded open set $\Omega \subseteq \R^2$ with
piecewise smooth boundary, and define the rescaled
lattice $\bbOmegadelta := \Omega \cap (\delta\Z)^2$, for $\delta > 0$.
We are going to obtain a non-trivial limit in distribution for the partition function
$Z^{+,\omega}_{\bbOmegadelta, \lambda, h}$,
in the continuum and weak disorder regime $\delta, \lambda, h \to 0$. We build on
recent results of Chelkak, Hongler and Izyurov~\cite[Theorem~1.3]{CHI12}
on the continuum limit of the spin correlations
under $\P_{\bbOmegadelta}^{+}$
(the two-dimensional critical Ising model with + boundary condition):
for all $n\in\N$ and distinct $x_1, \ldots, x_n \in \Omega$
\begin{equation}\label{CHI1.30}
\lim_{\delta \downarrow 0} \delta^{-\frac{n}{8}}
\, \E^+_{\bbOmegadelta}[\sigma_{x_1} \cdots \sigma_{x_n}] = \cC^n
\, \bphi_\Omega^+(x_1, \ldots, x_n),
\end{equation}
where $\bphi_\Omega^+ : \cup_{n\in\N} \Omega^n \to \R$ is a symmetric function
and $\cC := 2^{\frac{5}{48}} e^{-\frac{3}{2}\zeta'(-1)}$,
with $\zeta'$ denoting the derivative of Riemann's zeta function.\footnote{In~\cite{CHI12},
the mesh size of $\bbOmegadelta$ is $\sqrt{2}\delta$ instead of $\delta$,
which accounts for the difference between our formula for $\cC$ and that in \cite[(1.3)]{CHI12}.
Exact formulas for $\bphi^+_\Omega(x_1, \ldots, x_n)$ are available for $n=1,2$
when $\Omega$ is the upper half plane $\mathbb H$,
cf.~\cite[(1.4)]{CHI12}. The general case $\bphi^+_\Omega$
can be obtained from $\bphi^+_{\mathbb H}$ via conformal map.}
(For simplicity, in \eqref{CHI1.30} we have set $\sigma_x := \sigma_{x_\delta}$, where
$x_\delta$ denotes the point in $\bbOmegadelta$ closest to $x \in \Omega$.)

We can complement \eqref{CHI1.30} with a uniform bound (see Lemma~\ref{L:correlationbd} below): there exists
$C\in (0,\infty)$ such that for all $n\in \N$, $x_1, \ldots, x_n \in \Omega$ and $\delta \in (0,1)$
\begin{equation}\label{psiOmegabd0}
0 \le \delta^{-\frac{n}{8}}
\, \E^+_{\bbOmegadelta}[\sigma_{x_1} \cdots \sigma_{x_n}]
\leq \prod_{i=1}^n \frac{C}
{d(x_i, \partial\Omega\cup \{x_1, \ldots, x_n\}\backslash\{x_i\})^{\frac{1}{8}}} \,,
\end{equation}
where $d(x,A) := \inf_{y \in A} \|x-y\|$ and the RHS is shown to be in $L^2$ in Lemma~\ref{L:fOmega} below.
Therefore the convergence \eqref{CHI1.30} also holds in $L^2(\Omega)$, and $\| \bphi_\Omega^+ \|_{L^2(\Omega^n)} < \infty$ for every $n\in\N$.

\smallskip

We can now state our main result for the RFIM, to be proved in Section~\ref{sec:Isingproof}.

\begin{theorem}[\bf Scaling limits of RFIM]\label{T:RFIM}
Let $\Omega \subseteq \R^2$ be bounded and
simply connected open set with piecewise smooth boundary, and define
$\bbOmegadelta := \Omega \cap (\delta\Z)^2$, for $\delta > 0$.

Let $\omega:=(\omega_x)_{x\in \Z^2}$ be
i.i.d.\ random variables satisfying \eqref{eq:disasspin}
and define $\omega_\delta = (\omega_{\delta,x})_{x\in\bbOmegadelta}$ by
$\omega_{\delta,x} := \omega_{x/\delta}$.
Given two continuous functions
$\hlambda: \overline{\Omega} \to (0,\infty)$ and $\hh: \overline{\Omega} \to \R$, define
\begin{equation}\label{lhdelta}
\lambda_{\delta, x} := \hlambda(x) \,\delta^{\frac{7}{8}} \,, \qquad
h_{\delta, x} := \hh(x) \,\delta^\frac{15}{8} \,.
\end{equation}
and set $\lambda_\delta = (\lambda_{\delta,x})_{x\in\bbOmegadelta}$,
$h_\delta = (h_{\delta,x})_{x\in\bbOmegadelta}$. Then the rescaled partition function
\begin{equation} \label{eq:rescaledIsingZ}
	e^{- \frac{1}{2} \| \hlambda \|_{L^2(\Omega)}^2 \delta^{-1/4}}
	\, Z^{+, \omega_\delta}_{\bbOmegadelta, \lambda_\delta, h_\delta}
\end{equation}
converges in distribution as $\delta \downarrow 0$
to a random variable $\bZ^{+, W}_{\Omega, \hlambda, \hh}$ with Wiener chaos expansion
\begin{equation}\label{Z+Wexp}
\bZ^{+, W}_{\Omega, \hlambda, \hh}
= 1 + \sum_{n=1}^\infty \frac{\cC^n}{n!} \idotsint_{\Omega^n} \bphi_\Omega^+(x_1, \ldots, x_n)
\prod_{i=1}^n \, \big[\hlambda(x_i) W({\rm d}x_i) + \hh(x_i) {\rm d}x_i\big],
\end{equation}
where $W(\cdot)$ denotes white noise on $\R^2$ and $\bphi_\Omega^+(\cdot)$, $\cC$
are as in \eqref{CHI1.30}.
\end{theorem}

\begin{remark}\rm
We impose the continuity on $\hlambda$ and $\hh$ and the strict
positivity on $\hlambda$ for technical simplicity:
these conditions can be relaxed with a more careful analysis.
\end{remark}

We call the random variable $\bZ^{+, W}_{\Omega, \hlambda, \hh}$
in \eqref{Z+Wexp} the \emph{continuum RFIM partition function}.
The fact that the continuum correlation function $\bphi_\Omega^+$ is
\emph{conformally covariant}, proved in \cite[Theorem 1.3]{CHI12}, allows to
deduce the conformal covariance of $\bZ^{+, W}_{\Omega, \hlambda, \hh}$.

\begin{corollary}[\bf Conformal covariance]\label{C:coco}
Let $\Omega,\tilde \Omega \subseteq \R^2$
be two bounded and simply connected open sets
with piecewise smooth boundaries, and let $\varphi: \tilde \Omega \to \Omega$ be conformal.
Let $\hlambda : \overline{\Omega} \to (0,\infty)$ and $\hh: \overline{\Omega}\to \R$ be continuous
functions. Then
$$
\bZ^{+, W}_{\Omega, \hlambda, \hh} \overset{d}{=}
\bZ^{+, W}_{\tilde \Omega, \tilde \lambda, \tilde h} ,
$$
where we set
$\tilde \lambda(z) := |\varphi'(z)|^{\frac{7}{8}} \, \hlambda (\varphi(z)) $
and $\tilde h(z) := |\varphi'(z)|^{\frac{15}{8}} \, \hh (\varphi(z)) $
for all $z\in \tilde\Omega$.
\end{corollary}

\begin{remark}\rm
Recently, Camia, Garban and Newman~\cite{CGN12, CGN13} showed that
for a \emph{deterministic} external field,
more precisely $\lambda_{\delta, x} \equiv 0$ and $h_{\delta, x} \equiv \hh \delta^{\frac{15}{8}}$
for fixed $\hh \in\R$, the Ising measure $\P^+_{\bbOmegadelta, \lambda_\delta, h_\delta}$
converges as $\delta \downarrow 0$ (in a suitable sense)
to a limiting distribution-valued process $\Phi^{\infty, \hh}_\Omega$.
Theorem~\ref{T:RFIM} can be regarded as a first step towards the extension
of this convergence to the case of a \emph{random} external field,
where $\bZ^{+, W}_{\Omega, \hlambda, \hh}$ plays the role of the partition function
of a {\em continuum Ising model
with random external field $(\hlambda(x) W(\dd x)+\hh(x) \dd x)$ and + boundary condition}.
The next step towards the construction of such a continuum model would be to identify
the joint distribution of the partition functions
$(\bZ^{{\rm b}, W}_{\Gamma, \hlambda, \hh})_{\Gamma, {\rm b}}$
(note that they are random variables which are functions of the disorder $W$)
for a large enough family of
sub-domains $\Gamma \subseteq \Omega$ and ``boundary conditions'' ${\rm b}$.
\end{remark}

\begin{remark}\rm
Consider the case when the disorder strength
and bias are constant, i.e.\ $\lambda_x \equiv \lambda \ge 0$
and $h_x \equiv h \in \R$.
The free energy $F(\lambda,h)$ of the
critical random field Ising model can be defined as
follows: setting $\Omega := (-\frac{1}{2},\frac{1}{2})^2$ and $\Lambda_N := (N\Omega) \cap \Z^2=
\{\lceil -\frac{N}{2} \rceil, \ldots, \lfloor \frac{N}{2}\rfloor \}^2$,
\begin{equation} \label{eq:IsingFE}
	F(\lambda,h) := \lim_{N\to\infty} \, \frac{1}{N^2|\Omega|} \,
	\bbE \big[ \log Z_{\Lambda_N,\lambda,h}^{+,\omega} \big]
	= \lim_{\delta \downarrow 0} \, \frac{\delta^2}{|\Omega|} \,
	\bbE \big[ \log Z_{\bbOmega_\delta,\lambda,h}^{+,\omega} \big] \,,
\end{equation}
where the limit exists by standard super-additivity arguments and is independent of the choice of $\Omega$.
Note that $F(0,0) = 0$, that is $F(\lambda,h)$ represents the
\emph{excess} free energy with respect to the critical Ising model, cf.\ \eqref{RFIMZ}.

It is natural to define a continuum free energy $\bF(\hlambda, \hh)$
for $\hlambda \ge 0$, $\hh \in \R$ by
\begin{equation*}
	\bF(\hlambda, \hh) := \lim_{\Omega \uparrow \R^2}
	\frac{1}{Leb(\Omega)} \bbE \Big[ \log \bZ_{\Omega,\hlambda,\hh}^{+,W} \Big] \,,
\end{equation*}
provided the limit exists (at least along sufficiently nice domains $\Omega \uparrow \R^2$),
where $\bbE$ denotes expectation with respect to the white noise $W(\cdot)$.
In analogy with Remark~\ref{rem:pinfe}, Theorem~\ref{T:RFIM}
suggests the following conjecture on the universal behavior of the free energy
$F(\lambda,h)$ in the weak disorder regime $\lambda, h \to 0$.

\begin{conjecture}The following asymptotic relation holds:
\begin{equation} \label{eq:conjIsing}
	\lim_{\delta \downarrow 0} \frac{F\big( \hlambda \delta^{\frac{7}{8}},
	\hh \delta^{\frac{15}{8}} \big)}{\delta^2} = \bF(\hlambda, \hh) \,, \qquad
	\forall \hlambda \ge 0, \ \hh \in \R \,.
\end{equation}
\end{conjecture}

\noindent
One can go one step further: differentiating both sides of \eqref{eq:conjIsing}
with respect to $\hh$ and setting $\delta = h^{\frac{8}{15}}$,
for $\hh = 1$ one would obtain
\begin{equation} \label{eq:superconj}
	\lim_{h \to 0} \frac{\big\langle \sigma_0
	\big\rangle_{\hlambda h^{\frac{7}{15}}, h}}{h^{\frac{1}{15}}} =
	\frac{\partial\bF}{\partial \hh}(\hlambda, 1) \,,
\end{equation}
where $\langle \sigma_0 \rangle_{\lambda, h} :=
\frac{\partial F}{\partial h}(\lambda,h) = \lim_{N\to\infty}
\bbE \big[\E_{\Lambda_N, \lambda, h}^{+,\omega}[\sigma_0] \big]$
represents the average magnetization in the infinite-volume
random field critical Ising model, cf.\ \eqref{eq:IsingFE} and \eqref{RFIMP}.
Relation \eqref{eq:superconj} is supported by (non-matching) upper and lower bounds of the order
$h^{\frac{1}{15}}$ for $\langle \sigma_0 \rangle_{\lambda, h}$
in the non-disordered case $\lambda = 0$,
recently proved by Camia, Garban and Newman \cite{CGN12b}.
\end{remark}

In light of Theorem~\ref{T:RFIM}, and the fact that for the two-dimensional Ising
model below the critical temperature a disordered external field smoothens the phase
transition~\cite{AW90}, it is natural to conjecture that a similar smoothing effect occurs at the
critical temperature:
\begin{conjecture} For any fixed $\lambda>0$, the average magnetization in the
infinite-volume random field critical Ising model on $\Z^2$ satisfies
\begin{equation}
\langle \sigma_0\rangle_{\lambda, h} \sim  C h^\gamma \qquad \mbox{as } h\downarrow 0
\qquad \mbox{for some } \gamma>\frac{1}{15}.
\end{equation}
\end{conjecture}

\section{Proof of the Lindeberg principle}
\label{S:lindeberg}

In this section, we prove Theorems~\ref{th:lindeberg} and~\ref{C:lindeberg2}
on the Lindeberg principle for polynomial chaos expansions. We first deduce
Theorem~\ref{C:lindeberg2} from Theorem~\ref{th:lindeberg}, starting with the
following lemma which controls how a multi-linear polynomial $\Psi(x) = \Psi((x_i)_{i\in \bbT})$
is affected by a shift $x \mapsto x + \mu$, where $\mu:=(\mu_i)_{i\in \bbT}$.

\begin{lemma}[Effect of adding a mean]\label{L:ExpMean} Let $\Psi(x) = \Psi((x_i)_{i\in \bbT})$
be a multi-linear polynomial as in (\ref{eq:Psi}), with kernel $\psi$,
and let $\mu:=(\mu_i)_{i\in \bbT}$ be a family of real numbers.
Then $\tilde \Psi(x) := \Psi(x+\mu)$ is a multi-linear polynomial,
i.e. $\tilde \Psi(x)= \sum_{J\in \cP^\fin(\bbT)} \tilde \psi(J) x^J$,
with kernel
\begin{equation}\label{ExpM1}
\tilde \psi(J) = \sum_{I\in \cP^\fin(\bbT):\, I \supseteq J} \psi(I) \mu^{I \setminus J} \,.
\end{equation}
For $\epsilon>0$, let $\Psi^{(\epsilon)}(x)$ be defined as in \eqref{eq:Psieps}, and recall
the definitions of $\sfC_\Psi$ from (\ref{eq:CPsi}) and $\Inf_j[\Psi]$ from  (\ref{eq:Inf}).
Then, setting $\sfc_\mu:=\sum_{i\in \bbT}\mu_i^2$, for any $\epsilon>0$ we have
\begin{align}\label{ExpM3}
\sfC_{\tilde \Psi}
\le e^{\sfc_\mu/\epsilon} \, \sfC_{\Psi^{(\epsilon)}} \,, \qquad
\Inf_j[\tilde \Psi]
\le e^{c_\mu/\epsilon} \, \Inf_j[\Psi^{(\epsilon)}] \,.
\end{align}
\end{lemma}

\begin{proof} Note that (\ref{ExpM1}) follows from the expansion
\begin{equation*}
	\tilde\Psi (x) = \Psi(x+\mu)  = \sum_I \psi(I) (x+\mu)^I
	= \sum_I \psi(I) \sum_{J \subseteq I} \mu^{I \setminus J} \, x^J
	= \sum_J \bigg( \sum_{I \supseteq J} \psi(I) \mu^{I \setminus J} \bigg)  x^J.
\end{equation*}
For any $\epsilon>0$, we can apply Cauchy-Schwarz to write
\begin{equation}
\begin{split}
	\tilde \psi(J)^2 & \le
	\bigg( \sum_{I \supseteq J} (\epsilon^{-1}\mu^2)^{I \setminus J} \bigg)
	\bigg( \sum_{I \supseteq J} \epsilon^{|I \setminus J|} \psi(I)^2 \bigg)
	= (1 + \epsilon^{-1} \mu^2)^{\bbT \setminus J}
	\bigg( \sum_{I \supseteq J} \epsilon^{|I \setminus J|} \psi(I)^2 \bigg)\,,
\end{split}
\end{equation}
and note that
\begin{equation*}
\begin{split}
	(1 + \epsilon^{-1} \mu^2)^{\bbT \setminus J}
	\le (1 + \epsilon^{-1} \mu^2)^{\bbT}
	\le  e^{c_\mu/\epsilon} \,.
\end{split}
\end{equation*}
Therefore
\begin{equation*}	
\sfC_{\tilde \Psi} := \sum_J   \tilde  \psi(J)^2 \le
	 e^{\sfc_\mu/\epsilon} \sum_I \bigg( \sum_{J \subseteq I}
	\epsilon^{|I \setminus J|} \bigg)   \psi(I)^2
	= e^{\sfc_\mu/\epsilon} \sum_I (1+\epsilon)^{|I|} \psi(I)^2
=: e^{\sfc_\mu/\epsilon}\, \sfC_{\Psi^{(\epsilon)}}  \,,
\end{equation*}
proving the first relation in \eqref{ExpM3}. The second relation is obtained similarly:
\begin{equation} \label{eq:InfPsutilde}
\begin{split}
	\Inf_j[\tilde \Psi] := \sum_{J \ni j}  \tilde \psi(J)^2
	& \le  e^{\sfc_\mu/\epsilon} \sum_{I \ni j} \Big(\sum_{J: j\in J \subseteq I}
	\epsilon^{|I \setminus J|}\Big) \psi(I)^2  \\
& \leq e^{\sfc_\mu/\epsilon} \sum_{I \ni j} (1+\epsilon)^{|I|} \psi(I)^2
=: e^{c_\mu/\epsilon} \, \Inf_j[\Psi^{(\epsilon)}] \,.
\end{split}
\end{equation}
This concludes the proof of the lemma.
\end{proof}

\medskip

\begin{proof}[Proof of Theorem~\ref{C:lindeberg2}]
Since $\Psi(\tilde\zeta) = \Psi(\zeta +\mu) =: \tilde\Psi(\zeta)$,
by Lemma~\ref{L:ExpMean} the conditions
$\sfc_\mu <\infty$ and $\sfC_{\Psi^{(\epsilon)}} <\infty$ ensure that
$\sfC_{\tilde\Psi} <\infty$. The polynomial chaos expansion
$\Psi(\tilde\zeta)$ is then a well-defined
$L^2$ random variable, by Remark~\ref{rem:L2easy},
and the same holds for $\Psi(\tilde\xi)$.

To prove (\ref{eq:lindeberg0}),
we are first going to truncate $\Psi$ to degree $\ell$,
i.e.\ we consider $\Psi^{\le \ell}$ and $\Psi^{> \ell}$, defined in (\ref{eq:Psid}).
Note that
\begin{equation} \label{lindtrunc1}
	\big| \bbE[f(\Psi(\tilde \zeta)) - f(\Psi^{\le \ell}(\tilde \zeta))] \big| \le
	\|f'\|_\infty \, \bbE[ |\Psi(\tilde \zeta) - \Psi^{\le \ell}(\tilde \zeta) | ] \le
	\|f'\|_\infty \, \bbE[ |\Psi^{> \ell}(\tilde \zeta)|^2]^{\frac{1}{2}},
\end{equation}
and the same bound holds when $\tilde \zeta$ is replaced by $\tilde \xi$, therefore
\begin{equation}\label{lindtrunc2}
\begin{split}
	\big| \bbE[f(\Psi(\tilde \zeta))] - \bbE[f(\Psi(\tilde \xi))] \big| & \le
	2 \|f'\|_\infty \bbE[ |\Psi^{> \ell}(\tilde \zeta)|^2]^{\frac{1}{2}}
	+ \big| \bbE[f(\Psi^{\le \ell}(\tilde \zeta))] - \bbE[f(\Psi^{\le \ell}(\tilde \xi))] \big|,
\end{split}
\end{equation}
where we use the fact that, since $\zeta$ and $\xi$ are independent with zero mean,
$$
\bbE[ |\Psi^{> \ell}(\zeta)|^2] = \bbE[ |\Psi^{> \ell}(\xi)|^2] = \sum_{|I|>\ell}
\psi(I)^2 = \sfC_{\Psi^{>\ell}}.
$$

To bound the first term in (\ref{lindtrunc2}), we write $\tilde \Psi^{>\ell}(x) := \Psi^{>\ell}(x+\mu)$,
which has kernel $\tilde \psi^{>\ell}$
(note that first we truncate the kernel and then we shift $x \mapsto x+\mu$).
Since $\zeta$ are independent with zero mean
and variance one, by Lemma~\ref{L:ExpMean} we have
$$
\begin{aligned}
\bbE[|\Psi^{> \ell}(\tilde \zeta)|^2] = \bbE[ |\tilde \Psi^{>\ell}(\zeta)|^2]
= \sum_J \tilde \psi^{>\ell}(J)^2 &
=: \sfC_{\tilde \Psi^{>\ell}} \le e^{\sfc_\mu/\epsilon} \, \sfC_{\Psi^{(\epsilon), >\ell}} \,.
\end{aligned}
$$
Substituting this bound into (\ref{lindtrunc2}) then leads to the first term in (\ref{eq:lindeberg0}).

To bound the second term in (\ref{lindtrunc2}), we write $\tilde \Psi^{\le \ell}(x) := \Psi^{\le \ell}(x+\mu)$,
and then apply Theorem~\ref{th:lindeberg} to obtain
$$
\begin{aligned}
\big| \bbE[f(\Psi^{\le \ell}(\tilde \zeta))] - \bbE[f(\Psi^{\le \ell}(\tilde \xi))] \big| & = \big| \bbE[f(\tilde \Psi^{\le \ell}(\zeta))] - \bbE[f(\tilde \Psi^{\le \ell}(\xi))] \big| \\
& \leq C_f \, \sfC_{\tilde \Psi^{\le \ell}} \Big( 16 \ell^2 \,  m_{2}^{>M} \,+\,
	70^{\ell+1} \, \big(m^{\le M}_{3} \big)^\ell
	\max_{i \in \bbT} \sqrt{\Inf_i\big[ \tilde \Psi^{\le \ell} \big]}\,\Big).
\end{aligned}
$$
Applying the bounds in Lemma~\ref{L:ExpMean} to $\Psi^{\le \ell}$ then gives the remaining terms in (\ref{eq:lindeberg0}), where we have combined and upper bounded factors of $e^{\sfc_\mu/\epsilon}$.
\end{proof}

\medskip

\begin{proof}[Proof of Theorem~\ref{th:lindeberg}]
We note that it is sufficient to prove the theorem in the case $|\bbT| < \infty$, because the general case follows by considering finite $\Lambda_N \uparrow \bbT$. For notational simplicity, we assume that $\bbT = [N] := \{1,\ldots, N\}$.
\medskip

\noindent
\textbf{Step 1.} We first truncate $\Psi$ to a degree $\ell$ polynomial $\Psi^{\le \ell}$. By the same calculations as in (\ref{lindtrunc1}) and (\ref{lindtrunc2}), we have
\begin{equation}\label{eq:startingpoint}
\begin{split}
	\big| \bbE[f(\Psi(\zeta))] - \bbE[f(\Psi(\xi))] \big| & \le
2 \|f'\|_\infty \bbE[ |\Psi^{> \ell}(\zeta)|^2]^{\frac{1}{2}}
+ \big| \bbE[f(\Psi^{\le \ell}(\zeta))] - \bbE[f(\Psi^{\le \ell}(\xi))] \big| .
\end{split}
\end{equation}
This leads to the first term in the right hand side of \eqref{eq:lindeberg}.

\medskip

\noindent
\textbf{Step 2.}
For a fixed $f\in C^3_b(\bbR)$, we denote
$$
f(\Psi^{\le \ell}(x)) =: g(x_{1}, x_{2}, \ldots, x_N).
$$
For a vector $x\in \bbR^N$ and $y\in \bbR$, we also set
$$h_{j}^x(y)=g(x_1, \ldots ,x_{j-1},y,x_{j+1}, \ldots, x_N).
$$
Using this notation, we can write
\begin{eqnarray}\label{Lind-tele}
f(\Psi^{\le \ell}(\zeta)) - f(\Psi^{\le \ell}(\xi))
&=&   \sum_{j=1}^N g(\gz_{1},\ldots,\gz_{j},\xi_{j+1},\ldots, \xi_N)
       -g(\gz_{1},\ldots, \zeta_{j-1}, \xi_{j},\ldots, \xi_N)\nonumber\\
 &=&   \sum_{j=1}^N  \big(h^{X^j}_{j}(\gz_{j}) - h^{X^j}_{j}(\xi_{j})\big)  ,
\end{eqnarray}
where we have used the notation
\begin{equation} \label{eq:Xj}
	X^j := (X^j_1, \ldots, X^j_N):= (\gz_{1},\ldots,\gz_{j},\xi_{j+1},\ldots, \xi_N) \,.
\end{equation}
Next, we will be Taylor expanding each term in \eqref{Lind-tele}.
For this we note that for $y\in \bbR$
$$ h^x_{j}(y)=h^x_{j}(0) +y \frac{\dd h^x_{j}}{\dd y}(0) +
\frac{y^2}{2}\frac{\dd^2h^x_{j}}{\dd y^2}(0)+ R^x_{j}(y) \,,
$$
where the remainder term has the form
$$R^x_{j}(y) =\frac{1}{2}\int_0^y \frac{\dd^3h^x_{j}}{\dd y^3}(t) (y-t)^2 \dd t
=\frac{1}{2} \frac{\dd^2 h^x_{j}}{\dd y^2}(0) \, y^2
+ \int_0^y \frac{\dd^2 h^x_{j}}{\dd y^2}(t) \, (y-t) \, \dd t ,
$$
from which it follows that
\begin{eqnarray}\label{rem-est}
|\,R^x_{j}(y)\, |\leq \min \left\{ \frac{|y|^3}{6}\left\|
\frac{\dd^3h^x_{j}}{\dd y^3} \right\|_{\infty}, y^2 \left\|
\frac{\dd^2h^x_{j}}{\dd y^2} \right\|_{\infty}  \right\}.
\end{eqnarray}
Inserting the Taylor expansion into \eqref{Lind-tele} we obtain
\begin{equation*}
\begin{split}
f(\Psi^{\le \ell}(\zeta)) - f(\Psi^{\le \ell}(\xi)) = \sum_{j=1}^N \bigg\{ & h^{X^j}_{j}(0) +
\gz_{j}\frac{\dd h^{X^j}_{j}}{\dd y}(0) +
\frac{\gz^2_{j}}{2} \frac{\dd^2h^{X^j}_{j}}{\dd y^2}(0)
+R^{X^j}_{j}(\gz_{j})\\
& -  h^{X^j}_{j}(0) - \xi_{j}\frac{\dd h^{X^j}_{j}}{\dd y}(0)
-  \frac{\xi^2_{j}}{2}\frac{\dd^2h^{X^j}_{j}}{\dd y^2}(0)
-R^{X^j}_{j}(\xi_{j}) \bigg\} \,.
\end{split}
\end{equation*}
Since $\gz_i$ and $\xi_i$ both have zero mean and unit variance, taking expectation, we get
\begin{equation}\label{eq:starting}
\begin{split}
	\Big| \bbE[f(\Psi^{\le \ell}(\zeta))] - \bbE[f(\Psi^{\le \ell}(\xi))] \Big| &=
	\Bigg| \sum_{j=1}^N \bbE\Big[ R^{X^j}_{j}(\gz_{j}) - R^{X^j}_{j}(\xi_{j}) \Big] \Bigg| \\
	& \leq \sum_{j=1}^N \bbE\Big[ |R^{X^j}_{j}(\gz_{j})| \Big]  +
	\sum_{j=1}^N\bbE\Big[ | R^{X^j}_{j}(\xi_{j})| \Big].
\end{split}
\end{equation}
Since the estimates for both sums are identical, we will focus on the first one.

\medskip

\noindent
\textbf{Step 3.}
The derivatives of $h^x_{j}(\cdot)$ are easily computed:
\begin{equation*}
\begin{split}
	\frac{\dd^m h^x_{j}}{\dd y^m}(y)
	&= f^{(m)}\Big( \! \Psi^{\le \ell}(x_1,\ldots,x_{j-1},y,x_{j+1},\ldots, x_N) \!\Big)
	\Big(\!\frac{\partial \Psi^{\le \ell}(x_1,\ldots,x_{j-1},y,x_{j+1},\ldots, x_N)}{\partial y} \!\Big)^m \\
	&= f^{(m)}\Big( \Psi^{\le \ell}(x_1,\ldots,x_{j-1},y,x_{j+1},\ldots, x_N) \Big)
	\Big( \sum_{I\ni j, \, |I| \le \ell} \psi(I)
	x^{I \setminus\{j\}}  \Big)^m \,.
\end{split}
\end{equation*}
Then, setting $C_f =\max\{ \|f'\|_\infty, \|f^{(2)}\|_\infty, \|f^{(3)}\|_\infty\}$,
we can apply \eqref{rem-est} to bound
\begin{equation}\label{eq:basicbound}
\sum_{j=1}^N \bbE\Big[ |R^{X^j}_{j}(\gz_{j})| \Big]
  \leq C_f \sum_{j=1}^N \bbE \left[
	\varphi(L_j(X^j)) \right]  ,
\end{equation}
where $X^j$ was introduced in \eqref{eq:Xj}, and we define
\begin{equation}\label{eq:varphiL}
	\varphi(x) := \min\bigg\{\frac{|x|^3}{6}, |x|^2 \bigg\} \,, \qquad
	L_{j} (x) := \sum_{I \ni j, \, |I| \le \ell} \psi(I) x^I .
\end{equation}
Notice that $L_{j} (x)$ includes the variable $x_j$ in the product as a result of absorbing the powers of $y$ in \eqref{rem-est}.
Also note that $\varphi(x) = \varphi(|x|)$ and
\begin{equation*}
	\varphi(a+b) \le \varphi(2\max\{|a|, |b|\})
	\le \varphi(2 |a|) + \varphi(2 |b|) \le 4 |a|^2  + \frac{4}{3}  |b|^3 \,.
\end{equation*}
Writing $L_j(X^j) = (L_j(X^j)-L_j(X^{j-})) + L_j(X^{j-})$, where $X^{j-}:=(X^{j-}_1, \ldots, X^{j-}_N)$ is a suitably truncated version of $X^j$, we then obtain
\begin{equation} \label{eq:newde}
\begin{split}
	\bbE \big[ \varphi(L_j(X^j)) \big] & \le
	4\, \bbE \Big[ \big( L_j(X^j) - L_j(X^{j-}) \big)^2 \Big] +
	\frac{4}{3} \, \bbE \big[ |L_j(X^{j-})|^3 \big] \,.
\end{split}
\end{equation}
The two terms in the right hand side will give rise respectively to the two terms in the
right hand side of \eqref{eq:lindeberg}.

\medskip

\noindent
\textbf{Step 4.}
We now describe the truncation procedure.
This new ingredient, with respect to
\cite{MOO10}, is tailored to control random variables with finite second moments
and uniformly integrable squares.
Fix $M \in (0,\infty)$.
We want to decompose any real-valued random variable $Y$ with {\em zero mean} and finite variance in the following way:
\begin{equation} \label{eq:decomp}
	Y = Y^- + Y^+ \,,
\end{equation}
where $Y^-, Y^+$ are functions of $Y$ and possibly of some
extra randomness, such that
\begin{equation} \label{eq:propY}
\begin{split}
	& \bbE[Y^-] = \bbE[Y^+] = 0 \,,
	\qquad Y^- Y^+ = 0 \,, \\
	& |Y^-| \le |Y| \, \ind_{\{|Y| \le M\}} \,, \qquad
	\bbE[(Y^+)^2] \le 2 \, \bbE[Y^2 \ind_{\{|Y| > M\}}] \,.
\end{split}
\end{equation}
If $\bbE[Y \ind_{\{-M \le Y \le M\}}] = 0$ we are done: just
choose $Y^- := Y \ind_{\{-M \le Y \le M\}}$ and $Y^+ := Y - Y^-$. If, on the other hand,
$\bbE[Y \ind_{\{-M \le Y \le M\}}] > 0$ (the strictly negative case is analogous), we set
\begin{equation*}
	\overline T := \sup \{T \in [0,M]: \ \bbE[Y \ind_{\{-M \le Y \le T\}}] \le 0 \}
	\in [0,M] \,.
\end{equation*}
Note that $\bbE[Y \ind_{\{-M \le Y \le \overline T\}}] \ge 0$,
because $T \mapsto \bbE[Y \ind_{\{-M \le Y \le T\}}]$
is right-continuous. If
$\bbE[Y \ind_{\{-M \le Y \le \overline T\}}] = 0$,
defining $Y^- := Y \ind_{\{-M \le Y \le \overline T\}}$ and $Y^+ := Y - Y^-$,
all the properties in \eqref{eq:propY} are clearly satisfied,
except the last one that will be checked below.
Finally, we consider the
case $\bbE[Y \ind_{\{-M \le Y \le \overline T\}}] > 0$ (then necessarily
$\overline T > 0$). Since
$\bbE[Y \ind_{\{-M \le Y < \overline T\}}] \le 0$
by definition of $\overline T$, we must have $\bbP(Y = \overline T) > 0$.
Then take a random variable $U$ uniformly distributed in $(0,1)$ and independent of $Y$,
and define
\begin{equation*}
	Y^- := Y \big( \ind_{\{-M \le Y < \overline T\}} + \ind_{\{Y = \overline T,
	\, U \le \rho\}} \big) \,, \qquad \text{where} \qquad
	\rho := \frac{-\bbE[Y \ind_{\{-M \le Y < \overline T\}}]}{\overline T\,
	\bbP(Y = \overline T)} \in (0,1) \,.
\end{equation*}
Setting $Y^+ := Y - Y^-$, all the properties \eqref{eq:propY} but the last one are
clearly satisfied.

For the last property, we write
\begin{equation*}
\begin{split}
	\bbE[(Y^+)^2] & = \bbE[(Y^+)^2 \ind_{\{|Y| > M\}}] +
	\bbE[(Y^+)^2 \ind_{\{|Y| \le M\}} ]
	= \bbE[Y^2 \ind_{\{|Y| > M\}}] +
	\bbE[(Y^+)^2 \ind_{\{|Y| \le M\}} ] \,,
\end{split}
\end{equation*}
because $Y^+ = Y$ on the event $\{|Y| > M\}$. For the second term, since
$0 \le Y^+ \le M$ on the event $\{|Y| \le M\}$, we can write
$(Y^+)^2 \le M Y^+$ (no absolute value needed). Since $Y^- = Y^- \ind_{\{|Y| \le M\}}$
has zero mean by \eqref{eq:propY}, we obtain
\begin{equation*}
\begin{split}
	\bbE[(Y^+)^2 \ind_{\{|Y| \le M\}} ] & \le
	M\, \bbE[Y^+ \ind_{\{|Y| \le M\}} ] = M\, \bbE[(Y^+ + Y^-) \ind_{\{|Y| \le M\}} ] \\
	& = M\, \bbE[Y \ind_{\{|Y| \le M\}} ] = M\, (- \bbE[Y \ind_{\{|Y| > M\}} ])
	\le \bbE[Y^2 \ind_{\{|Y| > M\}}] \,,
\end{split}
\end{equation*}
where we have used the fact that $\bbE[Y] = 0$ by assumption, and Markov's inequality.
The last relation in \eqref{eq:propY} is proved.

\medskip

\noindent
\textbf{Step 5.}
We apply the decomposition \eqref{eq:decomp} to the random variables $(X^j_i)_{i\in [N]}$,
defined in \eqref{eq:Xj},
where the extra randomness used in the construction is taken independently for each variable:
then we write
 \begin{equation*}
 X^j_{i} = X^{j-}_{i} + X^{j+}_{i} ,
 \end{equation*}
and the properties in \eqref{eq:propY} are satisfied.
Note that by \eqref{eq:varphiL} we can write
\begin{equation*}
	L_j(X^j) - L_j(X^{j-}) = \sum_{I \ni j,\, |I| \le \ell} \psi(I)
	\sum_{\Gamma \subseteq I,\, |\Gamma| \ge 1} \, (X^{j+})^\Gamma
	(X^{j-})^{I \setminus \Gamma} \,.
\end{equation*}
Since the random variables $X^{j-}_1, X^{j+}_1, X^{j-}_2, X^{j+}_2, \ldots$
are orthogonal in $L^2$ by construction, setting $\sigma_{\pm, i}^2 :=
\bbE[(X^{j\pm}_i)^2]$ and observing that $\sigma_{-,i}^2 +
\sigma_{+,i}^2 = \bbvar(X^j_i) = 1$, we obtain
\begin{equation*}
\begin{split}
	\bbE \big[ \big( L_j(X^j) - L_j(X^{j-}) \big)^2 \big] & =
	\sum_{I \ni j,\, |I| \le \ell} \psi(I)^2 \sum_{\Gamma \subseteq I,\, |\Gamma| \ge 1}
	(\sigma_{+}^2)^\Gamma (\sigma_{-}^2)^{I \setminus \Gamma} \\
	& = \sum_{I \ni j,\, |I| \le \ell} \psi(I)^2 \big( 1 - (\sigma_-^2)^I \big) \le
	\sum_{I \ni j,\, |I| \le \ell} \psi(I)^2 \big( 1- (1 - \overline\sigma_+^2)^{|I|} \big) \,,
\end{split}
\end{equation*}
where
\begin{equation*}
	\overline\sigma_+^2 := \max_{i=1,\ldots, N} \sigma_{+,i}^2
	= \max_{i=1,\ldots, N} \bbE[(X^{j+}_i)^2] \le
	2 \max_{i=1,\ldots, N} \bbE[(X^{j}_i)^2 \,
	\ind_{\{|X^{j}_i|>M\}}] \le 2 \, m_{2}^{>M} \,,
\end{equation*}
having used \eqref{eq:propY} and the definition of $m_2^{>M}$ in \eqref{eq:mtrunc} (recall \eqref{eq:Xj}).
Since $1- (1-p)^n \leq np$, we obtain
\begin{eqnarray*}
	\sum_{j=1}^N \bbE \big[ \big( L_j(X^j) - L_j(X^{j-}) \big)^2 \big] &\le&
	2 \,  m_{2}^{>M}
	\sum_{j=1}^N \Bigg( \sum_{I \ni j,\, |I| \le \ell} |I| \,\psi(I)^2 \Bigg)\\
	&\le& 2  m_{2}^{>M}\, \ell^2 \sum_{|I|\leq \ell}  \psi(I)^2 \,.
\end{eqnarray*}
Tracing back through \eqref{eq:startingpoint}, \eqref{eq:starting}, \eqref{eq:basicbound} and \eqref{eq:newde},
we note that this gives the first term in the right hand side of \eqref{eq:lindeberg}.

\medskip

\noindent
\textbf{Step 6.}
We finally consider the contribution of the second term in \eqref{eq:newde}.
We apply the hypercontractivity results in~\cite{MOO10}: by Propositions~3.16 and~3.12 therein,
denoting by $\|Y\|_q := \bbE[|Y|^q]^{1/q}$ the usual $L^q$ norm,
for every $q > 2$ we have
\begin{equation} \label{eq:hyper}
	\| L_{j}(X^{j-}) \|_q
	\le (B_q)^\ell \, \| L_{j}(X^{j-}) \|_2 \,,
\end{equation}
where
\begin{equation*}
	B_q := 2 \sqrt{q-1} \,
	\max_{i\in [N]} \frac{\|X^{j-}_i\|_q}{\|X^{j-}_i\|_2} \,.
\end{equation*}
Let us set $Y := X^{j}_i$ for short and choose $q=3$. Since $|Y^-| \le M$, by \eqref{eq:propY},
we have
\begin{equation*}
	\|Y^-\|_3 \le \bbE[ |Y|^3 \ind_{\{|Y| \le M\}} ]^{1/3}
	\le \big(m^{\le M}_3 \big)^{1/3} \,,
\end{equation*}
where we recall \eqref{eq:Xj} and the definition of $m^{\le M}_3$ from (\ref{eq:mtrunc}). On the other hand, again by \eqref{eq:propY},
\begin{equation*}
\begin{split}
	\|Y^-\|_2^2 & =
	\|Y\|_2^2 - \|Y^+\|_2^2
	= \bbE[Y^2] - \bbE[(Y^+)^2] \ge \bbE[Y^2] - 2 \bbE[Y^2 \ind_{\{|Y| > M\}}] \\
	& = 1 - 2 \bbE[ (X^{j}_i)^2
	\ind_{\{|X^{j}_i |>M\}}] \ge 1 - 2 m^{>M}_{2} \,,
\end{split}
\end{equation*}
hence
\begin{equation*}
	B_3 \le 2 \sqrt{2} \frac{\big(m^{\le M}_3 \big)^{1/3}}
	{\sqrt{1 - 2 m^{>M}_{2}}} \le 4 \big(m^{\le M}_3 \big)^{1/3} \,,
\end{equation*}
provided $m_{>M}^{(2)} \le \frac{1}{4}$, as in the assumptions of Theorem~\ref{th:lindeberg}.
Therefore, \eqref{eq:hyper} for $q=3$ yields
\begin{equation*}
	\bbE \left[ | L_{j}(X^{j-})|^3 \right] \le
	64^\ell \, \big(m^{\le M}_{3} \big)^\ell \,
	\bbE \left[ L_{j} (X^{j-})^2 \right] ^{3/2} \,.
\end{equation*}
Note that, since $\bbE[(X_i^{j-})^2] \le \bbE[(X_i^{j})^2] = 1$, we have
\begin{equation*}
\begin{split}
	\bbE \left[ L_{j} (X^{j-})^2 \right] & =
	\sum_{I \ni j,\, |I| \le \ell} \psi(I)^2 \prod_{i\in I} \bbE[(X_i^{j-})^2]
	\le \sum_{I \ni j,\, |I| \le \ell}  \psi(I)^2 = \Inf_j[\Psi^{\le \ell}] \,.
\end{split}
\end{equation*}
Therefore
\begin{equation} \label{eq:reallyquasilast}
\begin{split}
	\sum_{j=1}^N \bbE \left[ | L_{j}(X^{j-})|^3 \right]
	& \le 64^\ell \, \big( m^{\le M}_{3} \big)^\ell\,
	\Big( \max_{i \in [N]} \sqrt{\Inf_i[\Psi^{\le \ell}]} \Big)
	\sum_{j=1}^N \sum_{I \ni j,\, |I| \le \ell} \psi(I)^2 \\
	& \le \ell \, 64^\ell \, \big(m^{\le M}_{3} \big)^\ell \,
	\Big( \max_{i \in [N]} \sqrt{\Inf_i[\Psi^{\le \ell}]} \Big)
	\sum_{|I|\le \ell}  \psi(I)^2 \,.
\end{split}
\end{equation}
This gives the third term in the right hand side of \eqref{eq:lindeberg}
(because $\frac{8}{3} \ell \, 64^\ell \le 70^{\ell+1}$),
which completes the proof.
\end{proof}

\section{Proof of the convergence to Wiener chaos}
\label{sec:genproof}

In this section, we prove Theorems~\ref{th:general} and \ref{th:general2} on the convergence of
polynomial chaos expansions to Wiener chaos expansions.

\begin{proof}[Proof of Theorem~\ref{th:general}]
Let $W(\cdot)$ be the $d$-dimensional white noise used to define the Wiener chaos expansion
for $\bPsi_{0}$ in Theorem~\ref{th:general}. Given the tessellation $\cC_\delta$
indexed by $\bbT_\delta$, where each cell $\cC_\delta(x)$ has the same volume $v_\delta$,
for each $x\in \bbT_\delta$, we define
\begin{equation}
\xi_{\delta, x} := \mu_\delta(x) + v_\delta^{-1/2} \int_{\cC_\delta(x)}
\sigma_\delta W(\dd y) = v_\delta^{-1/2} \int_{\cC_\delta(x)}
\big( \sigma_\delta W(\dd y) + \bar\mu_\delta(y)\dd y\big),
\end{equation}
where we recall that $\bar\mu_\delta := v_\delta^{-1/2}\mu_\delta$ by \eqref{eq:barmu}.
Note that $\xi_\delta := (\xi_{\delta, x})_{x\in \bbT}$ is a family of independent
Gaussian random variables with the same mean and variance as
$\zeta_\delta=(\zeta_{\delta, x})_{x\in \bbT_\delta}$.

We recall that our goal is to show that $\Psi_\delta(\zeta_\delta) \to \bPsi_0$ in distribution.
The strategy is first
to focus on $\Psi_\delta(\xi_\delta)$ instead of $\Psi_\delta(\zeta_\delta)$:
we can write
$\Psi_\delta(\xi_\delta)$ as a Wiener chaos expansion with respect to $W(\cdot)$,
like $\bPsi_0$,
and show that $\bbE[|\Psi_\delta(\xi_\delta) - \bPsi_{0}|^2] \to 0$ as $\delta\downarrow 0$;
then we use the Lindeberg principle, Theorem~\ref{C:lindeberg2},
to replace $\Psi_\delta(\xi_\delta)$ by $\Psi_\delta(\zeta_\delta)$.

\medskip

\noindent
{\bf Step 1.} We first show that for each degree $k \in \N_0$,
\begin{equation}\label{deglconv}
\lim_{\delta\downarrow 0} \bbE[ |\Psi^{(k)}_\delta(\xi_\delta) - \bPsi_{0}^{(k)}|^2] = 0,
\end{equation}
where $\Psi_\delta^{(k)}$ is a polynomial with kernel
$\psi_\delta^{(k)}(I) := \psi_\delta(I) \ind_{\{ |I|=k\}}$, and similarly, $\bPsi_{0}^{(k)}$
is defined as $\bPsi_{0}$ in (\ref{eq:general}), except the kernel $\bpsi_{0}$ therein is
replaced by $\bpsi^{(k)}_{0}(I) := \bpsi_{0}(I) \ind_{\{|I|=k\}}$
(that is, we take the $k$-th term in the sum).
Recalling from \eqref{eq:barpsi} that
$\bar \psi_\delta(I) := v_\delta^{-|I|/2} \psi_\delta(I)$,
and extending $\bar \psi_\delta$ to a function
defined on $\R^k$, as discussed before
Theorem~\ref{th:general}, we can write
\begin{eqnarray*}
\Psi_\delta^{(k)}(\xi_\delta) &=& \!\!\!\!\!\!\!\!\!\! \sum_{I\in \cP^\fin(\bbT_\delta), |I|=k}
\!\!\!\!\!\!\!\!\!\! \psi_\delta(I) \xi_\delta^I = \frac{1}{k!} \idotsint_{(\R^d)^k}
\bar \psi_\delta(y_1, \ldots, y_k) \prod_{i=1}^k \big( \sigma_\delta W(\dd y_i)
+ \bar \mu_\delta(y_i) \dd y_i\big) \\
&=& \frac{1}{k!} \sum_{I \subset [k]:=\{1, \ldots, k\}}
\int_{(\R^d)^{k-|I|}} \Big(\int_{(\R^d)^{|I|}}
\!\!\sigma_\delta^{k-|I|} \bar \psi_\delta(y_1, \ldots, y_k) \prod_{i\in I}
\bar\mu_\delta(y_i) \dd y_i\Big) \prod_{j\in [k]\setminus I} W(\dd y_j).
\end{eqnarray*}
A similar expansion holds for $\bPsi_{0}^{(k)}$ with $\bar \psi_\delta$, $\bar \mu_\delta$,
$\sigma_\delta$
replaced respectively by $\bpsi_{0}$, $\bmu_{0}$,
$\bsigma_0$. Comparing the two expansions term by term
for each $I\subset [k]$, we then obtain, by the triangle inequality and the
Ito isometry \eqref{eq:WI},
\begin{eqnarray}
&& \bbE[ |\Psi^{(k)}_\delta(\xi_\delta) - \bPsi_{0}^{(k)}|^2]^{\frac{1}{2}} \label{deglconv2}\\
&\leq&  \frac{1}{k!} \sum_{I\subset [k]} \Bigg(\int_{(\R^d)^{k-|I|}}
\Bigg\{\int_{(\R^d)^{|I|}} \Big(\sigma_\delta^{k-|I|}
\bar\psi_\delta \prod_{i\in I} \bar\mu_\delta(y_i)
- \bsigma_0^{k-|I|}\bpsi_{0} \prod_{i\in I}\bmu_{0}(y_i)\Big) \prod_{i\in I}\dd y_i\Bigg\}^2
\prod_{j\in [k]\setminus I } \dd y_j\Bigg)^{\frac{1}{2}}. \nonumber
\end{eqnarray}
To see that each term above tends to $0$ as $\delta\downarrow 0$, let us assume w.l.o.g.\ that either $I=\emptyset$ or $I=[n]$ for some $1\leq n\leq k$. We then write in a telescopic sum
\begin{equation}\label{deglconv3}
\begin{aligned}
& \sigma_\delta^{k-n}
\bar\psi_\delta \prod_{i=1}^n \bar\mu_\delta(y_i) -
\bsigma_0^{k-n}\bpsi_{0} \prod_{i=1}^n \bmu_{0}(y_i) \\
=\ & \Delta_0
\prod_{i=1}^n \bar\mu_\delta(y_i) +\bsigma_0^{k-n} \bpsi_{0}
\Delta_1 \prod_{i=2}^n \bar\mu_\delta(y_i) + \cdots +
\bsigma_0^{k-n}\bpsi_{0} \prod_{i=1}^{n-1}
\bmu_{0}(y_i) \Delta_n,
\end{aligned}
\end{equation}
where $\Delta_0 :=
\bsigma_\delta^{k-n} \bar\psi_\delta - \bsigma_0^{k-n} \bpsi_{0}$ and
$\Delta_i := \bar\mu_\delta(y_i) - \bmu_{0}(y_i)$ for $i=1,\ldots,n$.
The contribution of each term from \eqref{deglconv3}
to the integrals in (\ref{deglconv2}) can be bounded by applying Cauchy-Schwarz to the
inner integral, in such a way that the $\psi$ term is separated from the product of the $\mu$'s.
It is then easily seen that all terms tend to $0$ as $\delta\downarrow 0$
by the assumptions in Theorem~\ref{th:general} that $\bar\mu_\delta\to\bmu_{0}$
and $\bar\psi_\delta \to \bpsi_{0}$ in $L^2$,
together with $\sigma_\delta\to\bsigma_0 \in (0,\infty)$.
This implies (\ref{deglconv}).
\medskip

\noindent
{\bf Step 2.} We next give a uniform $L^2$ bound on the tail of the series for $\Psi_\delta(\zeta_\delta)$ and $\bPsi_{0}$. More precisely, for any $\ell <N$, let $\Psi_\delta^{(\ell, N)}:= \sum_{\ell<k<N} \Psi_\delta^{(k)}$ and $\bPsi_{0}^{(\ell, N)}:= \sum_{\ell<k<N} \bPsi_{0}^{(k)}$. Denote $\Psi_\delta^{>\ell}$ and $\bPsi_{0}^{>\ell}$ for the case $N=\infty$. We will show that
\begin{equation}\label{truncest1}
\lim_{\ell\to\infty} \limsup_{\delta \downarrow 0} \bbE[|\Psi_\delta^{>\ell}(\zeta_\delta)|^2] = 0, \qquad\mbox{and} \qquad
\bPsi_{0} = \sum_{k=0}^\infty \bPsi_{0}^{(k)} \quad \mbox{converges in }  L^2.
\end{equation}
Together with (\ref{deglconv}) and the fact that $\bbE[\Psi_\delta^{(\ell, N)}(\xi_\delta)^2] = \bbE[\Psi_\delta^{(\ell, N)}(\zeta_\delta)^2]$ for all $0\leq \ell <N\leq \infty$, it follows that
$$
\bbE[\Psi_\delta(\zeta_\delta)^2] = \bbE[\Psi_\delta(\xi_\delta)^2] \to \bbE[\bPsi_{0}^2] \qquad \mbox{as } \delta\downarrow 0,
$$
which is one of the claims in Theorem~\ref{th:general}.

If we denote $\xi_{\delta, x} = \mu_\delta(x) + \tilde \xi_{\delta, x}$ and let
$\tilde \Psi^{(\ell, N)}_\delta(x) := \Psi^{(\ell, N)}_\delta(x+\mu_\delta)$, then for $\epsilon>0$
as specified in Theorem~\ref{th:general}~\eqref{it:3}
we can apply Lemma~\ref{L:ExpMean} (actually a modification of it, where we take into account
that the random variables do not have normalized variance) to obtain
\begin{equation}\label{truncest2}
\bbE[|\Psi_\delta^{(\ell, N)}(\xi_\delta)|^2] = \bbE[|\tilde \Psi_\delta^{(\ell, N)}
(\tilde \xi_\delta)|^2]
\ \leq\  e^{\frac{1}{\epsilon\sigma_\delta^2} \sum_{x\in \bbT_\delta} \mu_\delta(x)^2}
\!\!\!\!\!\!\sumtwo{I\in \cP^\fin(\bbT_\delta)}{\ell<|I|<N}
\!\!\!\! (1+\epsilon)^{|I|} (\sigma_\delta^2)^{|I|}\, \psi_\delta(I)^2.
\end{equation}
Similar relation as \eqref{truncest2} holds for $\zeta_\delta$ replacing $\xi_\delta$.
Since $\sum_{x\in \bbT_\delta} \mu_\delta(x)^2 = \Vert \bar\mu_\delta\Vert^2_{L^2(\R^d)}$ (recall
the extension of $f : \bbT_\delta \to \R$ to $f: \R^d \to \R$ as specified before
Theorem~\ref{th:general}), the assumptions in Theorem~\ref{th:general}~\eqref{it:1}
and~\eqref{it:3}
immediately imply the first limit in (\ref{truncest1}) if we let $N\uparrow \infty$
in (\ref{truncest2}). It also shows that $\Psi_\delta(\xi_\delta)
=\sum_{k=0}^\infty \Psi^{(k)}_\delta(\xi_\delta)$, as well as $\Psi_\delta(\zeta_\delta)$,
are $L^2$ convergent series.

By (\ref{deglconv}), we can take the limit $\delta\downarrow 0$ in (\ref{truncest2})
to obtain
$$
\bbE[|\bPsi_{0}^{(\ell, N)}|^2] = \lim_{\delta \downarrow 0}
\bbE[|\Psi_\delta^{(\ell, N)}(\xi_\delta)|^2]
\leq e^{\frac{1}{\epsilon \bsigma_0^2}
\Vert \bmu_{0}\Vert^2_{L^2(\R^d)}} \limsup_{\delta \downarrow 0} \!\!
\sumtwo{I\in \cP^\fin(\bbT_\delta)}{|I|>\ell}\!\!\!\!\! (1+\epsilon)^{|I|}
(\sigma_\delta^2)^{|I|}\,\psi_\delta(I)^2.
$$
By assumptions~\eqref{it:1} and~\eqref{it:3} of Theorem~\ref{th:general},
this implies that $\bPsi_{0}=\sum_{k=0}^\infty \bPsi_{0}^{(k)}$ is an $L^2$ convergent
series, proving the second relation in \eqref{truncest1}.

We also observe that if $\mu_\delta(x) \equiv 0$ there is no need to
apply Lemma~\ref{L:ExpMean}: relation \eqref{truncest2} holds without the exponential
pre-factor and with $\epsilon = 0$, because $\xi_\delta(x)$, $x\in \bbT_\delta$
are random variables
with zero mean and finite variance $\sigma_\delta$ (cf. Remark~\ref{truncest1}).

\medskip

\noindent
{\bf Step 3.} We now use the Lindeberg principle, Theorem~\ref{C:lindeberg2}, to show that for
each $\ell\in \N_0$,
$\Psi_\delta^{\le \ell}(\zeta_\delta):=\sum_{k=0}^\ell \Psi^{(k)}_\delta(\zeta_\delta)$
has the same limiting distribution as $\Psi_\delta^{\le \ell}(\xi_\delta)$ as $\delta\downarrow 0$.
Together with the $L^2$ convergence of $\Psi_\delta^{\le\ell}(\xi_\delta)$ to
$\bPsi_{0}^{\le \ell} := \sum_{k=0}^\ell \bPsi_{0}^{(k)}$
proved in {\em Step 1}, cf.\ \eqref{deglconv},
as well as the uniform $L^2$ bound on $\Psi_\delta^{>\ell}(\zeta_\delta)$
and $\bPsi_{0}^{>\ell}$ shown in {\em Step 2}, cf.\ \eqref{truncest1},
this implies that $\Psi_\delta(\zeta_\delta)$
converges in distribution to $\bPsi_{0}$ as $\delta \downarrow 0$.

It suffices to show that for all $f\in {\mathscr C^3}$ with
$C_f := \max\{\|f'\|_\infty, \|f''\|_\infty, \|f'''\|_\infty\} <\infty$,
\begin{equation}\label{fdiffzero}
\lim_{\delta\downarrow 0} \big|\bbE\big[f(\Psi_\delta^{\le \ell}(\zeta_\delta))
- f(\Psi_\delta^{\le \ell}(\xi_\delta))\big]\big|=0.
\end{equation}
With $\epsilon$ as specified in Theorem~\ref{th:general}~\eqref{it:3},
we can apply Theorem~\ref{C:lindeberg2}
(actually a slight modification of it, taking into account the non normalized variance $\sigma_\delta$
of the variables used here): the
absolute value in the left hand side of \eqref{fdiffzero} is bounded by
\begin{equation}\label{eq:isbounded}
e^{\frac{2}{\epsilon\sigma_\delta^2}\Vert \bar\mu_\delta\Vert^2_{L^2(\R^d)}} \, C_f
\, \widehat\sfC_{\Psi_\delta^{(\epsilon),\le\ell}} \, \Big( 16 \ell^2
        \frac{m_{2}^{>M}}{\sigma_\delta^2} +
	70^{\ell+1} \bigg(\frac{m^{\le M}_{3}}{\sigma_\delta^3} \bigg)^\ell \max_{i \in \bbT_\delta} \sqrt{\widehat\Inf_i\big[\Psi_\delta^{(\epsilon),\le\ell} \big]}\Big),
\end{equation}
where
\begin{gather*}
\widehat\sfC_{\Psi_\delta^{(\epsilon),\le\ell}}
:= \sum_{|I|\le \ell} (1+\epsilon)^{|I|} (\sigma_\delta^2)^{|I|}\, \psi_\delta(I)^2
= \sum_{k=0}^{\ell} (1+\epsilon)^{k} \frac{1}{k!} \big\Vert \sigma_\delta\,
\bar\psi_\delta\big\Vert^2_{L^2((\R^d)^{k})} \\
\widehat\Inf_i\big[\Psi_\delta^{(\epsilon),\le\ell} \big] :=\  \!\!\!\!\!
\sum_{I: I\ni i, |I|\leq \ell} (1+\epsilon)^{|I|} (\sigma_\delta^2)^{|I|}\, \psi_\delta(I)^2
=  \sum_{k=1}^\ell \frac{(1+\epsilon)^k}{(k-1)!} \, (\sigma_\delta^2)^k
\big\Vert
\bar\psi_\delta \ind_{\{x_1\in \cC_\delta(i)\}} \big\Vert_{L^2((\R^d)^{k})}^2 \,,
\end{gather*}
and we recall that
$\cC_\delta(i)\subset \R^d$ is the cell indexed by $i\in \bbT_\delta$
in the tessellation $\cC_\delta$.
We are left with showing that \eqref{eq:isbounded} vanishes as $\delta \downarrow 0$.
Note that
\begin{itemize}
\item $\widehat\sfC_{\Psi_\delta^{(\epsilon),\le\ell}} $
is uniformly bounded by Theorem~\ref{th:general}~\eqref{it:2}-\eqref{it:3};
\item $\Vert \bar\mu_\delta\Vert^2_{L^2(\R^d)}\to \Vert \bmu_{0}\Vert^2_{L^2(\R^d)}$
and $\sigma_\delta \to \bsigma_0 > 0$ by Theorem~\ref{th:general}~\eqref{it:1};
\item
$m_3^{\le M} \le M^3$
and $m_2^{>M}$ can be made arbitrarily small by choosing $M$ large,
by the definition \eqref{eq:mtrunc}
(to be applied to the centered variables $\zeta_{\delta,x} - \bbE[\zeta_{\delta,x}]$)
and the fact that
$((\zeta_{\delta,x} - \bbE[\zeta_{\delta,x}])^2)_{\delta \in (0,1),\, x \in \bbT_\delta}$
are uniformly integrable by assumption.
\end{itemize}
It only remains to verify that $\widehat\Inf_i\big[\Psi_\delta^{(\epsilon),\le\ell} \big]$
vanishes as $\delta \downarrow 0$, uniformly in $i \in \bbT_\delta$.
Since
$$
\begin{aligned}
\widehat\Inf_i\big[\Psi_\delta^{(\epsilon),\le\ell} \big]
\leq \  2 \sum_{k=1}^\ell \frac{(1+\epsilon)^k}{(k-1)!} (\sigma_\delta^2)^k
\Big( \big\Vert
\bar\psi_\delta- \bpsi_{0} \big\Vert_{L^2((\R^d)^{k})}^2
+ \Vert \bpsi_{0}\ind_{\{x_1\in \cC_\delta(i)\}} \Vert_{L^2((\R^d)^{k})}^2\Big),
\end{aligned}
$$
one has $\big\Vert \bar \psi_\delta- \bpsi_{0} \big\Vert_{L^2((\R^d)^{k})}^2\to 0$ by
Theorem~\ref{th:general}~\eqref{it:2}, and
$\Vert \bpsi_{0}\ind_{\{x_1\in \cC(i)\}} \Vert_{L^2((\R^d)^{k})}^2\to 0$ uniformly in
$i\in \bbT_\delta$ because $Leb(\cC_\delta(i))=v_\delta\downarrow 0$ uniformly in $i\in\bbT_\delta$.
This completes {\em Step 3} and establishes the convergence of $\Psi_\delta(\zeta_\delta)$
to $\bPsi_{0}$ in distribution.

\medskip

\noindent
{\bf Step 4.}
Lastly, to show that the convergence of $\Psi_\delta(\zeta_\delta)$ to $\bPsi_{0}$ in
distribution extends to the joint distribution of a finite collection of polynomial chaos expansions
$(\Psi_{i, \delta}(\zeta_\delta))_{1\leq i\leq M}$, we note that by the Cram\'er-Wold device,
it suffices to show the convergence of $\sum_{i=1}^M c_i \Psi_{i, \delta}(\zeta_\delta)$ to
$\sum_{i=1}^M c_i \bPsi_{i, 0}$ for any $c_1, \ldots, c_m\in\R$. This follows from what we have
proved so far, since
$\sum_{i=1}^M c_i \Psi_{i, \delta}(x)$ is also a polynomial that satisfies all the
required conditions.
\end{proof}

\medskip

\begin{proof}[Proof of Theorem~\ref{th:general2}]
Instead of performing an $L^2$ estimate on the tail series $\Psi_\delta^{>\ell}(\zeta_\delta)$ as
done in {\em Step 2} in the proof of Theorem~\ref{th:general}, we shall give an $L^p$ estimate for
any $p\in (0,2)$. More precisely, we replace relation
\eqref{truncest1} by the following one: for any $p\in (0,2)$,
\begin{equation}\label{truncestLp}
\bPsi_{0} = \sum_{k=0}^\infty \bPsi_{0}^{(k)} \quad \mbox{converges in }  L^p, \qquad\mbox{and} \qquad
\lim_{\ell\to\infty} \limsup_{\delta \downarrow 0} \bbE[|\Psi_\delta^{>\ell}(\zeta_\delta)|^p] = 0,
\end{equation}
and we show that this holds under either condition \eqref{it:a} or \eqref{it:b} in
Theorem~\ref{th:general2}.
The rest of the proof of Theorem~\ref{th:general} then carries over without change.

The key to proving (\ref{truncestLp}) is a change of measure argument. For $\ell<N$,
let $\bPsi_{0}^{(\ell, N)}$ and $\Psi_\delta^{(\ell, N)}$ be defined as in {\em Step 2}
in the proof of Theorem~\ref{th:general}.
Note that $\bPsi_{0}^{(\ell, N)}$ is
a finite sum of stochastic integrals with
respect to the \emph{biased} white noise $\bsigma_0 W(\dd x) + \bmu_{0}(x) \dd x$.
By the discussion in Subsection~\ref{sec:beyond}, cf.\ \eqref{eq:RN},
the joint distribution of these stochastic integrals
is absolutely continuous with respect to the unbiased case $\bmu_0(x) \equiv 0$, with
Radon-Nikodym derivative
\begin{equation}\label{muWdensity}
{\sf f}(W) := \exp\Big\{\frac{1}{\bsigma_0} \int \bmu_{0}(y) W(\dd y) - \frac{1}
{2\bsigma_0^2} \int \bmu_{0}^2 (y) \dd y  \Big\}.
\end{equation}
Therefore, using $\bPsi_{0, \, \bmu_0 \equiv 0}^{(\ell, N)}$ to denote
$\bPsi_{0}^{(\ell, N)}$ with $\bmu_0(x) \equiv 0$,
for any $p\in (0,2)$ we have
\begin{equation}\label{eq:bass}
\begin{aligned}
\bbE\big[|\bPsi_{0}^{(\ell, N)}|^p\big] & =
\bbE\big[{\sf f}(W) |\bPsi_{0, \, \bmu_0 \equiv 0}^{(\ell, N)}|^p\big]
\leq \bbE\big[{\sf f}(W)^{\frac{2}{2-p}}\big]^{\frac{2-p}{2}}
\bbE\big[|\bPsi_{0, \, \bmu_0 \equiv 0}^{(\ell, N)}|^2\big]^{\frac{p}{2}} \\
& = e^{\frac{p}{2(2-p)} \Vert \bmu_{0} / \bsigma_0 \Vert^2_{L^2(\R^d)}}
\bbE\big[|\bPsi_{0, \, \bmu_0 \equiv 0}^{(\ell, N)}|^2\big]^{\frac{p}{2}},
\end{aligned}
\end{equation}
by H\"older's inequality. By Theorem~\ref{th:general},
when $\mu_\delta =\bmu_{0}\equiv 0$ it is enough to assume
that condition~\eqref{it:3} therein holds with $\epsilon=0$ to guarantee that
$\bPsi_{0, \, \bmu_0 \equiv 0} =
\sum_{k=0}^\infty \bPsi_{0, \, \bmu_0 \equiv 0}^{(k)}$ is an $L^2$ convergent series.
Therefore $\bPsi_{0} = \sum_{k=0}^\infty \bPsi_{0}^{(k)}$
is convergent in $L^p$, by \eqref{eq:bass}.

To control $\bbE[|\Psi_\delta^{>\ell}(\zeta_\delta)|^p]$ via a change of measure for $\zeta_\delta$
is more subtle, since $(\zeta_{\delta,i})_{i\in \bbT_\delta}$ are not assumed to have
finite exponential moments. We will instead perform an exponential change of measure on a
bounded subset of the support of $\zeta_{\delta, i}$.
Since by assumption $\Vert\mu_\delta\Vert_\infty \to 0$ and
$((\zeta_{\delta, i} - \mu_\delta)^2)_{i\in\bbT_\delta}$ are uniformly integrable,
also $(\zeta_{\delta, i}^2)_{i\in\bbT_\delta}$ are uniformly integrable.
We can then apply Lemma~\ref{L:tilt} in Appendix~\ref{sec:appB}:
there exist independent random variables
$\tilde \zeta_{\delta, i}$, whose law is absolutely continuous with respect to the law
of $\zeta_{\delta, i}$, with density $f_{\delta, i}(x)$,
which satisfy \eqref{A:tiltbd1}-\eqref{A:tiltbd2}. We can then write
\begin{eqnarray}
\bbE\big[ |\Psi^{(\ell, N)}_\delta (\zeta_\delta)|^p \big] &=& \bbE\Big[ \prod_{i\in \bbT_\delta} f_{\delta, i}(\zeta_{\delta, i})^{\frac{p}{2}} |\Psi^{(\ell, N)}_\delta (\zeta_\delta)|^p \prod_{i\in \bbT_\delta} f_{\delta, i}(\zeta_{\delta, i})^{-\frac{p}{2}} \Big] \nonumber \\
&\leq& \bbE\Big[ \prod_{i\in \bbT_\delta} f_{\delta, i}(\zeta_{\delta, i}) |\Psi^{(\ell, N)}_\delta (\zeta_\delta)|^2\Big]^{\frac{p}{2}} \bbE\Big[\prod_{i\in \bbT_\delta} f_{\delta, i}(\zeta_{\delta, i})^{-\frac{p}{2-p}} \Big]^{\frac{2-p}{2}} \nonumber\\
&=& \bbE\big[ |\Psi^{(\ell, N)}_\delta (\tilde \zeta_\delta)|^2\big]^{\frac{p}{2}} \prod_{i\in \bbT_\delta} \bbE\big[ f_{\delta, i}(\zeta_{\delta, i})^{-\frac{p}{2-p}} \big]^{\frac{2-p}{2}}. \label{Lpest1}
\end{eqnarray}
Applying \eqref{A:tiltbd1}, we have
$$
\prod_{i\in \bbT_\delta} \bbE\big[ f_{\delta, i}(\zeta_{\delta, i})^{-\frac{p}{2-p}}
\big]^{\frac{2-p}{2}} \leq e^{C_p \frac{2-p}{2} \sum_{i\in \bbT_\delta} \mu_\delta(i)^2}
= e^{C_p \frac{2-p}{2} \Vert \bar\mu_\delta\Vert_{L^2(\R^d)}^2},
$$
which is uniformly bounded for $\delta$ close to $0$ (recall that
$\bar\mu_\delta \to \bmu_0$ in $L^2$).
To bound the first factor in (\ref{Lpest1}), we use the fact that
$(\tilde\zeta_{\delta, i})_{i\in \bbT_\delta}$ are independent with zero mean to obtain
\begin{equation}\label{Lpest2}
\begin{aligned}
\bbE\big[ |\Psi^{(\ell, N)}_\delta (\tilde \zeta_\delta)|^2\big] & = \sum_{\ell<|I|<N}
\Big(\prod_{i\in I} \bbE[\tilde \zeta_{\delta, i}^2]\Big) \psi_\delta(I)^2
\leq \sum_{\ell<|I|<N} e^{C' \sum_{i\in \bbT_\delta} \mu_\delta(i)^2} (\sigma_\delta^2)^{|I|}
\psi_\delta(I)^2,
\end{aligned}
\end{equation}
where in the inequality we applied (\ref{A:tiltbd2}), provided $\zeta$ satisfies condition
\eqref{it:a} in Theorem~\ref{th:general2}.
Combined with the assumption in Theorem~\ref{th:general2} that (\ref{psidtail1}) holds with
$\epsilon=0$, and tracing back to (\ref{Lpest1}), we then obtain the desired $L^p$ bound on
$\Psi^{>\ell}_\delta(\zeta_\delta)$ in (\ref{truncestLp}).

If we assume instead condition \eqref{it:b}
in Theorem~\ref{th:general2}, then we can modify the calculation in (\ref{Lpest2}) by applying
the bound  $\bbE[\tilde \zeta_{\delta, i}^2] \le \sigma_\delta^2(1 + C \mu_\delta(i))$,
stated in (\ref{A:tiltbd1}), to obtain
\begin{equation}\label{Lpest3}
\bbE\big[ |\Psi^{(\ell, N)}_\delta (\tilde \zeta_\delta)|^2\big]
\leq \sum_{\ell <|I|<N} e^{C \Vert \mu_\delta\Vert_\infty
|I|} (\sigma_\delta^2)^{|I|} \psi_\delta(I)^2.
\end{equation}
Theorem~\ref{th:general2}~\eqref{it:b} and condition (\ref{psidtail1}) with $\epsilon=0$
then give the desired bound in (\ref{truncestLp}).
\end{proof}

\section{Proof for disordered pinning model}
\label{sec:pinproof}

In this section we prove Theorem \ref{pinning scaling}.
We recall that $\tau = (\tau_k)_{k\ge 0}$ is an aperiodic renewal
process such that either $\e[\tau_1] < \infty$, or relation \eqref{eq:taualpha}
holds for some $\alpha \in (\frac{1}{2}, 1)$.
Note that
\begin{equation}\label{eq:unpin}
	u(n) := \P(n\in\tau) \sim
	\begin{cases}
	\displaystyle\frac{1}{\e[\tau_1]} & \text{if } \e[\tau_1] < \infty \\
	\displaystyle\rule{0pt}{2em}\frac{C_\alpha}{L(n) \, n^{1-\alpha}}
	& \text{if } \frac{1}{2} < \alpha < 1 \quad
	(\text{where } C_\alpha := \tfrac{\alpha \sin(\pi\alpha)}{\pi}) \,,
	\end{cases}
\end{equation}
where
the first asymptotic relation is the classical renewal theorem, while
the second one is due to Doney~\cite[Thm.~B]{D97} (see also~\cite[\S A.5]{G07}).
We also recall that $\omega = (\omega_n)_{n\in\N}$,
representing the disorder, is an i.i.d. sequence
of random variables satisfying \eqref{eq:disasspin}.

\begin{proof}[Proof of Theorem~\ref{pinning scaling}]

It suffices to rewrite the partition function
as a polynomial chaos expansion and then to check that
all the conditions of Theorem~\ref{th:general} are satisfied.
We only consider the conditioned partition function
$Z_{Nt,\beta_N,h_N}^{\omega,c}$, as the proof for the free one
follows the same lines. We also set $t=1$, to lighten notation.

\medskip
\noindent
\textbf{Step 1.}
Consider for $N\in\N$ the lattice $\bbT_N := \frac{1}{N}\N$.
Note that in Section~\ref{ConvergencePoly}
we used the notation $\bbT_\delta$ (where $\delta$ would equal $\frac{1}{N}$):
here we prefer $\bbT_N$, as
it indicates the size of the polymer.
Consequently, all the quantities in this section will be indexed by $N$ instead of $\delta$.
For each $t\in \bbT_N$ we define the cell $\mathcal{C}_{N}(t):=(t-\frac{1}{N}, t]$,
which has volume $v_N = \frac{1}{N}$.

\medskip
\noindent
{\bf Step 2.}
We now rewrite the conditioned partition function
$Z_{N,\beta,h}^{\omega,c}$, cf.\ \eqref{eq:Zpinc}, as a polynomial chaos expansion.
This was already done in \eqref{eq:Zstart}-\eqref{eq:Zstart2},
in terms of the random variables $\epsilon_i = e^{\beta\omega_i - \Lambda(\beta) + h} - 1$.
It is actually convenient to rescale the $\epsilon_i$ so that their variance
is of order one, in order to apply Theorem~\ref{th:general}.
Since $\bbvar(\epsilon_i) \sim \beta^2$ as $\beta \downarrow 0$,
recalling \eqref{eq:scalingbetah} we set
\begin{equation} \label{eq:aN}
	    a_N = \begin{cases}
    \displaystyle\frac{1}{\sqrt{N}} & \text{if } \ \E[\tau_1] < \infty \\
    \displaystyle\rule{0em}{1.8em}
    \frac{L(N)}{N^{\alpha - 1/2}} & \text{if } \ \frac{1}{2} < \alpha < 1
    \end{cases} \,,
\end{equation}
so that $\beta_N = \hbeta a_N$,
and we define the random variables $\zeta_{N} = (\zeta_{N,t})_{t \in \bbT_N}$ by
\begin{equation}\label{eq:zetapinn}
	\zeta_{N,t} :=
	\frac{1}{a_N} \big( e^{\beta_N \omega_{Nt} - \Lambda(\beta_N) + h_N} - 1 \big) \,.
\end{equation}
In this way, arguing as in \eqref{eq:Zstart}-\eqref{eq:Zstart2}, we can write
\begin{equation}\label{eq:Zpindec}
	Z_{N,\beta_N,h_N}^{\omega,c} = \Psi_N(\zeta_N) :=
	1 + \sum_{k=1}^N \frac{1}{k!}
	\sum_{(t_1, \ldots, t_k) \in (\bbT_{N})^k}
	\psi^c_{N} \big( t_1, \ldots, t_k \big)
	\prod_{i=1}^k \zeta_{N,t_i} \,,
\end{equation}
where the kernel $\psi^c_N(t_1, \ldots, t_k)$ is a
\emph{symmetric function,
which vanishes when $t_i = t_j$ for some $i \ne j$ or when
some $t_i \not\in (0,1]$}, and for
$0 =: t_0 < t_1<\cdots< t_k\leq1$ is defined by
\begin{equation} \label{eq:dikec}
\begin{split}
	\psi^c_{N} \big( t_1, \ldots, t_k \big) & :=
	a_N^k \, \p(\{Nt_1, \ldots, Nt_k\} \subseteq \tau | N\in\tau) \\
	& = \frac{u(N (1 - t_{k}))}{u(N)}\,
	\prod_{i=1}^k  a_N u(N (t_i - t_{i-1})) \,,
\end{split}
\end{equation}
recall \eqref{eq:unpin}.
We extend $\psi^c_N(t_1, \ldots, t_k)$
from $(\bbT_N)^k$ to $\R^k$ in the usual way,
as a piecewise constant function on products of cells.
The same is done for $(s,t) \mapsto u(N(t-s))$.

\medskip
\noindent
\textbf{Step 3.}
We now verify that the conditions of Theorem~\ref{th:general} are satisfied.
By our assumptions \eqref{eq:disasspin} on the disorder,
for every fixed $N\in\N$ the random variables $(\zeta_{N,t})_{t \in \bbT_N}$
are i.i.d. with mean and variance given by
\begin{equation}	\label{eq:meanpin}
\begin{split}
	\mu_N & := \bbE(\zeta_{N,t}) = \frac{1}{\sigma_N}
	\big( e^{h_N} - 1 \big) \sim \frac{\hh}{\sqrt{N}} \,, \\
	\sigma_N^2 & := \bbvar(\zeta_{N,t}) =
	\frac{1}{a_N^2} \big( e^{\Lambda(2\beta_N) - 2 \Lambda(\beta_N)} - 1 \big) e^{2 h_N}
	\sim \frac{\beta_N^2}{a_N^2} \longrightarrow \hbeta^2 \,,
	\qquad \text{as } N \to \infty \,.
\end{split}
\end{equation}
Since $v_N = \frac{1}{N}$,
condition~\eqref{it:1} of Theorem~\ref{th:general} is satisfied
with $\bsigma_0 = \hbeta$ and $\bmu_0(t) = \hh \ind_{(0,1]}(t)$.
(More precisely, redefining $\zeta_{N,t}$ as $(\zeta_{N,t} - \bbE[\zeta_{N,t}])$
when $t \not\in (0,1]$ ---which is harmless, since such values of $t$ do not
contribute to \eqref{eq:Zpindec}--- one has
$\mu_N(t) = \mu_N \ind_{(0,1]}(t) \to \bmu_0(t)$
in $L^2(\R)$.)

To prove that the random variables
$((\zeta_{N,t} - \mu_N)^2)_{N\in\N,t\in\bbT_N}$ are
uniformly integrable,
we show that the moments $\bbE [ (\zeta_{N,t} - \mu_N)^4 ]$
are uniformly bounded.
Since $\Lambda(\beta) = O(\beta)$ as $\beta \downarrow 0$,
by \eqref{eq:disasspin}, for every $k\in\N$ we can estimate
\begin{align}\label{uniform2pin}
\bbE\Big[\big(e^{\beta\omega-\Lambda(\beta)}-1\big)^{2k} \Big]
 & \leq \, 2^{2k} e^{-2k\Lambda(\beta)} \bbE\Big[\big(e^{\beta\omega}-1\big)^{2k} \Big]
                                +2^{2k} \big( 1-e^{-\Lambda(\beta)} \big)^{2k} \\
  & = \, 2^{2k} e^{-2k\Lambda(\beta)} \bbE\bigg[\bigg( \int_0^\beta \omega e^{t\omega} \dd t
  \,\bigg)^{2k}  \bigg] + 2^{2k} \big( 1-e^{-\Lambda(\beta)} \big)^{2k}\nonumber \\
  & \leq \,  2^{2k} e^{-2k\Lambda(\beta)} \beta^{2k} \int_0^\beta
  \bbE[ \omega^{2k} e^{2kt\omega} ]\,\frac{\dd t}{\beta} + 2^{2k}
  \big( 1-e^{-\Lambda(\beta)} \big)^{2k} =O(\beta^{2k}) \,.\nonumber
\end{align}
Recalling that $\beta_N = \hbeta a_N$ and $h_N = o(1)$, we obtain the desired bound:
\begin{equation*}
	\bbE \big[ (\zeta_{N,t} - \mu_N)^4 \big] \le \frac{e^{4h_N}}{a_N^4} O(\beta_N^4) = O(1) \,.
\end{equation*}

\smallskip

Let us check condition~\eqref{it:2} of Theorem~\ref{th:general}.
The renewal estimates in \eqref{eq:unpin} imply that, for fixed $0<s<t$,
\begin{equation}\label{urenew}
\lim_{N\to\infty}
\sqrt{N} \, a_N \, u \big(N(t-s) \big)
= \begin{cases}
\displaystyle
\frac{1}{\e[\tau_1]}\quad \quad & \text{if  } \E[\tau_1] < \infty \\
\displaystyle\rule{0em}{1.8em}
\frac{C_\alpha }{(t-s)^{1-\alpha}}\quad & \text{if  } \frac{1}{2} < \alpha  < 1
\end{cases} \,.
\end{equation}
Recalling the definitions \eqref{eq:dikec}, \eqref{eq:cokec} of the discrete
and continuum kernels $\psi^c_N$, $\bpsi^c_{t}$ (for $t=1$),
as well as the fact that $v_N = \frac{1}{N}$, it follows that for every
fixed $k\in\N$ the convergence
\begin{equation} \label{eq:almostL2}
	v_N^{-k/2} \, \psi_N^c(t_1, \ldots, t_k) \xrightarrow[N\to\infty]{} \bpsi_1^c(t_1, \ldots, t_k)
\end{equation}
holds \emph{pointwise}, for distinct points $t_1, \ldots, t_k$.
To obtain the required $L^2$ convergence, it suffices to exhibit an $L^2$ domination.
The case $\E[\tau_1] < \infty$ is easy:
by \eqref{eq:unpin} there exists $A \in (0,\infty)$ such that
$\frac{1}{A} \le u(n) \le A$ for every $n\in\N$, and since $v_N^{-1/2} a_N = 1$ it follows that
\begin{equation}\label{eq:psiNuniform0}
	v_N^{-k/2} \, \psi_N^c(t_1,\ldots, t_k) \le A^{k+2} \,
	\ind_{(0,1]^k}(t_1,\ldots, t_k)  \,.
\end{equation}
We now focus on the case $\frac{1}{2} < \alpha < 1$.
By Karamata's representation theorem for slowly varying functions
\cite[Theorem 1.3.1]{BGT87},
we can write $L(n)=c(n) \exp(\int_1^n \frac{\epsilon(u)}{u}\dd u)$
for some functions $c(x)\to c>0$ and $\epsilon(x)\to 0$ as $x\to\infty$.
It follows that for any $\eta > 0$ there exits a constant $A' = A'_\eta \in (0, \infty)$
such that
\begin{equation*}
	\frac{1}{A'} \bigg( \frac{n}{m} \bigg)^{-\eta} \le
	\frac{L(n)}{L(m)} \le A' \bigg( \frac{n}{m} \bigg)^{\eta} \,,
	\qquad \forall n,m \in \N \text{ with } m \le n \,.
\end{equation*}
Recalling \eqref{eq:unpin} and \eqref{eq:aN}, for possibly a larger constant $A \in (0,\infty)$
we have
\begin{equation}\label{uNuniform}
\frac{C_\alpha }{A \, (t-s)^{1-\alpha-\eta}} \le
\sqrt{N} \, a_N \, u\big(N(t-s)\big) \le
\frac{A \, C_\alpha }{(t-s)^{1-\alpha+\eta}} \,,
\end{equation}
which plugged into \eqref{eq:dikec} yields that for $0 < t_1 < \ldots < t_k \le 1$
\begin{equation}\label{eq:psiNuniform}
	v_N^{-k/2} \, \psi_N^c(t_1,\ldots, t_k) \le
	\frac{A^{k+2} \, C_\alpha^k}{t_1^{1-\alpha'} \cdots
	(t_k - t_{k-1})^{1-\alpha'} (1-t_k)^{1-\alpha'}} \,,
\end{equation}
where we set $\alpha' := \alpha - \eta$ for short.
If we choose $\eta > 0$ sufficiently small,
so that $\alpha' > \frac{1}{2}$ (recall that $\alpha > \frac{1}{2}$),
we have obtained the required $L^2$ domination.

\smallskip

We finally check condition~\eqref{it:3} of Theorem~\ref{th:general}, that is
relation \eqref{psidtail1}.
Since $\sigma_N^2$ is bounded, cf.\ \eqref{eq:meanpin}, we let $B \in (0,\infty)$
denote a constant such that $(1+\epsilon) \sigma_N^2 \le B$, so that
\begin{equation} \label{eq:assd}
\sum_{I\subseteq \bbT_N, \, |I| > \ell}
(1+\epsilon)^{|I|} (\sigma_N^2)^{|I|} \, \psi^{c}_N(I)^2
\le \sum_{k>\ell} B^{k}
\sum_{(t_1,\ldots,t_k)\in (\bbT_{N})^k \atop 0<t_1<\cdots<t_k\leq 1 }
         \psi^{c}_N(t_1,\ldots,t_k)^2 \,.
\end{equation}
If $\E[\tau_1] < \infty$, applying \eqref{eq:psiNuniform0}
and recalling that $v_N = \frac{1}{N}$, this expression is bounded by
\begin{equation*}
	\sum_{k > \ell} B^{k} \, A^{2(k+2)} \, (v_N)^k \, \frac{N^k}{k!}
	= \sum_{k> \ell} \frac{A^{2(k+2)}\, B^{k}}{k!} \,,
\end{equation*}
which is arbitrarily small for $\ell$ large, proving \eqref{psidtail1}.
If $\alpha\in(\frac{1}{2},1)$
we apply \eqref{eq:psiNuniform}: setting $\chi := 2(1-\alpha') < 1$ for short
and bounding the sums by integrals, we estimate \eqref{eq:assd} by
\begin{equation*}
\begin{split}
	\sum_{k>\ell} B^k A^{2(k+2)} C_\alpha^{2k}
\idotsint\limits_{0<t_1<\cdots<t_k<1} \frac{\dd t_1\cdots \dd t_k}{t_1^{\chi}
\cdots (t_k-t_{k-1})^{\chi} (1-t_{k})^{\chi}}
	& \le \sum_{k>\ell} B^k  A^{2(k+2)} C_\alpha^{2k} \, c_1 \, e^{-c_2 k \log k} \,,
\end{split}
\end{equation*}
where for the last inequality we have applied Lemma~\ref{L:integralbd1} below
(we recall that $\chi < 1$). Again, the sum can be made arbitrarily
small by choosing $\ell$ large, proving \eqref{psidtail1}.

\medskip
\noindent
{\bf Step 4.}
Lastly, we prove formula \eqref{pinWCag1}, when $\E[\tau_1] < \infty$.
Since $\bpsi^c_t(t_1, \ldots, t_k) = (\frac{1}{\E[\tau_1]})^k$ in this case,
formula \eqref{eq:Zlimpin} for $\hh = 0$ yields
\begin{equation}\label{eq:acds}
	\bZ_{t,\hat\beta, 0}^{W, c} = 1+\sum_{k=1}^\infty
\frac{1}{k!} \bigg(\frac{\hat\beta^k}{\bE[\tau_1]}\bigg)^k
\idotsint\limits_{[0,t]^k}  \,W(\dd t_1) \cdots  W(\dd t_k)
=e^{ \frac{\hat\beta}{\bE[\tau_1]} W([0,t]) -\frac{1}{2} \big(\frac{\hat\beta}{\bE[\tau_1]}\big)^2 t},
\end{equation}
where the second equality follows by \cite[Theorem 3.33 and Example 7.12]{J97}.
This shows that \eqref{pinWCag1} holds for $\hh = 0$, because
$W([0,t]) \sim \cN(0,t)$.
In the general case, we introduce the
tilted law $\tilde\bbP$ defined by
$\dd \tilde\bbP/\dd \bbP = \exp\{(\frac{\hh}{\hbeta}) W([0,t]) - \frac{1}{2}
(\frac{\hh}{\hbeta})^2 t\}$ and note that $\bZ_{t,\hat\beta, \hh}^{W, c}$ under $\bbP$
has the same law as $\bZ_{t,\hat\beta, 0}^{W, c}$ under $\tilde \bbP$,
cf.\ \eqref{eq:RN} and \eqref{eq:general3}.
Since $W([0,t]) \sim \cN(\frac{\hh}{\hbeta} t, t)$ under $\tilde\bbP$,
formula \eqref{pinWCag1} is proved also when $\hh \ne 0$.
\end{proof}

\section{Proof for directed polymer model}
\label{sec:dpreproof}

In this section we prove Theorem \ref{DP scaling}.
We recall that $S = (S_n)_{n\ge 0}$ is a random walk
on $\Z$ satisfying Assumption~\ref{ass:rwdp}. We denote by $p\in\N$
the period of the random walk,
so that $\P(S_1 \in p\Z + r) = 1$ for some $r \in \{0,\ldots, p-1\}$.
Introducing the lattice
\begin{equation} \label{eq:T0ha}
	\bbT := \{(n,k) \in \bbZ^2: \ k \in p \bbZ + r n\} ,
\end{equation}
we have $\P( S = (S_n)_{n \ge 0} \in \bbT) = 1$. Defining
\begin{equation} \label{eq:qnk}
	q_n(k) := \P(S_n = k) \,, \qquad \forall n \ge 0,\ k \in \Z \,,
\end{equation}
Gnedenko's local limit theorem \cite[Theorem 8.4.1]{BGT87} yields
\begin{equation} \label{eq:Gnedenko}
	\sup_{k \in \Z: \, (n,k) \in \bbT} \big| n^{1/\alpha} q_n \big( k \big) -
	p \, g\big( k/n^{1/\alpha} \big)
	\big| \xrightarrow[n\to\infty]{} 0 \,,
\end{equation}
where $g(\cdot)$ denotes the density of the stable law to which $S$ is attracted.
We also recall that $\omega = (\omega(n,k))_{n\in\N,\, k\in\Z}$ is an i.i.d. sequence
of random variables satisfying \eqref{eq:disasspin}.

\begin{proof}[Proof of Theorem \ref{DP scaling}]
As in the proof of Theorem~\ref{pinning scaling},
the strategy is to rewrite the partition function as a polynomial
chaos expansion and then to apply Theorem~\ref{th:general}.
We focus on the conditioned point-to-point partition function
$Z_{Nt,\beta_N}^{\omega,c}(N^{1/\alpha}x)$,
as the free one follows the same lines.
For notational simplicity, we set $t=1$.

\medskip
\noindent
{\bf Step 1.}
We introduce for $N\in\N$ the rescaled lattice
\begin{equation*}
	\bbT_N:=\{(N^{-1}n, N^{-1/\alpha}k) \colon (n,k)\in \bbT \} \subseteq \R^2 \,,
\end{equation*}
cf.\ \eqref{eq:T0ha}.
Note that we use $N$ instead of $\delta := \frac{1}{N}$ as a subscript,
as it indicates the ``length'' of the polymer.
For each $(t,x)\in \bbT_N$, we define the cell
$\mathcal{C}_N((t,x)):=(t-\frac{1}{N},t]
\times (x-\frac{p}{N^{1/\alpha}}, x]$,
which has volume $v_N = p N^{-(1+1/\alpha)}$.

\medskip
\noindent
{\bf Step 2.}
We rewrite the conditioned partition function
$Z_{N,\beta_N}^{\omega,c}(N^{1/\alpha}x)$,
defined in \eqref{eq:Zpoint2point}, as a polynomial chaos expansion,
using the random variables $\zeta_{N} = (\zeta_N(s,y))_{(s,y) \in \bbT_N}$ given by
\begin{equation} \label{eq:zetamu}
\begin{split}
	\zeta_{N}(s,y) := N^{\frac{\alpha-1}{2\alpha}}
	\Big( e^{\beta_N \omega(Ns, N^{1/\alpha} y) -\Lambda(\beta_N)}
	- 1 \Big)  \,,
\end{split}
\end{equation}
where the prefactor has been chosen so that $\bbvar(\zeta_{N}(s,y)) = O(1)$,
see below.
Arguing as in \eqref{eq:Zstart}-\eqref{eq:Zstart2},
we can write
\begin{equation*}
	Z_{N,\beta_N}^{\omega,c}(N^{1/\alpha}x) = \Psi_N(\zeta_N) := 1 + \sum_{k=1}^N
	\frac{1}{k!}
	\sum_{(z_1, \ldots, z_k) \in (\bbT_N)^k}
	\psi_{N,(1,x)}^{c} (z_1, \ldots, z_k)
	\prod_{i=1}^k \zeta_{N}(z_i) \,,
\end{equation*}
where $\psi_{N,(1,x)}^{c}$ is a \emph{symmetric function of $(z_1, \ldots, z_k) \in (\bbT_N)^k$,
which vanishes when $z_i = z_j$ for some $i \ne j$ or when
some $z_i \not\in (0,1] \times \R$}, and for distinct
$z_1 = (t_1, x_1)$, \ldots, $z_k = (t_k, x_k)$, say with
$0 < t_1<\cdots< t_k\leq1$, is defined by (recall \eqref{eq:qnk})
\begin{equation}\label{psikDP}
\begin{split}
	& \psi_{N,(1,x)}^{c} \big( (t_1, x_1), \ldots, (t_k,x_k) \big) :=
	\frac{\P\big( S_{N t_1} = N^{1/\alpha} x_1, \ldots, S_{N t_k} = N^{1/\alpha} x_k
	\big| S_N = x \big)}{(N^{\frac{\alpha-1}{2\alpha}})^k} \\
	& \qquad \qquad =
	\frac{q_{N (1 - t_{k})} \big(
	N^{1/\alpha}(x - x_{k}) \big)}{q_{N } \big(
	N^{1/\alpha}x \big)}
	\prod_{i=1}^k  \Big( N^{-\frac{\alpha-1}{2\alpha}} \, q_{N (t_i - t_{i-1})} \big(
	N^{1/\alpha}(x_i - x_{i-1}) \big)\Big) \,,
\end{split}
\end{equation}
where $(t_0,x_0) := (0,0)$.
The kernel $\psi_{N,(1,x)}^{c} (z_1, \ldots, z_k)$
is extended from $(\bbT_N)^k$ to $(\R^2)^k$
in the usual way, as a piecewise constant function
which is constant on every product of cells.
The same extension is done for the function $((s,x), (t,y)) \mapsto q_{N(t-s)}(N^{1/\alpha}(y-x))$.

\medskip
\noindent
{\bf Step 3.}
We now check the assumptions of Theorem~\ref{th:general}.
Recalling \eqref{eq:disasspin} and the fact that
$\beta_N = \hbeta N^{-\frac{\alpha-1}{2 \alpha}}$,
cf. Theorem~\ref{DP scaling},
for every fixed $N\in\N$ the random variables $(\zeta_{N}(z))_{z \in \bbT_N}$
are i.i.d. with zero mean $\mu_N(z) \equiv 0$ and variance given by
\begin{equation}\label{eq:varzetadp}
	\sigma_N^2 = \bbvar(\zeta_N(z)) =
	N^{\frac{\alpha-1}{\alpha}} \, e^{\Lambda(2\beta_N) - 2 \Lambda(\beta_N)}
	\sim N^{\frac{\alpha-1}{\alpha}} \, \beta_N^2 \xrightarrow[N\to\infty]{}
	\hbeta^2 \,.
\end{equation}
Condition \eqref{it:1} of Theorem~\ref{th:general} is thus
satisfied with $\bmu_0(z) \equiv 0$ and $\bsigma_0 = \hbeta$.
The uniform integrability of $(\zeta_{N}(z)^2)_{N\in\N, z\in\bbT_N}$
is easily checked as for the pinning model, cf.\ \eqref{uniform2pin}.

\smallskip

Let us check condition \eqref{it:2}.
Recalling the definition \eqref{eq:stable transition} of the function $g_t(\cdot)$,
we observe that by \eqref{eq:Gnedenko}, for fixed $0 < s < t$ and $x,y \in \R$,
\begin{equation*}
	\lim_{N\to\infty} N^{1/\alpha} \, q_{N(t-s)}\big( N^{1/\alpha}(y - x) \big)
	= p \, g_{t-s}(y-x)  \,.
\end{equation*}
Recalling the definition
\eqref{eq:contkerDP} of the continuum kernel $\bpsi_{t,x}^c$ (for $t=1$),
since $v_N = p N^{-1-1/\alpha}$, it follows by  \eqref{psikDP} that
for every $k\in\N$ one has the \emph{pointwise} convergence
\begin{equation} \label{eq:alL2}
	\lim_{N\to\infty}
	v_N^{-k/2} \, \psi_{N,(1,x)}^{c} \big( (t_1, x_1), \ldots, (t_k,x_k) \big)
	= \bpsi_{1,x}^{c} \big( (t_1, x_1), \ldots, (t_k,x_k) \big) \,.
\end{equation}
We need to show that this convergence also holds in $L^2$.
Since the density $g(\cdot)$ is bounded, relation \eqref{eq:Gnedenko}
yields that for some constant $A \in (0,\infty)$
\begin{equation} \label{eq:easb}
	q_n(k) \le A n^{-1/\alpha} \,, \qquad \forall n \in \N_0, \ k \in \Z \,.
\end{equation}
For fixed $x \in \R$, one has $q_N(N^{1/\alpha} x) \ge \frac{p}{2} g(x) / N^{1/\alpha}$ for large $N$,
again by \eqref{eq:Gnedenko}, hence the prefactor in \eqref{psikDP}
is upper bounded as
$$
	\frac{q_{N (1 - t_{k})} \big(
	N^{1/\alpha}(x - x_{k}) \big)}{q_{N } \big(
	N^{1/\alpha}x \big)} \leq \frac{C_x}{(1-t_k)^{1/\alpha}}, \qquad
	 \text{where}\qquad C_x := \frac{2 A}{ (p g(x))}.
$$
Applying \eqref{eq:easb} to each term
in the product in \eqref{psikDP} and recalling that $v_N = p N^{-1-1/\alpha}$,
for $0 < t_1 < \ldots < t_k \le 1$ we get 
\begin{equation} \label{eq:keyboun}
\begin{split}
	& \big[ v_N^{-k/2}
	\, \psi_{N,(1,x)}^{c} \big( (t_1, x_1), \ldots, (t_k,x_k) \big) \big]^{2} 
	 \ \le \frac{C_x}{(1-t_k)^{1/\alpha}} \, \frac{q_{N (1 - t_{k})} \big(
	N^{1/\alpha}(x - x_{k}) \big)}{q_{N } \big(
	N^{1/\alpha}x \big)} \\
	& \qquad\quad \times \prod_{i=1}^k \big( v_N^{-1/2} \, N^{-\frac{\alpha-1}{2\alpha}} \big)^{2}
	\bigg( \frac{A}{N^{1/\alpha}(t_i - t_{i-1})^{1/\alpha}} \bigg)
	q_{N (t_i - t_{i-1})} \big( N^{1/\alpha}(x_i - x_{i-1}) \big) \\
	& \ = C'_{k,x} \, \big(N^{1/\alpha}\big)^k \,
	\frac{\P(S_{N t_1} = N^{1/\alpha} x_1, \ldots, S_{N t_k} = N^{1/\alpha} x_k\, \big| S_N=N^{1/\alpha} x)}
	{t_1^{1/\alpha} (t_2-t_1)^{1/\alpha} \cdots (t_k - t_{k-1})^{1/\alpha} \,(1-t_k)^{1/\alpha}} \,,
\end{split}
\end{equation}
where we set $C'_{k,x} := A^k p^{-k} C_x$.
A further application of \eqref{eq:easb} also yields 
\begin{equation} \label{eq:keyboun2}
	\big[ v_N^{-k/2} \, \psi_{N,(1,x)}^{c} \big( (t_1, x_1), \ldots, (t_k,x_k) \big) \big]^{2}
	\le 	\frac{A^k\, C'_{k,x}}
	{t_1^{2/\alpha} (t_2-t_1)^{2/\alpha} \cdots (t_k - t_{k-1})^{2/\alpha}\, (1-t_k)^{2/\alpha}} \,.
\end{equation}
We now decompose the domain
$\{0 < t_1 < \ldots < t_k < 1\}\times \R^k$  as $D_1 \cup D_2 \cup D_3$, where
\begin{align*}
	&\qquad \quad D_1 := \bigcap_{i=1}^k \{t_i - t_{i-1} > \eta, \ |x_i| < M \}\cap \{1-t_k>\eta \}\,, \quad \\
	&D_2 := \bigcup_{i=1}^k \{t_i - t_{i-1} \le \eta \}
	\cup \{1-t_k\leq \eta \} \,, \quad
	D_3 := \bigcup_{i=1}^k \{|x_i| \ge M \} \,,
\end{align*}%
for fixed $\eta, M \in (0,\infty)$. Relation \eqref{eq:keyboun2}
shows that the rescaled kernel $v_N^{-k/2} \psi_{N,(1,x)}^c(\cdot)$
is uniformly bounded on the (bounded) set $D_1$, hence the convergence \eqref{eq:alL2}
holds in $L^2$ on $D_1$. If we show that the integrals of
$[v_N^{-k/2} \psi_{N,(1,x)}^c(\cdot)]^2$ and of $[\bpsi_{1,x}^c(\cdot)]^2$
over the sets $D_2$ and $D_3$ can be made arbitrarily small, for suitable
$\eta, M$, we are done. Since $\bpsi_{1,x} \in L^2([0,1]^k \times \R^k)$,
by \eqref{eq:stable transition}, there is nothing to prove for $\bpsi_{1,x}^c(\cdot)$
and we may focus on $\psi_{N,(1,x)}^c(\cdot)$. By \eqref{eq:keyboun}
\begin{align}
	\int_{D_2} |v_N^{-k/2}
	\, \psi_{N,(1,x)}^{c}(\cdot) |^2
	& = \sum_{(t_1, x_1), \ldots, (t_k,x_k) \in (\bbT_N)^k}
	\!\!\!\!\! (v_N)^k
	\big[ v_N^{-k/2}
	\, \psi_{N,(1,x)}^{c} \big( (t_1, x_1), \ldots, (t_k,x_k) \big) \big]^{2} \nonumber \\
	& \le C'_{k,x} \, p^k \,
	\frac{1}{N^k} \sumthree{(t_1, \ldots, t_k) \in \frac{1}{N}\N}{0 < t_1 < \ldots < t_k \le 1}
	{t_i - t_{i-1} \le \eta \text{ for some } i}
	\frac{1}{t_1^{1/\alpha}\cdots (t_k-t_{k-1})^{1/\alpha} \, (1-t_k)^{1/\alpha}},
	\label{eq:L2+e} \\
	& \le C'_{k,x} \, p^k
	\ \ \idotsint\limits_{\substack{0<t_1<\cdots<t_k<1 \\
	t_i - t_{i-1} \le \eta \text{ for some } i}}
	\ \
	\frac{\dd t_1\cdots \dd t_k}{t_1^{1/\alpha}\cdots (t_k-t_{k-1})^{1/\alpha}\,
	(1-t_k)^{1/\alpha}}
	\nonumber  \,,
\end{align}
which vanishes as $\eta \downarrow 0$ (recall that $\alpha > 1$). Analogously,
using again \eqref{eq:keyboun}, we get
\begin{equation} \label{eq:D3}
\begin{split}
	 \int_{D_3} |v_N^{-k/2}
	\, \psi_{N,(1,x)}^{c}(\cdot) |^2 
         & \le C'_{k,x} \, p^k \Bigg( \ \	\idotsint\limits_{0<t_1<\cdots<t_k<1} \ \
	\frac{\dd t_1\cdots \dd t_k}{t_1^{1/\alpha}\cdots (t_k-t_{k-1})^{1/\alpha}\,
	(1-t_k)^{1/\alpha}}\Bigg)\\
	&\hskip 1cm \times \P\Big( \max_{0 \le n \le N} 
	| S_{n} | \ge N^{1/\alpha} M \,\big| \,S_N=N^{1/\alpha} x \Big) \,.
\end{split}
\end{equation}
As $N\to\infty$, the last probability converges to
$\P(\sup_{0 \le t \le 1} |X_t| \ge M \,|\, X_1 = x)$, where
$X = (X_t)_{t\ge 0}$ is the stable L\'evy process
to which the random walk is attracted, 
cf.~\cite{Lig68}, hence it can be made as small as one
wishes, uniformly in $N\in\N$, by choosing $M$ large. This completes
the verification of condition \eqref{it:2} in Theorem~\ref{th:general}.

\smallskip

Finally, we check condition \eqref{it:3}.
Since $\sigma_N^2$ is bounded, cf.\ \eqref{eq:varzetadp},
we have $(1+\epsilon) \sigma_N^2 \le B$
for some $B\in (0,\infty)$
(we can even set $\epsilon = 0$, because $\mu_N \equiv 0$ in this case).
Applying \eqref{eq:L2+e} for $\eta=1$, i.e. with no restriction on $t_i - t_{i-1}$
(equivalently, \eqref{eq:D3} with $M=0$), we obtain
\begin{align*}
	& \sum_{I\subset \bbT_N, \, |I|>\ell} (1+\epsilon)^{|I|} (\sigma_N^2)^{|I|}
	\psi^{c}_{N,(1,x)}(I)^2
	\le \sum_{k > \ell} B^k \frac{1}{k!}
	\sum_{(z_1, \ldots, z_k) \in (\bbT_N)^k} \psi^{c}_{N,(1,x)}(z_1, \ldots, z_k)^2 \\
	& \qquad = \sum_{k > \ell} B^k \frac{1}{k!} \|v_N^{-k/2} \psi^{c}_{N,(1,x)}\|_{L^2((\R^2)^k)}^2 \\
	& \qquad \le \sum_{k > \ell} \tilde B_x^k
	\ \ \idotsint\limits_{0<t_1<\cdots<t_k<1} \ \
	\frac{\dd t_1\cdots \dd t_k}{t_1^{1/\alpha}\cdots (t_k-t_{k-1})^{1/\alpha}\,
	(1-t_{k})^{1/\alpha}}
	\le \sum_{k > \ell} \tilde B_x^k \, c_1 \, e^{-c_2 k \log k} \,,
\end{align*}
where, recalling that $C'_{k,x} = A^k p^{-k} C_x$, we have set
$\tilde B_x := B \, A \, \sqrt{p} \, \max\{1,C_x\}^{2}$, which
is a finite constant for every fixed $x\in\R$, and we have
applied Lemma~\ref{L:integralbd1} for the last inequality.
This shows that \eqref{psidtail1} holds, hence
condition \eqref{it:3} in Theorem~\ref{th:general} is verified.
\end{proof}

\section{Proof for random field Ising model}
\label{sec:Isingproof}

In this section we prove Theorem~\ref{T:RFIM} and Corollary~\ref{C:coco}.
We recall that the disordered partition function
$Z^{+,\omega_\delta}_{\bbOmegadelta, \lambda_\delta, h_\delta}$ is defined as in
\eqref{RFIMZ}, where $\bbOmegadelta := \Omega \cap (\delta\Z)^2$
(with $\Omega \subseteq \R^2$ being a fixed bounded, simply connected open set
with piecewise smooth boundary) and where:
\begin{itemize}
\item $\P_{\bbOmegadelta}^+$ (with expectation $\E_{\bbOmegadelta}^+$) denotes the critical Ising model on $\Z^2$,
defined as in \eqref{eq:Isingbasic};

\item $\omega_\delta = (\omega_{\delta,x})_{x\in\bbOmegadelta}$ is an
i.i.d.\ family of random variables satisfying \eqref{eq:disasspin};

\item $\lambda_\delta = (\lambda_{\delta,x})_{x\in\bbOmegadelta}$,
$h_\delta = (h_{\delta,x})_{x\in\bbOmegadelta}$ are defined by
\begin{equation}\label{lhdelta2}
	\lambda_{\delta,x} := \hlambda(x) \,\delta^{\frac{7}{8}} \,, \qquad
	h_{\delta, x} := \hh(x) \,\delta^\frac{15}{8}\,,
\end{equation}
cf.\ \eqref{lhdelta}, where
$\hlambda: \overline{\Omega} \to (0,\infty)$ and $\hh: \overline{\Omega} \to \R$
are fixed continuous functions.
\end{itemize}
The heart of our proof are pointwise and $L^2$ estimates for the critical Ising correlation
functions, in particular near the diagonals (see
Lemmas~\ref{L:correlationbd}--\ref{L:fOmega} below). Complementary $L^1$ estimates
have been recently established in \cite[Prop.3.9]{CGN12}.

\begin{proof}[Proof of Theorem~\ref{T:RFIM}]
We are going to apply Theorem~\ref{th:general}, with $v_\delta=\delta^2$,
rewriting the partition function in terms of a polynomial chaos expansion.

\medskip
\noindent
{\bf Step 1.} By relation \eqref{RFIMZ}, we can write
\begin{equation}\label{eq:probi}
\begin{split}
Z^{+, \omega_\delta}_{\bbOmegadelta, \lambda_\delta, h_\delta}
& = \E^{+}_{\bbOmegadelta} \Bigg[ \prod_{x\in \bbOmegadelta}
\big(\cosh(\xi_{\delta, x}) + \sigma_x \sinh (\xi_{\delta, x})\big)\Bigg] \,, \\
&\text{where} \quad \xi_{\delta, x}:= \lambda_{\delta, x}\omega_{\delta, x} +h_{\delta, x} \,.
\end{split}
\end{equation}
Recalling the notation $\alpha^I := \prod_{x\in I} \alpha_x$,
a binomial expansion of the product yields
\begin{align}
e^{-\frac{1}{2} \| \hlambda \|_{L^2(\Omega)}^2 \, \delta^{-\frac{1}{4}}}
Z^{+, \omega_\delta}_{\bbOmegadelta, \lambda_\delta, h_\delta}
&= e^{-\frac{1}{2} \| \hlambda \|_{L^2(\Omega)}^2 \, \delta^{-\frac{1}{4}}}
\sum_{I\subseteq \bbOmegadelta}
\cosh(\xi_{\delta,\cdot})^{\bbOmegadelta\backslash I} \,
\E^+_{\bbOmegadelta} [ \sigma_\cdot^I ]  \, \sinh(\xi_{\delta,\cdot})^I \nonumber \\
&= e^{-\frac{1}{2} \| \hlambda \|_{L^2(\Omega)}^2 \, \delta^{-\frac{1}{4}}}
\cosh(\xi_{\delta,\cdot})^{\bbOmegadelta} \sum_{I\subseteq \bbOmegadelta}
\E^+_{\bbOmegadelta} [ \sigma_\cdot^I ] \, \tanh(\xi_{\delta,\cdot})^I \,. \label{RFIMZexp}
\end{align}
We first show that the pre-factor before the sum
converges to $1$ in probability as $\delta\downarrow 0$.

Recalling the definition \eqref{eq:probi} of $\xi_{\delta,x}$
and the fact that $\omega_{\delta,x}$ have zero mean,
unit variance and locally finite exponential
moments, cf.\ \eqref{eq:disasspin}, a Taylor expansion yields
\begin{equation}\label{cdelta}
\bbE[\log \cosh(\xi_{\delta, x})]
= \frac{\lambda_{\delta, x}^2}{2} + O(h_{\delta, x}^2 + \lambda_{\delta, x}^4)
= \frac{\hlambda(x)^2}{2} \delta^{\frac{7}{4}} + O(\delta^{\frac{7}{2}}),
\end{equation}
where the term $O(\delta^{\frac{7}{2}})$ is uniform over $x\in\bbOmegadelta$,
by the continuity of $\hlambda, \hh$.
Therefore, as $\delta \downarrow 0$,
\begin{equation} \label{eq:cnorm}
	\sum_{x\in \bbOmegadelta} \bbE[\log \cosh(\xi_{\delta, x})]  =
	\frac{1}{2} \| \hlambda \|_{L^2(\Omega)}^2 \, \delta^{-\frac{1}{4}}
	\,+\, o(1) \,,
\end{equation}
and the pre-factor in \eqref{RFIMZexp} can be rewritten as
\begin{equation*}
	\exp\bigg\{
	\sum_{x\in \bbOmegadelta}
	\Big( \log \cosh(\xi_{\delta, x}) - \bbE[\log \cosh(\xi_{\delta, x})] \Big)
	\bigg\} \big(1 + o(1) \big) \,.
\end{equation*}
The sum is over $|\bbOmegadelta| = O(\delta^{-2})$
i.i.d. centered random variables, hence it converges to zero in probability
provided $\bbvar[\log \cosh(\xi_{\delta, x})] = o(\delta^2)$. This is checked by
a Taylor expansion:
\begin{equation*}
	\bbvar[\log \cosh(\xi_{\delta, x})] \le
	\bbE[(\log \cosh(\xi_{\delta, x}))^2] = O(\lambda_{\delta,x}^4) = O(\delta^{\frac{7}{2}})
	= o(\delta^2) \,.
\end{equation*}

\medskip
\noindent
{\bf Step 2.} It remains to verify that the sum in \eqref{RFIMZexp}
converges to the desired Wiener chaos expansion, namely \eqref{Z+Wexp}.
Defining the family $\zeta_\delta = (\zeta_{\delta,x})_{x\in\bbOmegadelta}$ by
\begin{equation*}
	\zeta_{\delta,x} :=
	\frac{\tanh(\xi_{\delta,x})}{\bbvar(\tanh(\xi_{\delta,x}))^{1/2}} \,,
\end{equation*}
the sum in \eqref{RFIMZexp} can be written as a polynomial chaos expansion
$\Psi_\delta(\zeta_\delta) := \sum_{I \subseteq \bbOmegadelta}
\psi_\delta(I) \, \zeta_{\delta,\cdot}^I$,
where
\begin{equation}\label{eq:psideltaI}
	\psi_\delta (I) :=
\bbvar(\tanh(\xi_{\delta,\cdot}))^{\frac{|I|}{2}}
\, \E^+_{\bbOmegadelta} [ \sigma_I ] \,, \qquad I\subseteq \bbOmegadelta \,.
\end{equation}
We are thus left with checking that the conditions in Theorem~\ref{th:general}
are satisfied.

By a Taylor expansion, as $\delta\downarrow 0$ one has
\begin{equation}\label{eq:ee2}
\begin{split}
	\bbE[\tanh(\xi_{\delta,x})] & = h_{\delta,x} + O(h_{\delta,x}^3 + \lambda_{\delta,x}^3)
	= \hh(x) \, \delta^{\frac{15}{8}} + O(\delta^{\frac{21}{8}}) \,, \\
	\bbE[\tanh(\xi_{\delta,x})^2] & = \lambda_{\delta,x}^2 + O(h_{\delta,x}^2 + \lambda_{\delta,x}^4)
	= \hlambda(x)^2 \, \delta^{\frac{7}{4}} + O(\delta^{\frac{7}{2}}) \,,
\end{split}
\end{equation}
where the $O(\cdot)$ terms are uniform in $x$, by the continuity of $\hlambda, \hh$.
Therefore, uniformly in $x$,
\begin{equation*}
	\mu_\delta(x) := \bbE[\zeta_{\delta,x}] = \frac{\hh(x)}{\hlambda(x)} \, \delta + o(\delta) \,,
	\qquad \sigma_\delta^2 := \bbvar[\zeta_{\delta,x}] = 1 \,.
\end{equation*}
Recalling that $v_\delta = \delta^2$ and
$\hlambda: \overline\Omega \to (0,\infty)$ is continuous
(hence uniformly bounded away from zero),
condition~\eqref{it:1} of~Theorem~\ref{th:general} is satisfied, with
$\bsigma_0 = 1$ and $\bmu_0(x) := \hh(x)/\hlambda(x)$.

The uniform integrability of
$((\zeta_{\delta,x} - \mu_\delta(x))^2)_{\delta \in (0,1), x \in \bbOmegadelta}$
holds because the moments $\bbE [ (\zeta_{\delta,x} - \mu_\delta(x))^4 ]$
are uniformly bounded, as for the disordered pinning model,
cf.\ \eqref{uniform2pin}.

\medskip
\noindent
{\bf Step 3.}
It remains to check conditions~\eqref{it:2}
and~\eqref{it:3} of~Theorem~\ref{th:general}. By \eqref{eq:ee2}
\begin{equation*}
	\bbvar[\tanh(\xi_{\delta,x})] = \hlambda(x)^2 \, \delta^{\frac{7}{4}} +
	O(\delta^{\frac{7}{2}}) \,,
\end{equation*}
hence for fixed $I = \{x_1 ,\ldots, x_n\} \subseteq \Omega$,  by \eqref{eq:psideltaI},
\begin{equation*}
	\psi_\delta (I) =
	\hlambda(\cdot)^{|I|} \, \delta^{\frac{7}{8} |I| } \, \E^+_{\bbOmegadelta} [ \sigma_\cdot^I ]
	= \hlambda(x_1) \cdots \hlambda(x_n) \, \delta^{\frac{7}{8} n} \,	\E^+_{\bbOmegadelta} [
	\sigma_{x_1} \cdots \sigma_{x_n} ] \,.
\end{equation*}
Recalling that $v_\delta = \delta^2$,
relation \eqref{CHI1.30}, that was recently proved by Chelkak, Hongler and Izyurov~\cite{CHI12},
yields immediately that for every
$n\in\N$ and distinct $x_1, \ldots, x_n \in \Omega$
\begin{equation} \label{eq:pointI}
	\lim_{\delta \downarrow 0} \, v_\delta^{-\frac{n}{2}}
	\psi_\delta (\{x_1, \ldots, x_n\})
	= \hlambda(x_1) \cdots \hlambda(x_n)
	\, \cC^n \, \bphi_\Omega^+(x_1,\ldots, x_n)
	=: \bpsi_0(\{x_1, \ldots, x_n\})  \,.
\end{equation}
(Incidentally, since $\bsigma_0 = 1$ and $\bmu_0(x) := \hh(x)/\hlambda(x)$,
the Wiener chaos expansion of Theorem~\ref{th:general}, cf.\ \eqref{eq:general},
matches with the one of Theorem~\ref{T:RFIM}, cf.\ \eqref{Z+Wexp}.)

To extend the pointwise convergence \eqref{eq:pointI}
to $L^2$ convergence, we need uniform bounds on
$\psi_\delta(I)$. By Lemma~\ref{L:correlationbd} below, we have the following bound uniformly in $\delta\in (0,1)$:
\begin{equation}\label{Iineq4}
 v_\delta^{-\frac{|I|}{2}}
\psi_\delta (I) \leq (C \|\hlambda\|_\infty)^{|I|}
\prod_{i=1}^{|I|} \frac{1}{d(x_i, \partial\Omega\cup I\backslash\{x_i\})^{\frac{1}{8}}} =: (C \|\hlambda\|_\infty)^{|I|} f_\Omega(x_1,\ldots, x_{|I|}),
\end{equation}
where by Lemma~\ref{L:fOmega} below, given $|I|=n$ for any $n\in\N$,
\begin{equation}\label{Iineq12}
\frac{1}{n!}\Vert f_\Omega\Vert^2_{L^2(\Omega^n)} \leq C^n (n!)^{-\frac{3}{4}}.
\end{equation}
Combined with \eqref{eq:pointI}, it follows that conditions~\eqref{it:2} and~\eqref{it:3} of~Theorem~\ref{th:general} are satisfied and this completes the proof of Theorem~\ref{T:RFIM}.
\end{proof}

We next state and prove the lemmas needed to establish \eqref{Iineq4} and (\ref{Iineq12}).

\begin{lemma}\label{L:correlationbd}
Let $\Omega$, $\bbOmegadelta$ for $\delta>0$, and $\E_{\bbOmegadelta}^+$
be as introduced at the beginning of this section. Then there exists
$C = C(\Omega) \in (0,\infty)$ such that for any
$I=\{x_1, \ldots, x_n\}\subseteq \Omega$ with $|I|=n$
\begin{equation}\label{Iineq3}
0\leq \E^+_{\bbOmegadelta} [ \sigma_{x_1} \cdots \sigma_{x_n}] \leq
C^n \delta^{\frac{n}{8}} \prod_{i=1}^n
\frac{1}{d(x_i, \partial\Omega\cup I\backslash\{x_i\})^{\frac{1}{8}}},
\end{equation}
where for any $x\in \Omega$ we define
$\sigma_x:=\sigma_{x_\delta}$, with $x_\delta$ being the point in $\bbOmegadelta$
closest to $x$, and we set $d(x,A) := \inf_{y \in A} \|x-y\|$.
\end{lemma}

\begin{proof}
If $B(x;r)$ denotes a ball of radius $r$ centered at $x$,
and $\bbB(x;r)_\delta := B(x;r) \cap (\delta\Z)^2$,
then (\ref{CHI1.30}) with $\Omega=B(0;1)$, $n=1$ and $x_1=0$ implies that
for some $C \in (0,\infty)$
\begin{equation}\label{Iineq1}
\E^+_{\bbB(0;1)_\delta}[\sigma_0]
\leq C \delta^{\frac{1}{8}} \qquad \mbox{for all } \delta \in (0,1).
\end{equation}
Then, for any $x\in \Omega$, by imposing $+$ boundary condition on the ball $B(x; r)$
with radius $r:=d(x, \partial \Omega)$ and applying the FKG inequality
\cite[Chapter 2]{G06}, we obtain
\begin{equation}\label{Iineq2}
\E^+_{\bbOmegadelta} [\sigma_x] \leq \E^+_{\bbB(x;r)_\delta}[\sigma_x]
= \E^+_{\bbB(0;r)_\delta}[\sigma_0] =
\E^+_{\bbB(0;1)_{\delta/r}}[\sigma_0] \leq C \frac{\delta^{\frac{1}{8}}}{r^{\frac{1}{8}}}
=  \frac{C\, \delta^{\frac{1}{8}}}{d(x, \partial \Omega)^{\frac{1}{8}}} ,
\end{equation}
where in the last inequality we applied \eqref{Iineq1}.

Relation \eqref{Iineq3} follows by applying
\eqref{Iineq2} and Lemma~\ref{L:decoupling} below, choosing $\bbOmega_i$
therein to be $\Omega_i\cap (\delta\Z)^2$,
where $\Omega_i$ is the ball centered at $x_i$ with
radius $\frac{1}{4} d(x_i, \partial\Omega\cup I\backslash\{x_i\})$.
\end{proof}

\begin{lemma}\label{L:decoupling}
Let $x_1, \ldots, x_n \in \bbOmega \subseteq \Z^2$ and suppose that
$x_i\in \bbOmega_i \subseteq \bbOmega$, with
$\bbOmega_i \cap (\bbOmega_j\cup\partial \bbOmega_j) =\emptyset$ for all $i\neq j$.
Then
\begin{equation}\label{decineq}
0\leq \E^+_{\bbOmega} [ \sigma_{x_1}\cdots \sigma_{x_n} ]
\leq \prod_{i=1}^n \E^+_{\bbOmega_i} [ \sigma_i].
\end{equation}
\end{lemma}

\begin{proof}
Relation \eqref{decineq} is a consequence of the Griffiths-Kelly-Sherman (GKS)
inequalities (see e.g.~\cite[Chapter V.3]{E06}).
We recall that $\P^+_{\bbOmega}$ denotes the Ising measure on $\{\pm 1 \}^{\Omega}$ with
inverse temperature $\beta \in (0,\infty)$ and zero external field, cf.\ \eqref{eq:Isingbasic}.
(The fact that $\beta=\beta_c$ is immaterial for this proof,
and we could even include a positive
external field in $\P^+_{\bbOmega}$.) Given $h = (h_x)_{x\in\bbOmega}$, let
$\P^{\rm free}_{\bbOmega, h}$ denote the Ising measure with external
(site-dependent) field $h$, i.e.
$$
\P^{\rm free}_{\bbOmega, h}(\sigma) =
\frac{1}{Z^{\rm free}_{\bbOmega, \beta, h}} \exp\Big\{ \sum_{x\sim y\in \bbOmega}
\beta\sigma_x \sigma_y + \sum_{x\in \bbOmega} h_x \sigma_x\Big\} \,.
$$
Since $\P^+_{\bbOmega}=\P^{\rm free}_{\bbOmega, h^+}$ with the choice
$h^+_x := \beta |\{y\in \partial \Lambda: y\sim x\}|$,
we may focus on $\P^{\rm free}_{\bbOmega, h}$.

Let $I:=\{x_1, \ldots, x_n\}$ and $\sigma_I:=\sigma_{x_1}\cdots \sigma_{x_n}$.
If $h \ge 0$ (that is, $h_x\geq 0$ for all $x\in \bbOmega$),
$$
\E^{\rm free}_{\bbOmega, h}[\sigma_I] \geq 0,
$$
by the first GKS inequality, proving the first bound in \eqref{decineq}.
Always for $h\ge 0$,
$$
\frac{\partial \E^{\rm free}_{\bbOmega, h}[\sigma_I]}{\partial h_y}
= \E^{\rm free}_{\bbOmega, h}[\sigma_I \sigma_y] -
\E^{\rm free}_{\bbOmega, h}[\sigma_I]\E^{\rm free}_{\bbOmega, h}[\sigma_y] \ge 0,
\qquad \forall y\in\bbOmega ,
$$
by the second GKS inequality. Therefore
$\E^{\rm free}_{\bbOmega, h}[\sigma_I]$ is increasing in
$h_y$ for every $y\in\bbOmega$. Starting with
$h=h^+$ and increasing $h_y$ to $+\infty$ for each $y\in \cup_{i=1}^n \partial \bbOmega_i$,
the resulting Ising measure is equivalent to imposing $+$ boundary condition on
$\cup_{i=1}^n \partial \bbOmega_i$.
Under this limiting measure, the distribution of the spin configurations on the disjoint
subdomains $(\bbOmega_i)_{1\leq i\leq n}$ factorizes, leading to
the second bound in  \eqref{decineq}.
\end{proof}

\begin{lemma}\label{L:fOmega}
For any $n\in\N$ and distinct $x_1, \ldots, x_n\in \Omega$ ,
set $I:=\{x_1, \ldots, x_n\}$
and define
\begin{equation}\label{fOmegadef}
f_\Omega(x_1,\ldots, x_n) := \prod_{i=1}^{n} 
\frac{1}{d(x_i, \partial\Omega\cup I\backslash\{x_i\})^{\frac{1}{8}}}.
\end{equation}
Then there exists $C=C(\Omega) <\infty$, such that for all $n\in\N$,
\begin{equation}\label{Iineq11}
\Vert f_\Omega\Vert^2_{L^2(\Omega^n)} \leq C^n (n!)^{\frac{1}{4}} \,.
\end{equation}
\end{lemma}

\begin{proof} To prove (\ref{Iineq11}), it suffices to show that for all $n\in\N$
\begin{equation}\label{Iineq5}
\Vert f_\Omega\Vert^2_{L^2(\Omega^n)} \leq C \, n^{\frac{1}{4}} 
\Vert f_\Omega\Vert^2_{L^2(\Omega^{n-1})},
\end{equation}
where $\Vert f_\Omega\Vert^2_{L^2(\Omega^0)}:=1$.

Note that in
\begin{equation}\label{Iineq6}
\Vert f_\Omega\Vert^2_{L^2(\Omega^n)} = \idotsint\limits_{\Omega^n}
\prod_{i=1}^n \frac{1}{d(x_i, \partial\Omega\cup I\backslash\{x_i\})^{\frac{1}{4}}}
{\rm d}x_1 \cdots {\rm d}x_n,
\end{equation}
we can divide the domain of integration $\Omega$ for $x_n$ into disjoint
open sets $\Omega_0, \ldots, \Omega_{n-1}$ (modulo a set of measure $0$), such that
\begin{equation}\label{eq:casesI}
\begin{aligned}
x_n\in \Omega_0 \quad  & \mbox{if and only if} \quad &d(x_n, \partial \Omega\cup I\backslash \{x_n\}) &= d(x_n, \partial \Omega),   \\
x_n\in \Omega_i \quad  & \mbox{if and only if} \quad &d(x_n, \partial \Omega\cup I\backslash \{x_n\}) &= d(x_n, x_i),  \quad 1\leq i\leq n-1.
\end{aligned}
\end{equation}

We next bound $f_\Omega(x_1, \ldots, x_n)$ in terms of
$f_\Omega(x_1, \ldots, x_{n-1})$. First consider the case
$x_n\in \Omega_0$. For each $1\leq i\leq n-1$,
either
\begin{equation}\label{eq:Iin7}
	d(x_i, \partial \Omega \cup I\backslash \{x_i\})
	= d(x_i, \partial \Omega \cup I'\backslash \{x_i\}),
\end{equation}
where $I' :=\{x_1, \ldots, x_{n-1}\}$, or
\begin{equation}\label{Iineq7}
d(x_i, \partial \Omega \cup I\backslash \{x_i\}) = d(x_i, x_n).
\end{equation}
In the later case, by the triangle inequality and the assumption $x_n\in \Omega_0$, we find that
\begin{equation}\label{eq:leadsto}
	d(x_i, \partial \Omega) \leq d(x_i, x_n) + d(x_n, \partial \Omega) \leq 2 d(x_i, x_n),
\end{equation}
and hence
\begin{equation}\label{eq:uffI}
	\frac{1}{d(x_i, \partial \Omega \cup I\backslash \{x_i\})} = \frac{1}{d(x_i, x_n)}
	\leq \frac{2}{d(x_i, \partial \Omega)} \leq \frac{2}{d(x_i, \partial \Omega
	\cup I'\backslash \{x_i\})}.
\end{equation}
If $N_0$ denotes the number of points among $\{x_1, \ldots, x_{n-1}\}$ such that (\ref{Iineq7}) holds, then
\begin{equation}\label{eq:analI}
f^2_\Omega(x_1, \ldots, x_n) \leq 2^{\frac{N_0}{4}} f^2_\Omega(x_1, \ldots, x_{n-1})
\frac{1}{d(x_n, \partial \Omega)^{\frac{1}{4}}} .	
\end{equation}
We claim that $N_0\leq 6$, which would then imply
\begin{equation}\label{Iineq8}
\idotsint\limits_{\Omega^{n-1} \times \Omega_0} f_\Omega^2(x_1, \ldots, x_n) {\rm d}x_1\cdots {\rm d}x_n \leq 2^{\frac{3}{2}} \Vert f_\Omega\Vert^2_{L^2(\Omega^{n-1})} \int_\Omega \frac{{\rm d}x_n}{d(x_n, \partial \Omega)^{\frac{1}{4}}},
\end{equation}
where the last integral is bounded by some constant $C_3(\Omega) < \infty$,
because $\Omega$ is assumed to be a bounded simply connected domain with a
piecewise smooth boundary.

To verify the claim that $N_0\leq 6$, assume without loss of generality that $x_1, \ldots, x_k$ are the points which satisfy (\ref{Iineq7}). In particular,
$d(x_i, x_n)\leq d(x_i, x_j)$ for all $1\leq i\neq j\leq k$. We may shift the origin to $x_n$ and assume without loss of generality that $x_i\in\R^2$ has polar coordinates $(r_i, \theta_i)$, and the directional vectors $e^{i\theta_1}, \ldots, e^{i\theta_k}$ are ordered counter clockwise on the unit circle. For any two adjacent $e^{i\theta_j}$ and $e^{i\theta_{j+1}}$ on the unit circle, in order to satisfy
$$
\max\{d(x_j, 0), d(x_{j+1}, 0)\} \leq d(x_j, x_{j+1}),
$$
it is necessary that the angle between $e^{i\theta_j}$ and $e^{i\theta_{j+1}}$ is at least $\frac{\pi}{3}$. It then follows that there can be at most $6$ such points.

We now consider the case $x_n\in \Omega_i$ for $1\leq i\leq n-1$
(recall \eqref{eq:casesI}). Without loss of generality, assume that $x_n\in \Omega_1$.
By the same reasoning as above, for each $1 \le i \le n-1$, either
relation \eqref{eq:Iin7} or relation \eqref{Iineq7} holds; in the latter case
we can replace \eqref{eq:leadsto} by
\begin{equation*}
	d(x_i, x_1) \le d(x_i, x_n) + d(x_n, x_1) \le 2 d(x_i, x_n) \,,
\end{equation*}
because $x_n \in \Omega_1$. Thus for $i \ge 2$
relation \eqref{eq:uffI} still holds
if we replace $d(x_i, \partial\Omega)$ by $d(x_i, x_1)$ therein.
The case $i=1$ needs to be dealt with separately: for this we simply bound
$$
\begin{aligned}
\frac{1}{d(x_1, \partial \Omega \cup I \backslash \{x_1\})^{\frac{1}{4}}} & \leq \frac{1}{d(x_1, \partial \Omega \cup I' \backslash \{x_1\})^{\frac{1}{4}}} + \frac{1}{d(x_1, x_n)^{\frac{1}{4}}}  \\
& =
\frac{1}{d(x_1, \partial \Omega \cup I' \backslash \{x_1\})^{\frac{1}{4}}}
\bigg(1+ \frac{d(x_1, \partial \Omega \cup I' \backslash \{x_1\})^{\frac{1}{4}}}
{d(x_1, x_n)^{\frac{1}{4}}}\bigg).
\end{aligned}
$$
We thus obtain the following analogue of \eqref{eq:analI}
(with $N_0 \le 6$) when $x_n \in \Omega_1$:
$$
f^2_\Omega(x_1, \ldots, x_n) \leq 2^{\frac{3}{2}} f^2_\Omega(x_1, \ldots, x_{n-1})
\Bigg\{\bigg(1+\frac{d(x_1, \partial \Omega\cup I'\backslash\{x_1\})^{\frac{1}{4}}}
{d(x_1, x_n)^{\frac{1}{4}}}\bigg) \frac{1}{d(x_n, x_1)^{\frac{1}{4}}} \Bigg\} .
$$
Bounding the term in brackets by
$C_4(\Omega) / d(x_1, x_n)^{\frac{1}{2}}$, for some $C_4(\Omega) < \infty$
(recall that $\Omega$ is a bounded set), we obtain
\begin{align}
\idotsint\limits_{\Omega^{n-1} \times \Omega_1} f_\Omega^2(x_1, \ldots, x_n)
{\rm d}x_1\cdots {\rm d}x_n & \leq 2^{\frac{3}{2}} \, C_4(\Omega)
\idotsint\limits_{\Omega^{n-1}\times \Omega_1}  \frac{f^2_\Omega(x_1, \ldots, x_{n-1})}
{d(x_n, x_1)^{\frac{1}{2}}} {\rm d}x_1\cdots {\rm d}x_{n} \nonumber\\
& \leq C_5(\Omega) \, |\Omega_1|^{\frac{3}{4}} \Vert f_\Omega\Vert^2_{L^2(\Omega^{n-1})},
\label{Iineq9}
\end{align}
for some $C_5(\Omega) < \infty$, where we applied the Hardy-Littlewood
rearrangement inequality (see e.g.~\cite[Theorem~3.4]{LL01})
to bound
$$
\int_{\Omega_1} \frac{ {\rm d}x_n}{d(x_n, x_1)^{\frac{1}{2}}} \leq \int_{\Omega_1^*} \frac{{\rm d}x_n}{d(x_n, 0)^{\frac{1}{2}}} = \int_0^{r_*} \frac{2\pi r }{r^{\frac{1}{2}}} {\rm d}r = \frac{4}{3}\pi^{\frac{1}{4}} |\Omega_1|^{\frac{3}{4}},
$$
where $\Omega_1^*$ is the ball centered at the origin with the same area $|\Omega_1^*|= \pi r_*^2$ as $\Omega_1$.

Combining (\ref{Iineq8}) and (\ref{Iineq9}), and the analogue for
$x_n\in\Omega_i$ with $2\leq i\leq n-1$, we obtain
\begin{align}
\Vert f_\Omega\Vert^2_{L^2(\Omega^n)} &\leq
C_6(\Omega)\big(1+ |\Omega_1|^{\frac{3}{4}}+\cdots + |\Omega_{n-1}|^{\frac{3}{4}}\big)
\Vert f_\Omega\Vert^2_{L^2(\Omega^{n-1})} \nonumber \\
&\leq C_6(\Omega) \Big(1+ (n-1)\Big( \frac{|\Omega_1|+\cdots |\Omega_{n-1}|}{n-1}\Big)^{\frac{3}{4}}
\Big) \Vert f_\Omega\Vert^2_{L^2(\Omega^{n-1})} \nonumber\\
&\leq C_2(\Omega) \, n^{\frac{1}{4}} \, \Vert f_\Omega\Vert^2_{L^2(\Omega^{n-1})}, \label{Iineq10}
\end{align}
where we applied Jensen's inequality to the function $g(x)=x^{\frac{3}{4}}$. This establishes (\ref{Iineq5}) and concludes the proof of Theorem~\ref{T:RFIM}.
\end{proof}

\medskip

\begin{proof}[Proof of Corollary~\ref{C:coco}]
Let $\varphi'(\cdot)$ denote the complex derivative of the conformal map $\varphi: \tilde\Omega
\to \Omega$. Since $|\varphi'(z)|^2$ equals the Jacobian determinant of $\varphi$,
for all $f, g \in L^2(\Omega)$ we have
\begin{equation} \label{eq:covgau}
	\int_\Omega f(x) g(x) \dd x = \int_{\tilde\Omega}
	f(\varphi(z)) g(\varphi(z)) |\varphi'(z)|^2 \dd z \,,
\end{equation}
by the change of variables formula. As a consequence, if $W(\cdot)$ denotes white noise on $\R^2$,
the processes $(\int_\Omega f(x) W(\dd x))_{f \in L^2(\Omega)}$ and
$(\int_{\tilde\Omega} f(\varphi(z)) |\varphi'(z)| \, W(\dd z))_{f \in L^2(\Omega)}$
have the same distribution: they are both centered Gaussian processes with
the same covariance \eqref{eq:covgau}.
This extends to an equality in distribution for multiple integrals (recall Subsection~\ref{sec:wnn}):
\begin{equation*}
	\idotsint_{\Omega^n}\!\! f(x_1, \ldots, x_n)
	\prod_{i=1}^n  W(\dd x_i) \overset{d}{=}
	\idotsint_{\tilde\Omega^n}\!\! f(\varphi(z_1), \ldots, \varphi(z_n))
	\prod_{i=1}^n \big[ |\varphi'(z_i)| \, W(\dd z_i) \big] \,,
\end{equation*}
jointly for $n\in\N$ and symmetric $f \in L^2(\Omega^n)$.
Informally, we have $W(\dd \varphi(z)) \overset{d}{=}
|\varphi'(z)| W(\dd z)$, which is the stochastic analogue of
$\dd \varphi(z) = |\varphi'(z)|^2 \dd z$.
Recalling \eqref{Z+Wexp}, it follows that
\begin{equation} \label{eq:plugin}
\begin{split}
\bZ^{+, W}_{\Omega, \hlambda, \hh} \overset{d}{=}
1 + \sum_{n=1}^\infty \frac{\cC^n}{n!} \idotsint_{\tilde\Omega^n}
\bphi_\Omega^+\big(\varphi(z_1), \ldots, \varphi(z_n)\big)
\prod_{i=1}^n \, \big[& \hlambda\big(\varphi(z_i)\big) \,  |\varphi'(z_i)| \, W({\rm d} z_i) \\
& \, + \hh\big(\varphi(z_i)\big) \, |\varphi'(z_i)|^2 \, {\rm d}z_i\big].
\end{split}
\end{equation}
By~\cite[Theorem 1.3]{CHI12}, the function $\bphi_\Omega^+$ is conformally covariant with
\begin{equation}
\bphi_\Omega^+\big(\varphi(z_1), \ldots, \varphi(z_n)\big) =
\bphi_{\tilde \Omega}^+(z_1, \ldots, z_n) \prod_{i=1}^n |\varphi'(z_i)|^{-\frac{1}{8}},
\end{equation}
hence $\bZ^{+, W}_{\Omega, \hlambda, \hh} \overset{d}{=}
\bZ^{+, W}_{\tilde \Omega, \tilde \lambda, \tilde h}$, with
$\tilde \lambda(z) := |\varphi'(z)|^{\frac{7}{8}} \, \hlambda (\varphi(z)) $
and $\tilde h(z) := |\varphi'(z)|^{\frac{15}{8}} \, \hh (\varphi(z)) $.
\end{proof}
\appendix

\section{The Cameron-Martin shift}

\label{sec:wnapp}

Recalling Subsection~\ref{sec:wnn},
let $W = (W(f))_{f \in L^2(\R^d)}$ be a white noise on $\R^d$
defined on the probability space $(\Omega_W, \cA, \bbP)$.
We denote by $L^0 := L^0(\Omega_W, \sigma(W), \bbP)$ the
space of (equivalence classes of) a.s.\ finite random variables that are measurable with respect
to the $\sigma$-algebra generated by $W$, equipped with the
topology of convergence in probability. Note that all the multi-dimensional
stochastic integrals $W^{\otimes k}(f)$ belong to $L^0$.

Let us now fix $\bnu \in L^2(\R^d)$, representing the \emph{bias}.
Given $k \in \N$ and a symmetric square-integrable function $f: (\R^d)^k \to \R$,
the ``biased stochastic integral''
\begin{equation} \label{eq:biased}
	W_{\bnu}^{\otimes k}(f) = \idotsint_{(\R^d)^k} f(x_1, \ldots, x_k) \,
	\prod_{i=1}^k \Big( W(\dd x_i) + \bnu(x_i) \dd x_i \Big)
\end{equation}
was defined in Remark~\ref{rem:properdef} by expanding the product
and integrating out the deterministic variables corresponding to $\bnu(x_i) \dd x_i$,
thus reducing to a sum of lower-dimensional ordinary (unbiased) stochastic integrals.
In particular, for $k=1$ we can write
$W_{\bnu}(f) = W_{\bnu}^{\otimes 1}(f)$ as
\begin{equation}\label{eq:k=1}
	W_{\bnu}(f) := W(f) + \int_{\R^d} f(x) \bnu(x) \dd x = W(f) +
\bbE[W(f) \, \xi] \,, \qquad \text{with} \quad \xi := W(\bnu) \,,
\end{equation}
by the It\^o isometry \eqref{eq:WI}.

Thus $W_{\bnu}(f) = \rho_\xi(W(f))$, where
we define the map $\rho_\xi(X) := X + \bbE[X\xi]$ for every
one-dimensional stochastic integral $X$.
By \cite[Theorem 14.1]{J97}, such a map admits
a unique extension $\rho_\xi: L^0 \to L^0$, called the \emph{Cameron-Martin shift},
which is continuous, linear and satisfies
\begin{equation} \label{eq:homo}
	\rho_\xi(1) = 1 \,, \qquad
	\rho_\xi(XY) = \rho_\xi(X) \rho_\xi(Y) \quad \forall X,Y \in L^0 \,.
\end{equation}
As a consequence, the multi-dimensional biased stochastic integrals
\eqref{eq:biased} correspond to
\begin{equation}\label{eq:kgen}
	W_{\bnu}^{\otimes k}(f) = \rho_\xi(W^{\otimes k}(f)) \,.
\end{equation}
This is easily checked for ``special simple functions'' $f$ (recall
Subsection~\ref{sec:wnn})
using \eqref{eq:k=1}-\eqref{eq:homo}, and
is extended to general symmetric $f \in L^2((\R^d)^k)$ using the continuity of $\rho_\xi$.

For any $X\in L^0$, the random variable $\rho_\xi(X)$ has
the same distribution as $X$ under the probability $\bbP_{\bnu}$
defined in \eqref{eq:RN},
by \cite[Theorem 14.1 (iii)-(iv)]{J97}. In particular, choosing for $X$
the series in \eqref{eq:general2}, whenever it converges in probability,
one obtains relation \eqref{eq:general3}.

\section{Technical lemmas}
\label{sec:appB}

The following lemma is used in the proof of Theorem~\ref{th:general2}.

\begin{lemma}[Exponential tilting] \label{L:tilt}
Let $(\zeta_{\delta,i})_{\delta \in (0,1), i\in\bbT_\delta}$ be a family of independent
random variables in $L^2$, with mean
$\mu_{\delta}(i)$ and variance $\sigma_\delta^2$, with
$\sigma_\delta \to\bsigma_0 \in (0,\infty)$
and $\|\mu_\delta\|_\infty := \sup_{i\in\bbT_\delta} |\mu_\delta(i)| \to 0$
as $\delta \downarrow 0$.
Assume further that $(\zeta_{\delta,i}^2)_{\delta\in (0,1),i \in\bbT_\delta}$
are uniformly integrable.
Then one can construct independent random variables
$(\tilde \zeta_{\delta,i})_{\delta \in (0,1), i\in\bbT_\delta}$ such that
\begin{equation*}
	\bbP(\tilde \zeta_{\delta,i}\in \dd x) = f_{\delta,i}(x) \bbP(\zeta_{\delta,i}\in \dd x) \,,
\end{equation*}
and there
exist $\delta_0,C \in (0,\infty)$, and $C_p  \in (0,\infty)$
for every $p\in\R$, such that
\begin{equation}\label{A:tiltbd1}
\bbE[\tilde \zeta_{\delta,i}]  =0, \quad \bbE[\tilde \zeta^2_{\delta,i}] \leq
\sigma_\delta^2 \big(1 + C|\mu_{\delta}(i)|\big), \quad
\mbox{and}\quad \bbE[f_{\delta,i}(\zeta_{\delta,i})^p]  \leq 1+ C_p \,\mu_{\delta}(i)^2 \,,
\end{equation}
for all $\delta \in (0,\delta_0)$ and $i\in\bbT_\delta$.
Furthermore, if
\begin{equation}\label{zetaassum}
\inf\limits_{\delta\in (0,1), i\in\bbT_\delta} \min\big\{\bbP(\zeta_{\delta,i}>0),
\bbP(\zeta_{\delta,i}<0), {\rm \bbV ar}(\zeta_{\delta, i}|\zeta_{\delta, i}>0),
{\rm \bbV ar}(\zeta_{\delta, i}|\zeta_{\delta, i}<0) \big\} >0,
\end{equation}
then there exists $C' \in (0,\infty)$ such that
the following improved bound holds:
\begin{equation}\label{A:tiltbd2}
\bbE[\tilde \zeta^2_{\delta,i}]\leq \sigma_\delta^2\big(1 + C' \mu_\delta(i)^2\big) .
\end{equation}
\end{lemma}

The proof of Lemma~\ref{L:tilt} is an easy corollary of the following general result,
which concerns exponential tilting of a single random variable in order
to shift its mean to zero (since the random variables are not assumed to have
finite exponential moments, the tilting is performed on a bounded
subset). The assumptions in Lemma~\ref{L:tilt} guarantee
that conditions \eqref{eq:assEX} and \eqref{eq:assEXbis} are fulfilled,
and the constants in \eqref{eq:RN}, \eqref{eq:uffest},
\eqref{eq:improvedbound} are uniformly bounded.

\begin{theorem} \label{th:new}
Let $X$ be a square-integrable random variable
and let $A  > 0$ be such that
\begin{equation} \label{eq:defA}
	\bbE[X^2 \ind_{\{|X| > A\}}] \le \frac{1}{4} \bbE[X^2]  \,.
\end{equation}
Assume that $\bbE[X]$ is sufficiently small, more precisely,
\begin{equation}\label{eq:assEX}
	|\bbE(X)| \le \epsilon := \frac{\bbE[X^2]^2}{144\,A^3} \,.
\end{equation}
Then one can define a random variable $\tilde X$, such that
$\bbP(\tilde X \in \dd x) = f(x) \, \bbP(X \in \dd x)$, satisfying
\begin{gather}	\label{eq:RNbis}
	\quad \ \bbE\big[ f(X)^p \big] \le
	1 + C_p \, \bbE[X]^2 \, \quad \forall p \in \R \,,
	\qquad \text{with } \ C_p := \frac{4 e^{|p|}}{A \epsilon}\,, \\
	\label{eq:uffest}
	\bbE[\tilde X] = 0 \,, \qquad
	\bbE[\tilde X^2] \le \bbE[X^2] + C \, |\bbE[X]| \,,
	\qquad
	\text{with } \ C := \frac{A^{3/2}}{\sqrt{\epsilon}} \,.
\end{gather}

If $\bbE[X] \ge 0$, and $A$ is chosen such that
\begin{equation} \label{eq:defAbis}
	\bbE[X^2 \ind_{\{X > A\}}] \le \frac{1}{4} \bbE[X^2 \ind_{\{X \ge 0\}}] \,,
\end{equation}
$($replace $X$ by $-X$ if $\bbE[X] \le 0$$)$, and further assume that
\begin{equation}\label{eq:assEXbis}
	|\bbE[X|X\ge 0]| \le \epsilon' := \frac{\bbE[X^2|X \ge 0]^2}{144\,A^3} \,,
\end{equation}
then we can define $\tilde X$ such that \eqref{eq:RNbis} holds with $\epsilon$ replaced
by $\epsilon'$, while \eqref{eq:uffest} is improved to
\begin{equation}\label{eq:improvedbound}
	\bbE[\tilde X] = 0 \,, \qquad
	\bbE[\tilde X^2] \le \bbE[X^2] + C'\, \bbE[X]^2 \,,
	\qquad \text{with } \ C':= \frac{A}{2 \bbP(X \ge 0)\epsilon'} \,.
\end{equation}
\end{theorem}

\begin{proof}
Without loss of generality, we assume that $\E[X] \ge 0$ (otherwise consider $-X$).

\medskip
\noindent
\textbf{Step 1 (Strategy).}
We will fix $I \subseteq \R$, which is either $[-A,A]$ (assuming \eqref{eq:defA}-\eqref{eq:assEX}), or $[0,A]$ (assuming \eqref{eq:defAbis}-\eqref{eq:assEXbis}), and we define random variables $Y$, $Z$ with laws
\begin{equation} \label{eq:YZ}
	\bbP(Y \in \cdot) := \bbP(X \in \cdot \,|\, X \in I) \,, \qquad
	\bbP(Z \in \cdot) := \bbP(X \in \cdot \,|\, X \not\in I) \,.
\end{equation}
Taking independent copies of $X,Y, Z$, we have the following equality in distribution:
\begin{equation} \label{eq:Xdist0}
	X \overset{\mathrm{dist}}{=} \ind_{\{X \in I\}} \, Y
	\,+\, \ind_{\{X \not\in I\}} \, Z \,,
	\qquad \text{(note that}  \quad |Y| \le A \text{)}.
\end{equation}
We then exponentially tilt $Y$, defining for $\lambda\in\R$ a random variable $Y_\lambda$ with law
\begin{equation}\label{FNlambda0}
	\bbP(Y_\lambda \in \dd x) := e^{\lambda x - F(\lambda)} \,
	\bbP(Y \in \dd x) \,, \qquad \text{where} \qquad
	F(\lambda):= \log \bbE[e^{\lambda Y}] \,.
\end{equation}
As we will show at the end of the proof,
we can choose $\lambda = \tilde\lambda \in \R$ such that
\begin{equation}\label{eq:constru0}
	\bbE[Y_{\tilde\lambda}] = -\frac{\bbP(X\not\in I)}{\bbP(X\in I)} \bbE[Z] \,.
\end{equation}
If we define $\tilde X$ replacing $Y$ by $Y_{\tilde\lambda}$ in the definition \eqref{eq:Xdist0}, that is
\begin{equation} \label{eq:Xdist0new}
	\tilde X := \ind_{\{X \in I\}} \, Y_{\tilde\lambda}
	\,+\, \ind_{\{X \not\in I\}} \, Z \,,
\end{equation}
then $\bbE[\tilde X] = 0$ by construction.
Moreover $\bbP(\tilde X \in \dd x) = f(x) \bbP(X \in \dd x)$ with density
\begin{equation} \label{eq:effe}
	f(x) = e^{\tilde\lambda x - F(\tilde\lambda)} \, \ind_{\{x\in I\}} \,+\,
	\ind_{\{x \in \R \setminus I\}}
	= 1 + \big(e^{\tilde\lambda x - F(\tilde\lambda)} - 1 \big)\ind_{\{x\in I\}} \,.
\end{equation}

The rest of the proof is devoted to estimating $\bbE[\tilde X^2]$ and $\bbE[f(X)^p]$.
We are going to use the following bounds on $\tilde\lambda$, which will be proved at the very end:
\begin{align} \label{eq:mart1}
	&\text{$I := [-A,A]$, assuming \eqref{eq:defA}-\eqref{eq:assEX}:} \ \
	& |\tilde\lambda| &
	\le  \frac{|\bbE[X]|}{2 A^{3/2} \sqrt{\epsilon}} \, ; \\
	\label{eq:mart2}
	&\text{$I := [0,A]$, assuming \eqref{eq:defAbis}-\eqref{eq:assEXbis}:} \ \
	&|\tilde\lambda| &
	\le \frac{|\bbE[X]|}{\sqrt{6} A^{3/2}
	\sqrt{\epsilon'}\sqrt{\bbP(X\in I) \bbP(X\ge 0)}} \,; \\
	\label{eq:lambdabound}
	&\text{In either case:} \ \
	&|\tilde\lambda| & \le \frac{1}{27 A} \,.
\end{align}

\medskip
\noindent
\textbf{Step 2 (Bounds on $\bbE[\tilde X^2]$).}
Denote $G(\lambda) := \bbE[Y_\lambda^2] =
\bbE[Y^2 \, e^{\lambda Y}]/\bbE[e^{\lambda Y}]$.
Recalling \eqref{eq:Xdist0new} and \eqref{eq:Xdist0}, we can write
\begin{equation}\label{eq:resu}
\begin{split}
	\bbE[\tilde X^2] & = \bbE[X^2] + \bbP(X\in I)
	\big( \bbE[Y_{\tilde\lambda}^2] - \bbE[Y^2] \big)
	= \bbE[X^2] + \bbP(X\in I) \int_0^{\tilde\lambda} G'(\lambda) \, \dd \lambda \,.
\end{split}
\end{equation}
Since $G'(\lambda) = \bbE[(Y_\lambda)^3]
- \bbE[(Y_\lambda)^2] \bbE[Y_\lambda]$ and $|Y_\lambda| \le A$,
we have $|G'(\lambda)| \le 2 A^3$, and hence
\begin{equation} \label{eq:eheh}
	\bbE[\tilde X^2] \le \bbE[X^2] + 2A^3 \bbP(X\in I) |\tilde\lambda| .
\end{equation}
Applying \eqref{eq:mart1}, we obtain precisely the second bound in \eqref{eq:uffest}.

To prove \eqref{eq:improvedbound}, let us assume \eqref{eq:defAbis}-\eqref{eq:assEXbis} and set $I := [0,A]$.
By \eqref{eq:Xdist0}--\eqref{eq:constru0}, we have
\begin{equation} \label{eq:equivconstru0}
	\bbE[Y_{\tilde\lambda}] = \bbE[Y] - \frac{\bbE[X]}{\bbP(X\in I)} \,,
\end{equation}
and hence $\bbE[Y_{\tilde\lambda}]  \le \bbE[Y]$.
Since $\bbE[Y_\lambda] = F'(\lambda)$ is increasing in $\lambda$
(because $F''(\lambda) = \bbvar[Y_\lambda] \ge 0$), it follows that
$\tilde\lambda \le 0$.
We then refine \eqref{eq:resu} as follows:
\begin{equation} \label{eq:daje}
\begin{split}
	\bbE[\tilde X^2]  = \bbE[X^2] + \bbP(X\in I) \tilde\lambda G'(0)
	+ \bbP(X\in I)
	\int_0^{\tilde\lambda} \bigg( \int_0^\lambda G''(s) \, \dd s \bigg) \, \dd \lambda \, \,.
\end{split}
\end{equation}
Note that $G'(0) = \bbE[Y^3] - \bbE[Y^2] \bbE[Y] \ge 0$, because $Y \in [0,A]$
and hence $Y^2$ and $Y$ are positively correlated. Therefore the second term in 
\eqref{eq:daje} is bounded by $0$. Also note that
\begin{equation*}
	G''(\lambda)
	= \bbE[(Y_\lambda)^4] - 2 \bbE[(Y_\lambda)^3] \bbE[Y_\lambda]
	+ 2 \bbE[(Y_\lambda)^2] \bbE[Y_\lambda]^2
	- \bbE[(Y_\lambda)^2]^2 \,,
\end{equation*}
and hence $|G''(\lambda) | \le 6 A^4$. Substituting into \eqref{eq:daje} then yields
\begin{equation} \label{eq:concl2}
	\bbE[\tilde X^2]  \le \bbE[X^2]
	+ 3 A^4 \bbP(X\in I)  \tilde\lambda^2 \,.
\end{equation}
Applying \eqref{eq:mart2}, we obtain precisely \eqref{eq:improvedbound}.

\medskip
\noindent
\textbf{Step 3 (Bounds on $\bbE[f(X)^p]$).}
Recall $f$ from \eqref{eq:effe} and $F(\lambda)$ from \eqref{FNlambda0}.
Since $F(0)=0$ and $|F'(\lambda)| = |\bbE[Y_\lambda]| \le A$,
cf. \eqref{FNlambda0},
we have $|F(\tilde\lambda)| \le A |\tilde\lambda|$ and
hence $|\tilde\lambda x - F(\tilde\lambda)| \le 2A|\tilde\lambda|$
for every $x\in I \subseteq [-A,A]$.
Applying \eqref{eq:lambdabound}, we obtain
\begin{equation} \label{eq:ab}
	 e^{-2/27} \le f(x) \le
	 e^{2/27} \,, \qquad \forall x \in \R \,.
\end{equation}
For any $p\in\R$, Taylor expansion gives $y^p \le 1 + p(y-1) + C'_p(y-1)^2$
for all $y \in [e^{-2/27},e^{2/27}]$, with
\begin{equation} \label{eq:Cp}
	C'_p := \max_{y \in [e^{-2/27},e^{2/27}]} |p(p-1)y^{p-2}| =
	|p(p-1)| e^{\frac{2}{27}|p-2|} \le 2 \,e^{|p|} \,.
\end{equation}
Therefore
\begin{equation*}
\begin{split}
	f(x)^p & \le 1 + p(f(x)-1) + C'_p(f(x)-1)^2 \\
	& = 1 + (p-2C'_p)(f(x)-1)
	+ C'_p \big(e^{2\tilde\lambda x - 2F(\tilde\lambda)} - 1 \big)\ind_{\{x\in I\}} \,.
\end{split}
\end{equation*}
Since $f$ is a probability density, recalling the definition of $F$ from \eqref{FNlambda0}, we obtain
\begin{equation} \label{eq:Efp}
	\bbE\big[ f(X)^p \big] \le 1 + C'_p \, \bbP(X \in I)
	\big(e^{F(2\tilde\lambda) - 2 F(\tilde\lambda)} - 1 \big) \,.
\end{equation}
Since $|F''(\lambda)| = |\bbE[Y_\lambda^2] - \bbE[Y_\lambda]^2| \le 2A^2$,
the Mean Value Theorem and \eqref{eq:lambdabound} yield
\begin{equation} \label{eq:F22F}
	0 \le F(2\tilde\lambda) - 2F(\tilde\lambda)
	= \big(F(2\tilde\lambda) - F(\tilde\lambda)\big)
	- \big(F(\tilde\lambda)- F(0)\big) \le 4 A^2 \tilde\lambda^2
	\le \frac{4}{27^2} \le 1 \,,
\end{equation}
where the first inequality holds by convexity of $F$
(note that $F''(\lambda) = \bbvar[Y_\lambda] \ge 0$).

Consider first the case $I=[-A,A]$, assuming
\eqref{eq:defA}-\eqref{eq:assEX}: since
$e^x-1 \le 2x$ for $0 \le x \le 1$, applying \eqref{eq:F22F}, \eqref{eq:mart1}
and \eqref{eq:Cp} we obtain
\begin{equation} \label{eq:RNfi}
\begin{split}
	\bbE\big[ f(X)^p \big] &
	\le 1 + 2 C'_p \big( F(2\tilde\lambda) - 2F(\tilde\lambda) \big)
	\le 1 + C'_p \, 8 A^2 \tilde\lambda^2
	\le 1 + \frac{4 e^{|p|}}{A \epsilon} \bbE[X]^2 \,,
\end{split}
\end{equation}
proving \eqref{eq:RNbis}.

The case $I=[0,A]$, assuming \eqref{eq:defAbis}-\eqref{eq:assEXbis}, is similar:
we keep the term $\bbP(X\in I)$ in \eqref{eq:Efp} when writing \eqref{eq:RNfi},
so that applying \eqref{eq:mart2} gives 
\begin{equation*}
	\bbE\big[ f(X)^p \big]
	\le 1 + C'_p \, \bbP(X \in I) \,
	8 A^2 \tilde\lambda^2 \le
	1 + \frac{16 e^{|p|}}{6 A \epsilon'} \bbE[X]^2
	\le 1 + \frac{4 e^{|p|}}{A \epsilon'} \bbE[X]^2 \,,
\end{equation*}
which coincides with \eqref{eq:RNbis}, where $\epsilon$ is replaced by $\epsilon'$.

\medskip
\noindent
\textbf{Step 4 (Bounds on $\tilde\lambda$).}
We finally show that one can choose $\lambda = \tilde \lambda$
so that \eqref{eq:constru0} holds and
the bounds \eqref{eq:mart1}-\eqref{eq:mart2}-\eqref{eq:lambdabound} are satisfied.

Since $F'(\lambda) = \bbE[Y_\lambda]$, cf.\ \eqref{FNlambda0},
we can rewrite \eqref{eq:constru0}, cf.\ \eqref{eq:equivconstru0}, as
\begin{equation} \label{eq:uao}
	F'(\tilde\lambda) - F'(0) = \tilde x \,, \qquad
	\text{where} \qquad \tilde x := - \frac{\bbE[X]}{\bbP(X \in I)} \,.
\end{equation}
Since $F'''(\lambda) = \bbE[Y_\lambda^3]
-3\bbE[Y_\lambda] \bbE[Y_\lambda^2] + 2 \bbE[Y_\lambda]^3$ and
$|Y_\lambda| \le A$, we have $|F'''(\lambda)| \le 6 A^3$. Therefore
\begin{equation*} 
	F''(\lambda)\geq F''(0) - 6 A^3 |\lambda| \ge \frac{F''(0)}{2}
	\qquad \text{for all} \qquad
	|\lambda| \le c := \frac{F''(0)}{12 A^3} = \frac{\bbvar(Y)}{12 A^3}  \,.
\end{equation*}
In particular, equation \eqref{eq:uao}
has exactly one solution $\tilde\lambda \in [-c,c]$, provided
\begin{equation} \label{eq:tocheckxtilde}
	|\tilde x| \le \frac{F''(0)}{2} c \,, \qquad
	\text{that is} \qquad
	|\bbE[X]| \le \frac{\bbP(X\in I)\bbvar(Y)^2}{24 A^3} \,,
\end{equation}
in which case  $\tilde\lambda$ satisfies
\begin{equation}\label{lambdabar}
|\tilde\lambda| \leq \frac{|\tilde x|}{\frac{1}{2}F''(0)} =
\frac{2|\bbE[X]|}{\bbP(X\in I) \bbvar(Y)} \,.
\end{equation}
It only remains to check that condition \eqref{eq:tocheckxtilde}
is indeed satisfied, under either assumptions \eqref{eq:defA}-\eqref{eq:assEX}
or \eqref{eq:defAbis}-\eqref{eq:assEXbis}, and to show that
\eqref{lambdabar} yields
the bounds \eqref{eq:mart1}-\eqref{eq:mart2}-\eqref{eq:lambdabound}.
For this we need to estimate $\bbP(X\in I)$ and $\bbvar(Y)$.

In order to avoid repetitions, let $\bbP^*$ denote the
original law $\bbP$ when we assume
\eqref{eq:defA}-\eqref{eq:assEX} or the conditional law $\bbP(\,\cdot\,|\,X \ge 0)$
when we assume \eqref{eq:defAbis}-\eqref{eq:assEXbis}.
In either case $I= [-A,A]$ or $I = [0,A]$ we can write
$\bbP(Y \in \cdot\, ) := \bbP(X \in \cdot \,|\, X \in I)
= \bbP^*(X \in \cdot\, |\, |X| \le A)$, therefore
\begin{equation}\label{eq:varY}
	\bbvar(Y) = \bbvar^*(X \,|\, |X| \le A) \,.
\end{equation}
Note that assumptions \eqref{eq:defA}-\eqref{eq:assEX} and \eqref{eq:defAbis}-\eqref{eq:assEXbis}
can be written as follows:
\begin{equation} \label{eq:defAnew}
	\bbE^*[X^2 \ind_{\{|X| > A\}}] \le \frac{1}{4} \bbE^*[X^2] \,,
	\qquad |\bbE^*(X)| \le
	\frac{\bbE^*[X^2]^2}{144\,A^3} \,.
\end{equation}
Since $\bbE^*[X^2 \ind_{\{|X| \le A\}}] \le
A^2 \bbP^*(|X| \le A)$, it follows that
\begin{equation}\label{eq:easyboundA}
	A^2 \ge \frac{3}{4} \frac{\bbE^*[X^2]}{\bbP^*(|X| \le A)}
	\ge \frac{3}{4} \bbE^*[X^2] \,.
\end{equation}
We thus get
\begin{gather} \label{eq:p23}
	\bbP^*(|X| \le A) = 1 - \bbP^*(|X|> A)
	\ge 1- \frac{\bbE^*[X^2 \ind_{\{|X| > A\}}]}{A^2}
	\ge 1- \frac{\bbE^*[X^2]}{4A^2} \ge \frac{2}{3} \,,
	\\ \nonumber
	\big| \bbE^*[X \ind_{\{|X| > A\}}] \big|
	\le \frac{\bbE^*[X^2 \ind_{\{|X| > A\}}] }{A}
	\le \frac{\bbE^*[X^2] }{4A}
	\le \frac{\sqrt{\bbE^*[X^2]}}{2\sqrt{3}} \,.
\end{gather}
Together with \eqref{eq:defAnew} and \eqref{eq:easyboundA}, this gives 
\begin{equation*}
	\big| \bbE^*[X \ind_{\{|X| \le A\}}] \big|
	\le \big| \bbE^*[X] \big|
	+ \big| \bbE^*[X \ind_{\{|X| > A\}}] \big|
	\le \bigg( \frac{(\frac{4}{3})^{3/2}}{144} + \frac{1}{2\sqrt{3}} \bigg) \sqrt{\bbE^*[X^2]}
	\le \frac{1}{3} \sqrt{\bbE^*[X^2]} \,,
\end{equation*}
which yields
\begin{gather*}
	\big| \bbE^*[X \,|\, |X| \le A] \big| = \frac{\big| \bbE^*[X \ind_{\{|X| \le A\}}] \big|}
	{\bbP^*(|X| \le A)}
	\le \frac{1}{2} \sqrt{\bbE^*[X^2]}.
\end{gather*}
Applying one more time \eqref{eq:defAnew} we get
\begin{equation*}
	\bbE^*[X^2 \,|\, |X| \le A]
	= \frac{\big| \bbE^*[X^2 \ind_{\{|X| \le A\}}] \big|}
	{\bbP^*(|X| \le A)} \ge \frac{3}{4} \bbE^*[X^2]  \,,
\end{equation*}
which finally yields, cf.\ \eqref{eq:varY},
\begin{equation} \label{eq:varYfin}
	\bbvar(Y) = \bbE^*[X^2 \,|\, |X| \le A]
	- \bbE^*[X \,|\, |X| \le A]^2 \ge \frac{1}{2} \bbE^*[X^2] \,.
\end{equation}

By \eqref{eq:assEX} and \eqref{eq:assEXbis}, and applying \eqref{eq:varYfin} and \eqref{eq:p23}, 
we obtain
\begin{equation} \label{eq:analog}
	| \bbE^*[X] | \le \frac{\bbE^*[X^2]^2}{144\,A^3}
	\le \frac{\bbvar(Y)^2}{36 \, A^3}
	\le \frac{\bbP^*(|X| \le A) \, \bbvar(Y)^2}{24 \, A^3} \,.
\end{equation}
Consider first the case $I = [-A,A]$,
assuming \eqref{eq:defA}-\eqref{eq:assEX}: since
$\bbP^*(|X| \le A) = \bbP(X \in I)$, relation \eqref{eq:analog}
concides precisely with the condition \eqref{eq:tocheckxtilde} to be checked.
Next we consider the case $I = [0,A]$,
assuming \eqref{eq:defAbis}-\eqref{eq:assEXbis}, where we recall that
$\bbP^*(\,\cdot\,) = \bbP(\,\cdot \,|\, X \ge 0)$.
By assumption $\bbE[X] \ge 0$, we have $|\bbE[X]| \le |\bbE[X \ind_{\{X \ge 0\}}]| =
\bbP(X\ge 0) |\bbE^*[X]|$.
Since we can write $\bbP^*(|X| \le A) = \bbP(X \in I) / \bbP(X \ge 0)$,
relation \eqref{eq:analog} again yields \eqref{eq:tocheckxtilde}.

To conclude, for the  case $I = [-A,A]$,
applying \eqref{eq:p23} and \eqref{eq:varYfin}
to \eqref{lambdabar} and recalling the definition of $\epsilon$
in \eqref{eq:assEX} gives \eqref{eq:mart1}.
For the case $I=[0,A]$, the bound \eqref{eq:mart2} follows similarly,
recalling the definition of $\epsilon'$ in \eqref{eq:assEXbis} and
observing that $\bbP(X \in I) \ge \frac{2}{3} \bbP(X \ge 0)$ by \eqref{eq:p23}.
Finally, to obtain \eqref{eq:lambdabound} from \eqref{eq:mart1}-\eqref{eq:mart2}, 
apply \eqref{eq:easyboundA}
and the assumptions \eqref{eq:defA}, \eqref{eq:defAbis}.
\end{proof}

Finally, we prove the following bound on iterated integrals.
\begin{lemma}\label{L:integralbd1}
Let $\chi \in [0,1)$. Then there exist $c_1, c_2>0$ such that for all $k\in\bbN$,
\begin{equation}\label{eq:intgbd1}
\idotsint\limits_{0<t_1<\cdots<t_k<1} \frac{\dd t_1\cdots \dd t_k}{t_1^{\chi}
\cdots (t_k-t_{k-1})^{\chi}} \leq c_1 e^{-c_2k\log k},
\end{equation}
and
\begin{equation}\label{eq:intgbd2}
\idotsint\limits_{0<t_1<\cdots<t_k<1} \frac{\dd t_1\cdots \dd t_k}{t_1^{\chi}\cdots 
(t_k-t_{k-1})^{\chi}(1-t_k)^{\chi}} \leq c_1 e^{-c_2k\log k}.
\end{equation}
\end{lemma}

\begin{proof}
It is enough to prove \eqref{eq:intgbd2}, since the integral therein bounds \eqref{eq:intgbd1}.
Recognizing the density of the Dirichlet distribution (with parameters $k+1$ and $1-\chi$)
allows to evaluate
$$
\idotsint\limits_{0<t_1<\cdots<t_k<1} \frac{\dd t_1\cdots \dd t_k}{t_1^{\chi}
\cdots (t_k-t_{k-1})^{\chi}(1-t_k)^{\chi}} =\frac{\Gamma(1-\chi)^{k+1}}
{\Gamma\big((k+1)(1-\chi)\big)} \,,
$$
and \eqref{eq:intgbd2} follows by standard properties of the gamma function $\Gamma(\cdot)$.
\end{proof}

\section*{Acknowledgements}

We are grateful to Tom Spencer for fruitful discussions
and we thank the referee for helpful comments.
We also thank Niccol\`o Torri and Ran Wei for spotting 
some mistakes in the draft.
We acknowledge support of NUS grant R-146-000-148-112 (R.S.),
a Marie Curie International Reintegration Grant within the 7th European
Community Framework Programme IRG-246809 and EPSRC grant EP/L012154/1 (N.Z.),
and ERC Advanced Grant 267356 VARIS (F.C.).
We also thank the kind hospitality of the Institute of Mathematics at the University of Leiden,
of Academia Sinica in Taipei and of the Department of Mathematics at Hokkaido University
in Sapporo, where parts of this work were completed.


\end{document}